\documentclass[11pt]{article}
\usepackage{amssymb}
\usepackage{amsmath}
\usepackage{amsthm}
\usepackage[textwidth=16cm, textheight=24cm]{geometry}

\usepackage[numbers]{natbib}
\usepackage[english]{babel}
\usepackage[utf8]{inputenc}
\usepackage[T1]{fontenc}
\usepackage[pdfborder={0 0 0}]{hyperref}
\usepackage{tikz-cd}
\usepackage{enumitem}

\usepackage{cancel}
\numberwithin{equation}{section}
\newtheorem{thm}[equation]{Theorem}
\newtheorem{cor}[equation]{Corollary}
\newtheorem{lem}[equation]{Lemma}

\newtheorem{prop}[equation]{Proposition}
\theoremstyle{definition}
\newtheorem{defn}[equation]{Definition}
\newtheorem{notation}[equation]{Notation}
\newtheorem{remark}[equation]{Remark}
\newtheorem{example}[equation]{Example}

\def\rk{\operatorname {rank}}

\DeclareMathOperator{\Add}{Add}
\def\dim{\operatorname{dim}}
\DeclareMathOperator{\codim}{codim}

\def\ker{\operatorname{ker}}

\DeclareMathOperator{\cchar}{char}

\def\im{\operatorname{im}}

\def\id{\operatorname{id}}

\def\Spec{\operatorname{Spec}}
\def\Sym{\operatorname{Sym}}

\def\tensor{\otimes}
\def\onto{\twoheadrightarrow}
\def\into{\hookrightarrow}

\DeclareMathOperator{\Hom}{Hom}
\DeclareMathOperator{\Hilb}{Hilb}
\DeclareMathOperator{\GL}{GL}

\newcommand{\reduced}[1]{{#1}_{\operatorname{red}}}

\newcommand{\set}[1]{\left\{#1\right\}}
\newcommand{\fromto}[2]{#1, \dotsc, #2}
\newcommand{\setfromto}[2]{\set{\fromto{#1}{#2}}}
\newcommand{\sspan}[1]{\left\langle #1 \right\rangle}

\newcommand{\ccE}{{\mathcal{E}}}
\newcommand{\ccF}{{\mathcal{F}}}

\newcommand{\ccH}{{\mathcal{H}}}
\newcommand{\ccI}{{\mathcal{I}}}
\newcommand{\ccJ}{{\mathcal{J}}}

\newcommand{\ccO}{{\mathcal{O}}}

\newcommand{\ccS}{{\mathcal{S}}}

\newcommand{\ccU}{{\mathcal{U}}}

\newcommand{\gotK}{\mathfrak{K}}
\newcommand{\gotm}{\mathfrak{m}}
\newcommand{\gotn}{\mathfrak{n}}
\newcommand{\gotp}{\mathfrak{p}}

\def\AA{{\mathbb A}}
\def\CC{{\mathbb C}}
\def\FF{{\mathbb F}}
\def\GG{{\mathbb G}}

\def\KK{{\mathbb K}}

\def\PP{{\mathbb P}}

\def\RR{{\mathbb R}}
\def\ZZ{{\mathbb Z}}

\def\kk{{\Bbbk}}
\def\kkbar{\overline{\Bbbk}}

\newcommand{\DPV}[1]{\operatorname{DP}^{#1} V}
\newcommand{\hook}{\lrcorner}%

\newcommand{\cactus}[2]{\gotK_{#1}\left( #2 \right)}

\renewcommand{\theenumi}{(\roman{enumi})}

\newcounter{betweenenumi}

\usepackage{stmaryrd}
\usepackage{todonotes}
\usepackage{xcolor}

\newenvironment{red}{\color{red}}{}
\newcommand{\bred}{\begin{red}}
\newcommand{\ered}{\end{red}}

\newenvironment{blue}{\color{blue}}{}
\newcommand{\bblue}{\begin{blue}}
\newcommand{\eblue}{\end{blue}}

\newenvironment{green}{\color{green}}{}
\newcommand{\bgreen}{\begin{green}}
\newcommand{\egreen}{\end{green}}

\definecolor{BrazA}{RGB}{140,64,3}
\definecolor{ZielonyA}{RGB}{78,104,18} 
\newcommand{\colordynia}[1]{{\color{BrazA}#1}}
\newcommand{\colorzielony}[1]{{\color{ZielonyA}#1}}

\title{Finite schemes and secant varieties over arbitrary characteristic}
\author{Jaros\l{}aw Buczy\'nski \and Joachim Jelisiejew}

\date{March 8, 2017}
\begin{document}

\maketitle
\begin{abstract}
We present scheme theoretic methods that apply to the study of secant
varieties.
This mainly concerns finite schemes and their smoothability.
The theory generalises to the base fields of any characteristic, and even to non-algebraically closed fields.
In particular, the smoothability of finite schemes does not depend on the embedding into a smooth variety or on base field extensions.
Independent of the base field, secant varieties to high degree Veronese reembeddings behave well with respect to the intersection 
  and they are defined by minors of catalecticants whenever a suitable smoothability condition for Gorenstein subschemes holds.
The content of the article is largely expository, although many results are presented in a stronger form than in the literature.
\end{abstract}

\medskip
{\footnotesize
\noindent\textbf{addresses:} \\
J.~Buczy\'nski, \nolinkurl{jabu@mimuw.edu.pl},
   Institute of Mathematics of the Polish Academy of Sciences, ul.~\'Sniadeckich 8, 00-656 Warsaw, Poland,
 and Faculty of Mathematics, Computer Science and Mechanics, University of Warsaw, ul.~Banacha 2, 02-097 Warszawa, Poland\\
J.~Jelisiejew, \nolinkurl{jjelisiejew@mimuw.edu.pl},
   Faculty of Mathematics, Computer Science and Mechanics, University of Warsaw, ul.~Banacha 2, 02-097 Warszawa, Poland\\

\noindent\textbf{keywords:}
finite scheme, smoothable, secant varieties, finite Gorenstein scheme, cactus variety, Veronese reembedding, Hilbert scheme\\


\noindent\textbf{financial support:}
Buczy\'nski is supported by a grant Iuventus Plus of the Polish Ministry of Science, project 0301/IP3/2015/73, and by a scholarship of the Polish Ministry of Science.
Jelisiejew is supported by Polish National Science Center, project
2014/13/N/ST1/02640.
This article was partially written during the Polish Algebraic Geometry mini-Semester (miniPAGES),
   which was supported by the grant 346300 for IMPAN from the Simons
   Foundation and the matching 2015-2019 Polish MNiSW fund.}

\tableofcontents

\section{Introduction}

The purpose of this article is two-fold.
In the first place, we study the structure of finite algebraic schemes over an arbitrary field and their deformations.
We provide purely algebraic proofs of several facts, some of which are already known over the base field of complex numbers. 
Furthermore we apply these algebraic statements to study secant varieties over 
   fields of any characteristic.
We also take this opportunity to review some material on such secant varieties.
Throughout we are not going to make any assumption on the characteristic of
the base field $\kk$ (except when recalling results from the literature, or giving examples).

\subsection{Smoothability of finite schemes}

In algebraic geometry, a \emph{finite scheme} $R$ over a field $\kk$ is an
abstract geometric interpretation (called \emph{spectrum})
  of a  $\kk$-algebra $A$ which has finite dimension as a vector space over $\kk$.
We can also talk about an \emph{embedded} interpretation:
If $X$ is an algebraic variety, 
  an embedding $R \subset X$ is a finite collection of informations about the functions on $X$, i.e.~their restrictions to $R$, 
  which are consistent with the underlying algebraic structure.
So if $R$ simply consists of a couple of reduced $\kk$-rational points on $X$, 
  we get information about the values of functions at those points.
If $R$ is a double point in $X$, we obtain the value of functions at this
point and their first order behaviour.
Similarly, more complicated schemes encode more complicated higher order
properties of the functions.

We are principally interested in the \emph{smoothability} of a finite scheme $R$.
Roughly, if $\kk$ is algebraically closed, $R$ is smoothable if and only if it is a limit of $r$ distinct $\kk$-rational points.
Such limit might be taken abstractly or embedded into a variety $X$.
Correspondingly, there are notions of \emph{abstract} and \emph{embedded}
smoothability, which we discuss in Section~\ref{sec:smoothability}. The
following Theorem~\ref{thm_equivalence_of_abstract_and_embedded_smoothings}
proves that those notions are equivalent when $X$ is smooth.
\begin{thm}\label{thm_equivalence_of_abstract_and_embedded_smoothings}
  Suppose $X$ is a smooth variety over an arbitrary field $\kk$ and
  $R\subset X$ is a finite $\kk$-subscheme.
  Then the following conditions are equivalent:
  \begin{enumerate}
      \item\label{it:Rabssm} $R$ is abstractly smoothable,
      \item\label{it:Rembsm} $R$ is embedded smoothable in $X$,
      \item\label{it:Rconabssm} every connected component of $R$ is abstractly smoothable,
      \item\label{it:Rconembsm} every connected component of $R$ is embedded smoothable in $X$.
  \end{enumerate}
\end{thm}
We prove the theorem in Section~\ref{sect_smoothings_abstract_embedded}.
Some of the equivalences in the theorem are probably known to experts, but it is hard to find an explicit reference.
In particular, the recent \cite{casnati_jelisiejew_notari_Hilbert_schemes_via_ray_families} 
  refers to a draft version of this article.
Special cases are also in \cite[Lemma~4.1]{cartwright_erman_velasco_viray_Hilb8}, 
  \cite[Lemma~2.2]{casnati_notari_irreducibility_Gorenstein_degree_9},
  \cite[Proposition~2.1]{nisiabu_jabu_cactus}.
In particular, if $\kk$ is the field of complex numbers $\CC$, some of the known proofs rely on the analytic topology of $X$.

In the literature on smoothability, e.g.~\cite{cartwright_erman_velasco_viray_Hilb8,
casnati_notari_irreducibility_Gorenstein_degree_9,
casnati_notari_irreducibility_Gorenstein_degree_10,
casnati_jelisiejew_notari_Hilbert_schemes_via_ray_families}, usually $\kk$ is
assumed to be algebraically closed. The following result implies that we may
reduce to this case by a base extension, see~Section~\ref{sec_product_and_base_change}.
\begin{prop}[Corollary~\ref{ref:basechangesmoothings:cor}]\label{prop_base_change_equivalence}
    Let $R$ be a finite scheme over $\kk$ and $R' = R \times_{\Spec \kk} \Spec \kkbar$ be
    a finite scheme over $\kkbar$.
    Then $R$ is smoothable (over $\kk$) if and only if $R'$ is smoothable (over $\kkbar$).
\end{prop}
In particular, we can remove the assumption $\kk = \kkbar$ from results
of~\cite{cartwright_erman_velasco_viray_Hilb8,
casnati_jelisiejew_notari_Hilbert_schemes_via_ray_families} to obtain the
following proposition.
\begin{prop}\label{prop_smoothability}
    Let $\kk$ be a field of characteristic not equal to $2$ or $3$.
    All finite $\kk$-schemes of degree at most $7$ are smoothable and all finite
    Gorenstein $\kk$-schemes of degree at most $13$ are smoothable.
\end{prop}
See Section~\ref{sect_embedded_smoothings_and_Hilb} for the proof.
More generally, it follows that the statements \emph{all finite schemes (or all finite Gorenstein
schemes) of degree $d$ over an algebraically closed field are smoothable} only depend on the characteristic of
the base field, see Remark~\ref{remark_smoothability_depends_only_on_char}.

Our next result is tailored to fit the needs of the secant varieties, as explained in Section~\ref{sec_intro_secant_varieties},
  but it is also of independent interest.
It treats the situation, when we have two smooth varieties $X\subset Y$ and a finite subscheme $R \subset Y$ 
  supported in $X$.
We are interested in the intersection $R \cap X$, which is a finite subscheme of the smaller variety $X$.
Ideally, one would like an implication ``if $R$ is smoothable (in $Y$), then $R \cap X$ is
smoothable (in $X$)'' to be true.
But this is too much to ask for as we illustrate in Corollary~\ref{cor_nonsmoothable_intersection_of_smoothable_and_smooth}.
Instead, we prove:
\begin{thm}[{Proposition~\ref{ref:smooth_subvar_and_pushing:prop}}]\label{thm_R_cap_X_is_contained_in_smoothable_Q}
   Suppose $X \subset Y$ are two smooth varieties over an algebraically closed field $\kk$.
   Let $R \subset Y$ be a finite smoothable subscheme of degree $r$, whose all points of support are on $X$.
   Then there exists a finite smoothable scheme $Q \subset X$ of degree $r$ containing $R\cap X$.
\end{thm}
In other words, although $R \cap X$ needs not to be smoothable, it is contained in a not-too-large smoothable subscheme of $X$.
We prove the theorem in Section~\ref{sect_smoothings_abstract_embedded}.

\subsection{Secant varieties}\label{sec_intro_secant_varieties}

Suppose $\kk$ is a field, $X \subset \PP^N_{\kk}$ is a projective variety with a fixed embedding into a projective space
  and fix a positive integer $r$.
First suppose the set $X(\kk)$ of $\kk$-rational points of $X$ is Zariski dense in $X$ 
  (for instance, this is always true if $\kk$ is algebraically closed). 
For a definition of $\kk$-rational points of a variety or of a scheme, see Section~\ref{sect_varieties}.
Then the $r$-th secant variety $\sigma_r(X)\subset \PP^N_{\kk}$ is defined as the
following Zariski closure:
\[
   \sigma_r(X) = \overline{\bigcup \set {\langle \fromto{x_1}{x_r}\rangle : x_i \in X(\kk)}},
\]
  where $\langle\fromto{x_1}{x_r} \rangle$ denotes the projective linear span of the $\kk$-rational points $\fromto{x_1}{x_r}$.
Note that the secant variety depends also on the embedding of $X$ in $\PP_{\kk}^N$, which is implicit in our notation.
The more general definition of secant variety (when $\kk$-rational points are not necessarily dense) in terms of Hilbert scheme and the notion of relative linear span of families of schemes, basic properties,
  and also interactions of secant varieties with finite schemes are surveyed in Section~\ref{sect_secants}. 
The relative linear span is defined and studied in Sections~\ref{sec_relative_linear_span} and \ref{sect_universal_ideal_and_graded_pieces}.
A consequence of this definition is that the secant variety $\sigma_r(X)$ is the minimal closed subvariety of the projective space, that contains all the secant planes  
$\langle \fromto{x_1}{x_r}\rangle_{\KK} \subset \PP^N_{\KK}$ for any field extension $\kk\subset \KK$, and any choice of $\KK$-rational points $x_i$ in $(X\times_{\Spec \kk} \Spec \KK)$.

In this article we generalise the main results of \cite{jabu_ginensky_landsberg_Eisenbuds_conjecture} and \cite{nisiabu_jabu_cactus}
  to any base field.
These results involve secant varieties to Veronese reembeddings of projective varieties.
Let \mbox{$\nu_d\colon \PP_{\kk}^{n} \to \PP_{\kk}^{\binom{n+d}{d} -1}$}
  be the algebraic morphism given by all monomials of degree $d$ in the coordinates of $\PP_{\kk}^n$.
The map $\nu_d$ is called the \emph{Veronese} embedding.
Suppose $X \subset \PP_{\kk}^n$ is a projective variety.
We want to describe the secant varieties to Veronese reembeddings of $X$, or $\sigma_r(\nu_d (X))$.
We always assume that $d$ is sufficiently high.
The first result compares $\sigma_r(\nu_d (X))$ and $\sigma_r(\nu_d (\PP_{\kk}^n))$.

\begin{thm}\label{thm_BGL}
    Suppose $\kk$ is any field and fix a smooth projective variety $X \subset \PP_{\kk}^n$ and a positive integer $r$.
    Then there exists an integer $d_0 >0$, such that for all $d\ge d_0$ the following sets are equal:
    \[
       \sigma_r(\nu_d (X))  = \reduced{\left(\sigma_r(\nu_d (\PP_{\kk}^n)) \cap \langle \nu_d(X)\rangle\right)}.
    \]
    Here $\langle \nu_d(X)\rangle$ denotes the linear span of $X$ embedded in $\PP_{\kk}^{\binom{n+d}{d} -1}$ via Veronese map,
       and $\reduced{(\cdot)}$ denotes the reduced subscheme.
\end{thm}

In other words, with the assumptions of the theorem, 
    the (set-theoretic) defining equations of $\sigma_r(\nu_d (X))$ can be read from the defining equations
       of $\sigma_r(\nu_d (\PP_{\kk}^n))$ 
       and those of $X \subset \PP_{\kk}^n$.
The statement of the theorem is known when $\kk=\CC$ by \cite[Thm~1.1]{jabu_ginensky_landsberg_Eisenbuds_conjecture}.
Theorems~\ref{thm_equivalence_of_abstract_and_embedded_smoothings} and \ref{thm_R_cap_X_is_contained_in_smoothable_Q} 
    are used as substitutes for transcendental methods used in that paper, 
    in order to extend the result to arbitrary characteristic.

The next result concerns the problem if we can describe explicitly the defining equations of $\sigma_r(\nu_d (\PP_{\kk}^n))$ 
    (again with the assumption that $d\gg 0 $).
There is plenty of equations vanishing on that secant variety, that arise from line bundles and vector bundles on $\PP_{\kk}^n$.
The question is, when these are enough to define the secant variety.
The answer is, that this is governed by a condition on smoothability of finite subschemes of $\PP_{\kk}^n$.

The following condition depends on the base field $\kk$ (more precisely, on its algebraic closure), 
   an algebraic variety $X$, and an integer $r$.
It does not depend on any embedding of $X$, in particular, on the integer $d$, as in Theorems~\ref{thm_BGL} or \ref{thm_BB}. 
In fact, the dependence on $X$ is restricted to the type of its singularities, see~Proposition~\ref{prop_smoothability_depends_only_on_sing_type}.
So fix a positive integer $r$ and an algebraic variety $X$ over the base field $\kk$
  and let $\kkbar$ be the algebraic closure of $\kk$.
We say that \ref{item_condition_on_smoothability_of_all_Gorenstein} holds if
   \renewcommand{\theenumi}{$(\star)$}
   \begin{enumerate}
     \item \label{item_condition_on_smoothability_of_all_Gorenstein}
         every finite Gorenstein subscheme over $\kkbar$ of
         $X_{\kkbar}:=X\times_{\Spec \kk}\Spec \kkbar$ 
         of degree at most $r$ is smoothable in $X_{\kkbar}$.
   \end{enumerate}
\renewcommand{\theenumi}{(\roman{enumi})}
In Section~\ref{sec_on_condition_star} we review for which triples $(\kk,r, X)$ 
   the condition \ref{item_condition_on_smoothability_of_all_Gorenstein} holds. 
For example, if $X$ is smooth and $\dim X\le 3$, then \ref{item_condition_on_smoothability_of_all_Gorenstein} holds.
Our point here is that the natural equations of secant varieties associated to line and vector bundles are enough 
   to describe the secant variety to high degree Veronese reembedding 
   if and only if the condition \ref{item_condition_on_smoothability_of_all_Gorenstein} holds.
 
Assume $X\subset \PP_{\kk}^n$ is a projective variety and suppose $d\gg 0$ is a sufficiently lagre integer.
Then the $d$-th Veronese map embeds $X$ using the \emph{complete} linear system $H^0(\ccO_X(d))$
  (see Lemma~\ref{lem_regular_then_complete_linear_system} for an explicit bound on such $d$):
  \[
   \nu_d(X) \subset \PP (H^0(\ccO_X(d))^*)
  \]
Suppose $\ccE$ is a vector bundle on $X$ of rank $e$ and $\ccE^*$ is the dual vector bundle.
The evaluation map $\ccE \otimes \ccE^*(d) \to \ccO_X(d)$ determines the map of sections 
$H^0(\ccE) \otimes H^0(\ccE^*(d)) \to H^0(\ccO_X(d))$, 
or equivalently, a bilinear map:
\[
   H^0(\ccO_X(d))^* \times H^0(\ccE) \to H^0(\ccE^*(d))^*.   
\]
That is, for every $v\in H^0(\ccO_X(d))^*$ we have a linear map $\phi_{\ccE}(v)\colon H^0(\ccE) \to H^0(\ccE^*(d))^*$,
  which depends linearly on $v$.
It is a straightforward observation that $\kk$-rational points of $\nu_d(X)$ are contained in the locus of those $[v] \in \PP (H^0(\ccO_X(d))^*)$, 
  such that $\rk(\phi_{\ccE}(v)) \le e$. This statement and similar can be translated into determinantal equations of $\nu_d(X)$, and we will denote it by $\nu_d(X) \subset (\rk(\phi_{\ccE}) \le e)$,
  see Section~\ref{sec_secant_cactus_catalecticants} for a formal definition of such locus.
As a consequence, for any integer $r \ge 1$ the secant variety $\sigma_r(\nu_d(X))$ is contained 
  in the locus $(\rk(\phi_{\ccE}(v)) \le r \cdot e)$ 
  \cite[Prop.~5.1.1]{landsberg_ottaviani_VB_method_equ_for_secants}.
In fact, even larger set, called \emph{a cactus variety} is contained in the same locus~\cite[Thm~4]{galazka_vb_cactus} or Theorem~\ref{thm_Galazka}.

A special and important case is when the vector bundle $\ccE$ is a line bundle
$\ccO_X(i)$, hence $e = 1$ and 
$\sigma_r(\nu_d(X)) \subset (\rk(\phi_{\ccO_X(i)}) \leq r)$.
Then the linear map $\phi_{\ccE}(v)$ and the minors described above has a very explicit algebraic description in terms of apolarity.
The following theorem compares subvarieties of $\PP(H^0(\ccO_X(d))^*)$: 
   the secant variety $\sigma_r(\nu_d(X))$ and $\reduced{(\rk(\phi_{\ccO_X(i)}) \leq r)}$.
Under suitable regularity assumptions, if the condition \ref{item_condition_on_smoothability_of_all_Gorenstein} holds, 
  then the equations coming from $ \ccO_{X}(i)$ are enough, but if  \ref{item_condition_on_smoothability_of_all_Gorenstein} fails to hold, then the equations for all vector bundles $\ccE$ are not enough.

\begin{thm}\label{thm_BB}
Suppose $\kk$ is any field and $X \subset \PP_{\kk}^n$ is a projective variety and let $r \ge 1$ be an integer.
For a vector bundle $\ccE$ of rank $e_\ccE$ define the linear map $\phi_{\ccE}(v)$ as above.
Then
   \begin{enumerate}
      \item \label{item_if_star_holds_then_sigma_eq_cactus}
         If \ref{item_condition_on_smoothability_of_all_Gorenstein} holds, $d\gg 0$, 
            and $i$ is an integer such that $r \le i \le d-r$, and $\ccE = \ccO_{X}(i)$, then
         \[
            \sigma_r(\nu_d(X)) = \reduced{(\rk(\phi_{\ccO_X(i)}) \le r)}.
         \]
      \item \label{item_if_sigma_eq_cactus_then_star_holds}
         If \ref{item_condition_on_smoothability_of_all_Gorenstein} does not hold 
           and $d\ge 2r-1$, then  
         \[
            \sigma_r(\nu_d(X))  \subsetneq \bigcap_{\ccE}
            \reduced{(\rk(\phi_{\ccE}) \le r\cdot e_{\ccE})}
         \]
         where the intersection is over all vector bundles $\ccE$ on $X$.
   \end{enumerate}
\end{thm}
We prove a slightly stronger version of the theorem in Section~\ref{sec_secant_cactus_catalecticants}. 
The proof contains an explicit bound on sufficiently large $d$.

\subsection{Overview}

In Section~\ref{sec_preliminaries_and_finite_schemes} we present the necessary
notions from scheme theory and theory of finite schemes; this section may
freely be skipped by expert readers. In Section~\ref{sec:smoothability} we present our results
on smoothability of finite schemes, including Theorems~\ref{thm_equivalence_of_abstract_and_embedded_smoothings}, \ref{thm_R_cap_X_is_contained_in_smoothable_Q} and 
Proposition~\ref{prop_smoothability}.
In Section~\ref{sec_projective_schemes_and_embedded_geometry} we give some
overview of linear and projective geometry, including apolarity theory and Castelnuovo-Mumford regularity; again, the experts can freely omit it. 
In Section~\ref{sec_Hilbert} we discuss Hilbert schemes and their loci and we prove
Proposition~\ref{prop_base_change_equivalence}. In
Subsection~\ref{sec_relative_linear_span} we discuss the notion of linear
spans of points of the Hilbert scheme; the emphasis is on the situation when
$\kk$ is not algebraically closed. We discuss relative analogue of linear
spans in Subsection~\ref{sect_universal_ideal_and_graded_pieces}. In
Section~\ref{sect_secants} we apply the theory of smoothability and relative
linear spans to discuss properties of secant varieties and in
Section~\ref{sect_secants_of_Veronese} we consider the special case of high
enough Veronese reembedding, proving Theorems~\ref{thm_BGL} and \ref{thm_BB}.

\subsection*{Acknowledgements}
The authors thank Gavin Brown, Weronika Buczy{\'n}ska, Dustin Cartwright, Mateusz Michałek, and \L{}ukasz Sienkiewicz for hints and many helpful discussions.
The computer algebra program Magma \cite{magma} was helpful in calculation of explicit examples.
We thank IMPAN and the participants of the mini-semester miniPAGES 04-06.2016
for an excellent and inspiring atmosphere during the event.
The article is a part of the activities of AGATES research group.
   
\section{Preliminaries and finite schemes}\label{sec_preliminaries_and_finite_schemes}

Throughout the article we work over a field $\kk$ of any characteristic.
In general, we do \emph{not} require that $\kk$ is algebraically closed, 
  but a few statements require this additional assumption.
A good introduction to the theory schemes is, for instance, \cite{eisenbud_harris}.
Much more details are provided in \cite{gortz_wedhorn_algebraic_geometry_I}, \cite{vakil_FoAG} or \cite{stacks_project}.

Expert readers may easily skip this section after they familiarise themselves with Notation~\ref{not_scheme} 
  which is valid throughout the article.
For non-experts, we explain in more details the terminology used in Notation~\ref{not_scheme} and other terminology related to algebraic schemes and also we review related material.
\begin{notation}\label{not_scheme}
   All schemes and morphisms that we consider here are
   separated, locally Noetherian schemes over $\kk$.
   By a \emph{variety} we mean a reduced (not necessarily irreducible) scheme of finite type.
   By default, the schemes are equipped with Zariski topology, except when stated otherwise.
\end{notation}

\subsection{General properties of schemes}\label{sect_general_properties_of_schemes}

An \emph{affine scheme} $\Spec A$ (over $\kk$) is the spectrum of a $\kk$-algebra $A$,
  i.e.~the set of prime ideals of $A$ equipped with the Zariski topology 
  (a subset $\ccS$ of prime ideals is closed if and only if there exists an ideal $I \subset A$, 
   such that $\ccS$ is the set of prime ideals containing $I$)
   and with a structure sheaf $\ccO_{\Spec A}$, whose global sections are $A$.
For the description of the sheaf on other open subsets, 
  see for instance \cite[Sect.~I.1.4]{eisenbud_harris}, \cite[Sect.~(2.10)]{gortz_wedhorn_algebraic_geometry_I}, 
  \cite[Sect.~4.1]{vakil_FoAG}, or \cite[Tag 01HR]{stacks_project}.
The sections of the structure sheaf are interpreted as \emph{functions} on $\Spec A$.
To avoid confusion, we say a point $x \in X$ corresponds to a prime ideal $\gotp \subset A$, and vice versa,
  rather than just write $\gotp \in X$.
The field $\kk$ is called a base field of $X$.

\begin{example}
   Suppose $A  =\kk^r = \underbrace{\kk \times \kk \times \dotsb \times \kk}_{r \text{ times}}$.
   Then $\Spec A$ consist of $r$ closed points. 
   For an open subset $U \subset \Spec A$ consisting of $k$ points we have $\ccO_{\Spec A}(U) = \kk^k$.
\end{example}
  
\begin{example}
   Suppose $A  =\kk[\alpha]/(\alpha^r)$.
   Then $\Spec A$ consists of a single point (which must be closed), and $\ccO_{\Spec A}(\Spec A) = \kk[\alpha]/(\alpha^r)$. 
\end{example}

\begin{example}
   Suppose $A  =\kk[[\alpha]]$ is a power series ring in one variable.
   Then $\Spec A$ consists of two points: the closed point corresponding to $(\alpha)\subset A$ and the generic point $\eta$ corresponding to $\set{0} \subset A$.
   As usually, $\ccO_{\Spec A}(\Spec A) = \kk[[\alpha]]$, while $\ccO_{\Spec A}(\set{\eta}) = \kk((\alpha))$, the field of fractions of $A$. 
\end{example}
  
\begin{example}
   Suppose $A  =\kk[[\alpha, \beta]]$ is a power series ring in two variables.
   Then $\Spec A$ consists of infinitely many points, but  there is only one closed point corresponding to $(\alpha,\beta)$.
\end{example}

\begin{example}
   Suppose $\kk =\RR$ and $A = \RR[\alpha,\beta]/(\alpha^2 + \beta^2+1)$.
   Then the closed points of $\Spec A$ are 1-to-1 correspondence with pairs $(z, \bar{z})$ of complex conjugate points of the quadric $\alpha^2 + \beta^2 = -1$ in $\CC^2$.
   In addition, there is a generic point corresponding to the zero ideal in $A$.
\end{example}

\begin{example}
   The scheme $\Spec \kk[\fromto{\alpha_1}{\alpha_n}]$ is called the affine space of dimension $n$ and it is denoted $\AA^n_{\kk}$.
\end{example}

More generally, a \emph{scheme} is glued from its open subsets, each of which is an affine scheme. 
Details of this construction can be found in \cite[Sect.~I.2]{eisenbud_harris}, 
  \cite[Sect.~(3.1)]{gortz_wedhorn_algebraic_geometry_I}, 
  \cite[Sect.~4.3]{vakil_FoAG}, or \cite[Tag 01II]{stacks_project}. 
Note that to properly understand these references one should first befriend the notion of locally ringed space,
  which is explained in earlier chapters of these books.
As in the affine case, a scheme $X$ consists of the topological space (also denoted $X$)
  and a structure sheaf $\ccO_X$ which is responsible for its algebraic properties, and encodes functions from $X$ 
  or from its open subsets. Open affine subset form the basis of the topology of $X$.
Many properties of schemes are local in nature, 
  which means that it is enough to test or define them for affine schemes $\Spec A$ only, 
  in terms of the underlying algebra $A$.
In fact, it is often enough to test or define them for the local rings $\ccO_{X,x}$, where $x \in X$ is any point:
  the ring $\ccO_{X,x}$ is defined to be the localisation of $\ccO_X(U)$
  at its prime ideal $\gotp$ corresponding to $x$ for any open affine subset
  $U \subset X$, see~\cite[Sect.~I.2.2]{eisenbud_harris},
  \cite[Sect.~(2.6)]{gortz_wedhorn_algebraic_geometry_I}, or \cite[2.1.1,~4.3.5]{vakil_FoAG}.
It is one of the intermediate results of this article, that smoothability of
finite schemes (Definition~\ref{ref:abstractsmoothable:def})
   is a local property, see Corollary~\ref{ref:smoothingcomponentsresult:cor}.
   
Open subsets of a scheme $X$ have a natural structure of a scheme,
see~\cite[Sect.~I.2.1]{eisenbud_harris}, \cite[Sect.~(3.2)]{gortz_wedhorn_algebraic_geometry_I},
\cite[Exercise~4.3.C]{vakil_FoAG}.
Also closed subsets are themselves schemes, although here the structure of a scheme is not unique --- 
it depends on the choice of an ideal sheaf $\ccI\subset \ccO_X$, 
whose radical is equal to a fixed radical ideal sheaf,
see~\cite[Sect.~I.2.1]{eisenbud_harris},
\cite[Sect.~(3.5)]{gortz_wedhorn_algebraic_geometry_I}, or
\cite[Sect.~8.1]{vakil_FoAG}.
A union of a finite number of closed subschemes or an intersection of a (perhaps infinite) family of subschemes
has a canonical structure of a closed
subscheme, see~\cite[Section~I.2.1, p.~24]{eisenbud_harris}, \cite[Sect.~(4.11)]{gortz_wedhorn_algebraic_geometry_I},
or~\cite[Exercise~8.1.J]{vakil_FoAG}.
Taking the union of closed subschemes is independent of the order, similarly for intersections.
However, these two operations \emph{do not} satisfy the distributive law (similarly to algebraic sum and intersection of vector spaces, but unlike topological spaces).

\begin{example}[Intersection and union are nondistributive]
  Let $\AA^2_{\kk} = \Spec\kk[\alpha, \beta]$ and
   \[R_1 = \Spec \kk[\alpha,\beta]/(\alpha),\quad R_2 = \Spec \kk[\alpha,\beta]/(\beta),\quad R_3= \Spec
   \kk[\alpha,\beta]/(\alpha+\beta).\]
  That is, $R_1, R_2, R_3$ are three lines through the origin in $\AA^2_{\kk}$.
  Then:
  \begin{itemize}
   \item $(R_1 \cup R_2)\cap R_3 = \Spec \kk[\alpha,\beta]/(\alpha \beta, \alpha+\beta)$, and
   \item $(R_1 \cap R_3)\cup (R_2\cap R_3) = \Spec \kk[\alpha,\beta]/(\alpha, \beta)$.
  \end{itemize}
\end{example}

A scheme $X$ is \emph{reducible}, if its topological space can be presented as a union $X= Y \cup Z$,
   with both $Y, Z$ Zariski closed subsets of $X$ 
   and the presentation is non-trivial, i.e.~$Y \ne X$, and $Z\ne X$, see
   \cite[p.~25]{eisenbud_harris},
   \cite[Def.~1.18]{gortz_wedhorn_algebraic_geometry_I},
   or~\cite[3.6.4,~3.6.12]{vakil_FoAG}.
We say $X$ is \emph{irreducible} if it is not reducible.
Irreducible components of a scheme $X$ are defined in a standard way:
if $X= X_1\cup X_2 \cup \dotsb \cup X_k$ with all $X_i$ irreducible closed subschemes and there are no redundancies 
   (i.e. for any $i \ne j$ there is no inclusions of topological spaces $X_i \subset X_j$), then the schemes $X_i$ are the irreducible components of $X$.
Note that the underlying topological spaces of $X_i$ are unique (up to reordering), but the scheme structure is not necessarily unique.
If $X$ is irreducible, then any open subset is also irreducible.
If $X = \Spec A$ is affine, then it is irreducible if and only if the radical
ideal of $\set{0} \subset A$ (that is, the set of nilpotent elements
$\sqrt{\set{0}} = \set{f \in A \mid f^k = 0 \text{ for some }k\ge 1}$)
  is a prime ideal in $A$. 
See~\cite[Exercise~I.31]{eisenbud_harris}, \cite[Cor.~2.7]{gortz_wedhorn_algebraic_geometry_I}, or \cite[3.7.F]{vakil_FoAG}.

A \emph{generic point} of a scheme $X$ is a point $\eta \in X$, whose closure is $X$. A generic point exists if and only if $X$ is irreducible.
In such case, $\eta$ is contained in any open affine subset of $X$, and it corresponds to the radical ideal of $\set{0}$, which is a prime ideal.

An affine scheme $X = \Spec A$ is \emph{reduced} if $A$ has no nonzero nilpotent elements, or $\sqrt{\set{0}} = \set{0}$.
A general scheme is reduced if all its affine open subsets are reduced.
Every scheme $X$ has a unique maximal reduced subscheme $\reduced{X}$,
  which has the same topological structure as $X$, but the structure sheaf $\ccO_{\reduced{X}}(U)$ is the quotient $\ccO_X(U)/(nilpotents)$.

A scheme is \emph{locally Noetherian} if it admits an open affine covering by schemes $\Spec A_i$,
  where each $A_i$ is a Noetherian ring.
A scheme is \emph{Noetherian} if it admits a finite covering as above.
The (locally) Noetherian condition is to guarantee several finiteness criteria, 
  for instance the dimension of such scheme is (locally) well defined, well behaved, and finite.
As already mentioned in Notation~\ref{not_scheme}, we are always assuming our schemes are locally Noetherian.
To some extent, this property has analogous purpose to the ``second countable'' property of  topological spaces, 
  which is always assumed for topological or differentiable manifolds.

A stronger finiteness condition is ``finite type''.
A scheme $X$ is of \emph{finite type} (over $\kk$), if $X$ has a finite open affine cover of the form $\Spec A_i$,
  such that each $A_i$ is a finitely generated $\kk$-algebra \cite[Section (3.12)]{gortz_wedhorn_algebraic_geometry_I}.

The \emph{dimension} $\dim(X,x)$ of a scheme $X$ at a point $x \in X$ is the
maximal length of a chain of prime ideals in $\ccO_{X,x}$.
The \emph{dimension} $\dim X$ of $X$ is the maximum of its dimensions at all points $x \in X$.

Suppose $X$ is a scheme and $x\in X$ is a point. The \emph{residue field} of
$x$ is $\kappa(x):= \ccO_{X,x}/\gotm$, where $\gotm \subset \ccO_{X,x}$ is the unique maximal ideal.
We always have $\kk \subset \kappa(x)$ (recall, we work with schemes over $\kk$).
A point $x\in X$ is a closed point if and only if it corresponds to a maximal ideal in $\ccO_X(U)$ for some open (or for any) affine subset $U$.
If in addition $X$ is a of finite type, then $x$ is closed if and only if the
field extension $\kk\subset \kappa(x)$ is finite.

Let $\KK$ be a field containing $\kk$.
A \emph{$\KK$-point} is a morphism $\Spec \KK \to X$. It corresponds to a
point $x\in X$, the image of the unique point of $\Spec \KK$, together with an
extension $\kappa(x) \subset \KK$. In particular if $X$ is of finite type, then
$\kkbar$-points are exactly the closed points $x\in X$ together with
injections $\kappa(x) \subset \kkbar$. A \emph{$\KK$-rational point} is a
$\KK$-point such that $\kk(x) = \KK$.

Usual algebraic operations give rise to scheme morphisms.
Recall that for an algebra $A$ and its ideal $\gotm$ there is an
($\gotm$-adic) \emph{completion
map} $A \to \hat{A}_{\gotm}:= \varprojlim  A/\gotm^i$.
The algebra $A$ is ($\gotm$-adically) \emph{complete} with respect to its ideal $\gotm$
if the natural map $A \to \hat A_{\gotm}$ is an isomorphism~\cite[Tag~0317]{stacks_project}.
If $X$ is an affine scheme with a subscheme $Y \subset X$ given by an ideal
$\gotm \subset A$, then the
\emph{completion} of $X$ at $Y$ is defined as $\Spec \hat{A}_{\gotm}\to \Spec A$, see~\cite[Tag~00M9]{stacks_project}.
Similarly, if $X$ is reduced and irreducible, then $X =
\Spec A$ for a domain $A$ and the \emph{normalization} of $X$ is $\Spec \tilde{A} \to
\Spec A$, where $\tilde{A}$ is the integral closure of $A$ in its fraction
field.
In fact, completion and normalization may be defined also for all reduced
schemes (not necessarily affine or irreducible) by gluing.
We say that that a reduced scheme is \emph{normal}, if it is equal to its normalisation.  

One source of closed schemes is the construction of the scheme-theoretic
image of a morphism $f:X\to Y$. Consider all closed subschemes $Z \subset Y$
such that $f$ factors as $f:X\to Z\to Y$. Their intersection also satisfies
this factorization property, so we define the \emph{scheme-theoretic image of $f$} as the
unique smallest $Z$ such that $f:X\to Z\to Y$. This is always
a closed subscheme; under a mild condition it is the closure of the set $f(X)$,
see~\cite[Section~8.3, Theorem~8.3.4]{vakil_FoAG}.

\subsection{Fibre product and base change}\label{sec_product_and_base_change}

Suppose $X = \Spec A$  and $Y =\Spec B$ are two affine schemes (over $\kk$).
Their \emph{product} is defined as $X \times Y := \Spec A \otimes_{\kk} B$.
If $X$ is any scheme, and $Y$ is affine, then the product is glued locally from affine open pieces $U \times Y$ with $U\subset X$ affine open subset.
If $X$ and $Y$ are arbitrary schemes, then the product is glued locally from open pieces $X \times V$ with $V\subset Y$ affine open subset.

There are natural \emph{projection} maps $X\times Y \to X$ and $X\times Y \to Y$ locally determined by $A \to A\otimes_{\kk} B$, $a \mapsto a \otimes 1$,
  and $B \to A\otimes_{\kk} B$, $b \mapsto 1 \otimes b$.
The projection maps are universal in the sense, that if there is a scheme $W$ with two maps $W \to X$ and $W \to Y$, then there is a unique morphism $W \to X\times Y$ that makes the diagram commutative:
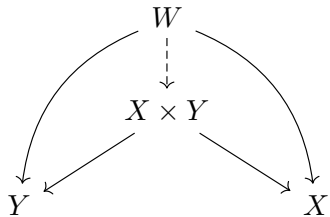
\[
\begin{tikzcd}
{} & W
\arrow[bend right,swap]{ddl}{}
\arrow[bend left]{ddr}{}
\arrow[dashed]{d}{} & & \\
& X \times Y \arrow{dr}{} \arrow{dl}[swap]{} \\
Y  & & 
X 
\end{tikzcd}
\]

The product described above is a simple case of the more general fibre product.
Suppose now $X = \Spec A$, $Y =\Spec B$ and $Z = \Spec C$  are three affine schemes and fix two morphisms $X \to Z$ and $Y \to Z$.
We define the fibre product $X \times_{Z} Y : = \Spec A \otimes_C B$, where the $C$ algebra structure of $A$ and $B$ is given by the underlying morphisms $C \to A$ and $C \to B$.
As above, we glue the affine pieces of $X$ and $Y$ to get $X\times_{\Spec C } Y$ for any schemes $X$, $Y$ with morphism to and affine scheme $\Spec C$.
If $Z$ is arbitrary, we cover it with affine open subsets $U$ and replace $X$ and $Y$ by the preimages of $U$.
Then we glue these together.
For details see~\cite[Chapter 9]{vakil_FoAG}. By construction $X \times Y = X\times_{\Spec \kk} Y$.

As above, there are usual natural maps and a commutative diagram:
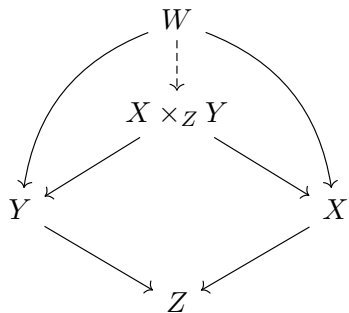
\[
\begin{tikzcd}
{} & W
\arrow[bend right,swap]{ddl}{}
\arrow[bend left]{ddr}{}
\arrow[dashed]{d}{} & & \\
& X \times_{Z} Y \arrow{dr}{} \arrow{dl}[swap]{} \\
Y \arrow[swap]{dr}{} & & 
X \arrow{dl}{} \\
& Z
\end{tikzcd}
\]

Bizarre glueings of the affine pieces might lead to pathological schemes that are often excluded from considerations.
In this paper, we assume that all schemes are \emph{separated}, 
  that is the diagonal $\Delta \to X \times X$ is a (Zariski) closed subset of the product
  $X \times X$, see~\cite[Sect.~III.1.2]{eisenbud_harris},
  \cite[Chap.~9]{gortz_wedhorn_algebraic_geometry_I}, or \cite[Sect.~10.1]{vakil_FoAG}.
As indicated in Notation~\ref{not_scheme}, we are always assuming our schemes are separated.
This is the algebro-geometric analogue of the Hausdorff property of topological spaces, and similarly to the theory of topological or differentiable manifolds,
  such assumption is often required to assure nice behaviour.

  A scheme $X$ is \emph{proper} (over $\kk$) if it is of finite type,
  separated and for every scheme $Y$ the underlying topological map of the morphism of schemes $X \times Y \to Y$ is
  closed. Every $X \subset \PP V$ is proper. Properness may be thought as an
  abstraction of key properties of projective varieties.
  
It is often desirable to think of fibre product as of a \emph{base change}. 
That is, suppose $f \colon X \to Z$ is a fixed morphism, where we think of $Z$ as of a \emph{base}.
Let $Y \to Z$ be some morphism. 
Then $f_Y\colon X \times_{Z} Y \to Y$ is now a morphism with base changed to $Y$, called a \emph{base change} of $f$.
\begin{example}
  Suppose $Y \to Z$ is an immersion of a subscheme. 
  Then the base change is the restriction of the map $f\colon X \to Z$ to the preimage of $Y$.
  Particularly, if $Y$ is a single reduced point $y$, then the base change is the fibre $f$ over $Y$ often denoted $X_y$.
\end{example}

\begin{example}
   Suppose $\kk \subset \KK$ is an extension of fields.
   Then using fibre product we can replace any scheme $X$ (over the field
   $\kk$) with $X_{\KK} := X \times \Spec \KK$ to obtain a scheme over
   $\KK$, and we may change the base field to $\KK$.
   If $X$ is finite  over $\kk$ (see Section~\ref{sect_finite_schemes}),
     then also $X_{\KK}$ will be finite \emph{over $\KK$.} 
   In general $X_{\KK}$ is \emph{not} finite over
   $\kk$. This illustrates the critical importance of keeping track of the base field.
\end{example}

Many properties of morphisms are preserved under base change. 
For example, if $f \colon X \to Z$ is flat (respectively: smooth, proper, or finite), then also $f_Y \colon X \times_{Z} Y \to Y$ is flat (respectively: smooth, proper, or finite). 
See Sections~\ref{sec_smooth_regular_flat}, \ref{sect_finite_schemes} for a definitions of a flat, smooth, proper and finite morphism.
See \cite[Tag~02WE]{stacks_project} for a longer list of properties preserved under base change.

It is important to remark, that the underlying topological space of the product $X \times Y$ is not necessarily equal to the product of topological spaces of $X$ and $Y$.
A classical example is $\AA^1_{\kk} \times \AA^1_{\kk} \simeq \AA^2_{\kk}$. The non-closed points of $\AA^2_{\kk}$ which come from the topological product, 
  are ``vertical'' or ``horizontal'', such as those corresponding to ideal $(x-c)$ or or $(y-d)$ (where $x$ and $y$ are the coordinates on the first or on the second factor, 
  and $c$, $d$   are constants in $\kk$). There are many other prime ideals, one example is $(x-y)$.
As another example, let $\kk= \RR$ and consider $X=Y=\Spec \CC = \Spec \RR[x]/(x^2+1)$.
Thus each $X$ and $Y$ is a single point.
Then $ X\times_{\Spec\RR} Y \simeq \Spec \CC[x]/(x^2+1) \simeq \Spec \CC \sqcup \Spec \CC$, i.e.~$X\times Y$ consists of two closed points.
Nevertheless, there are many topological properties that behave analogous to the topological product.
We note one of them that will be used later. 
It uses the notion of $\kk$-rational points of a scheme $X$, i.e. points of $X$ with residue field equal to $\kk$, see Section~\ref{sect_varieties}.
\begin{lem}\label{lem_product_of_dense_is_dense}
  Suppose $X$ and $X'$ are finite type schemes (over $\kk$) and $S \subset X$ and $S'\subset X'$ are
  two Zariski-dense subsets consisting of $\kk$-rational points.
  Then $S \times S' \subset X \times X'$ is a Zariski-dense subset.
\end{lem}

\begin{proof}
  Pick an nonempty open subset $U\subset X \times X'$.
  We must show, that there is a point $s \times s' \in U$ with $s \in S$, and $s' \in S'$.
  The scheme $X'$ is finite type and it is flat over $\kk$
  trivially, so the projection morphism $X \times X' \to X$ is
  open~\cite[Tag~01UA]{stacks_project}.
  The image $V \subset X$ of $U$ under this projection is also open (and non-empty).
  Choose a $\kk$-rational point $s \in V\cap S$.
  Then the scheme-theoretic fibre $\set{s}\times X'$ is isomorphic to $X'$ and
  intersects $U$ in a nonempty subset, i.e.~$U \cap (\set{s}\times X')$ is an
  open subset of $X'$.
  In particular, there exists an $s' \in S'$ contained in this open subset.
  This produces a point $\set{s\times s'}$ from $S \times S'$ contained in $U$.
\end{proof}

\subsection{Smooth and regular schemes, smooth and flat morphisms}\label{sec_smooth_regular_flat}

A scheme $X$ is \emph{regular}, if every closed point $x \in X$ (corresponding
to a maximal ideal $\gotm \subset \ccO_X(U)$ for some affine open subset $U \subset X$)
   satisfies $\dim(X,x) = \dim_{\kappa(x)} \gotm/\gotm^2$, where $\kappa(x) =
   \ccO_{X,x}/\gotm$ is the residue field of $x$,
   see~\cite[Tag~02IR]{stacks_project}, \cite[Section 6.11]{gortz_wedhorn_algebraic_geometry_I},
   or~\cite[Section~12.2]{vakil_FoAG}.

A scheme $X$ is \emph{smooth of dimension $n$} if every irreducible component of $X$ has dimension $n$ 
  and $X$ can be covered by affine open subsets 
  $\Spec\kk[\fromto{\alpha_1}{\alpha_k}]/(\fromto{f_1}{f_l})$, such that the Jacobian matrix of the functions $f_i$ has corank $n$.
Every smooth scheme is regular, but not vice-versa,
see~\cite[Tag~00TQ]{stacks_project}, \cite[Theorem
6.28]{gortz_wedhorn_algebraic_geometry_I}, or
\cite[Theorem~12.2.10]{vakil_FoAG}.

\begin{example}
   The scheme $\Spec \kk[[\fromto{\alpha_1}{\alpha_n}]]$ is regular, but it is not of
   finite type, so it is not smooth. However it is \emph{formally smooth}, see
   Section~\ref{sect_smoothings_abstract_embedded}.
\end{example}
\begin{example}
    Let $\kk \subset \KK$ be a field extension. 
    Then $\Spec \KK$ is always regular. 
    It is smooth if and only if it is algebraic and separable over $\kk$, see \cite[tag 02GL and tag
    02GH]{stacks_project}.
\end{example}

More generally, we define a \emph{smooth morphism of relative dimension $n$} in several steps. 
Firstly, suppose $X = \Spec A$ and $Y =\Spec B$ are affine schemes with a
closed immersion $X \subset Y \times \AA_{\kk}^N$,
 so that 
  $X \simeq \Spec B[\fromto{\alpha_1}{\alpha_N}]/(\fromto{f_1}{f_k})$ for some polynomials $f_i \in  B[\fromto{\alpha_1}{\alpha_N}]$.
  Then the natural projection morphism $f\colon X \to Y$ is smooth of relative dimension $d$ if and only if 
  the Jacobian matrix of the defining equations is of rank $N-d$ at every point of $X$. 
Further, open immersions (i.e.~embeddings of open subschemes) are smooth morphisms of relative dimension $0$. 
A composition of smooth morphisms is a smooth morphism of relative dimension which is a sum of the respective relative dimensions.
Finally, smoothness of a morphism is a local property.
Altogether, a combination of the above properties describe all possible smooth
morphisms, see \cite[Tags~01V5 and 01V7]{stacks_project}, \cite[Section~(6.8)]{gortz_wedhorn_algebraic_geometry_I},
or~\cite[Definition~12.6.2]{vakil_FoAG} for details.

\begin{example}
   If $Y = \Spec \kk$, then $X$ is smooth of dimension $n$ if and only if the natural morphism $X \to Y$ is smooth of relative dimension $n$.
\end{example}

A morphism of affine schemes $f\colon X \to Y$, $X = \Spec A$ and $Y =\Spec B$ is \emph{flat} if $A$ is a flat $B$-module under the underlying morphism $f^*\colon B \to A$ determining the $B$-module structure of $A$.
Again, flatness is a local property, providing a definition of flat morphism for general schemes $X$ and $Y$, see also \cite[Section~(7.18) and Chapter~14]{gortz_wedhorn_algebraic_geometry_I}.
A morphism of schemes $f\colon X\to Z$ is \emph{of finite type} if there is an open
cover $\{U_i\}$ of $Z$ such that $f^{-1}(U_i) \to U_i$ factors as $f^{-1}(U_i)
\subset U_i \times \AA_{\kk}^N \to U_i$.
Assuming Notation~\ref{not_scheme}, 
  a morphism of schemes $X\to Z$ is \emph{proper} if it is of finite
type and for every $Y\to Z$ the topological map $X \times_Z Y \to Y$ is
closed. Therefore, fibres of a proper morphism are proper schemes.
Flat morphisms serve to define continuous families in algebraic geometry and
proper morphisms encode the notion of compactness. Assuming a morphism is
flat and proper, many important properties of fibres are preserved under
limits, but in general, fibres of morphism may vary wildly:

\begin{example}
   Let $X = \Spec\kk[\alpha_1,\alpha_2]/(\alpha_1\alpha_2, \alpha_2^2 -\alpha_2)$ be the union of the $\alpha_2=0$ line and $(0,1)$ point in the affine plane.
   Suppose $Y = \Spec \kk[\beta]$ is a line and $f \colon X \to Y$ is a projection map $(a_1,a_2) \mapsto (a_1)$.
   Then $f$ is proper, but not flat. The fibres over a point $b \in Y$, $b \ne 0$, consist of single reduced point, 
     while the fibre over $0 \in Y$ consists of two disjoint points.
\end{example}

\begin{example}
   Let $X = \Spec\kk[\alpha_1,\alpha_2]/(\alpha_1\alpha_2-1)$ be $\AA_{\kk}^1 \setminus \set{0}$, the open subset of affine line.
   Suppose $Y = \Spec \kk[\beta]$ is the affine line and $f \colon X \to Y$ is the open immersion map $(a_1,a_2) \mapsto (a_1)$.
   Then $f$ is flat but not proper and the fibres over point $b \in Y$, $b \ne 0$, consist of single reduced point, 
     while the fibre over $0 \in Y$ is empty.
\end{example}

\subsection{Rational points and varieties}\label{sect_varieties}

The set of \emph{$\kk$-rational points} of a scheme $X$, denoted $X(\kk)$, is
the set of morphisms $\Hom(\Spec \kk \to X)$ over $\kk$.
Equivalently, this is the set of points $x$ in $X$, such that the residue
field $\ccO_{X,x}/\gotm_x$ is equal to $\kk$.
It follows that $\kk$-rational points are special closed points of $X$.
If $\kk$ is algebraically closed, then the set of $\kk$-rational points is equal to the set of closed points.
Moreover, $X(\kk) = \reduced{X}(\kk)$, where $\reduced{X}$ is the largest reduced
subscheme of $X$.
\begin{example}\label{exam_quadrics_in_the plane_real_points}
  Suppose $\kk=\RR$ is the field of real numbers.
  Let define the following schemes:
    $X=\Spec \RR[\alpha,\beta]/(\alpha^2+\beta^2+1)$,  
    $Y=\Spec \RR[\alpha,\beta]/(\alpha^2+\beta^2)$,
    $Z=\Spec \RR[\alpha,\beta]/(\alpha^2+\beta^2-1)$,
    $T=\Spec \RR[\alpha,\beta]/\alpha^2$.
  Then the respective sets of points are: 
  $X(\RR)$ is empty,
  $Y(\RR)$ is the origin of the affine plane,
  $Z(\RR)$ is a circle,
  $T(\RR)$ is a line.
\end{example}

A scheme $X$ is called a \emph{variety} if it is of finite type over $\kk$ and reduced. 
The first condition implies that $X$ has a open covering by $U_i
\subset \mathbb{A}^{n_i}$. The second condition tells that those $U_i$ are not
glued pathologically.
The third condition is sometimes replaced by \emph{reduced and irreducible},
  but here we also need to consider ``reducible varieties''.
In other words, in the affine case, $X = \Spec A$ is a variety if and only if $A = \kk[\fromto{\alpha_1}{\alpha_N}]/I$ with $I \subset \kk[\fromto{\alpha_1}{\alpha_N}]$ a radical ideal.
The same holds if $X$ is projective and $I$ is the homogeneous ideal defining
$X$ in a projective space (see Section~\ref{sec_embedded_proj_geom}).
   
If $X$ is affine,
then $X(\kk)$ is exactly the set of solutions of the collection of polynomial equations of
  the ideal $I$ defining $X$ in the affine space.
Also vice-versa, if $X$ is an affine variety such that $X(\kk)$ is dense in $X$ and if $I$ is the ideal of all  polynomials vanishing on $X(\kk)$,
  then $I$ is the ideal defining $X$.
\begin{example}
    Schemes $X$, $Y$, $Z$ appearing in Example~\ref{exam_quadrics_in_the
    plane_real_points} are varieties, whereas $T =\Spec \RR[\alpha,\beta]/\alpha^2$ is
    non-reduced and thus not a variety.
\end{example}

\begin{example}
   The scheme $\Spec \kk[[\fromto{\alpha_1}{\alpha_n}]]$ is not a variety, since it is not of finite type.
\end{example}

A smooth irreducible variety is regular and a regular variety over a perfect
field is smooth, see~\cite[Tag 056S, 0B8X]{stacks_project}. Moreover a variety
$X$ is regular of dimension $n$ if and only if for every closed point $x \in
X$ the ring $\hat{\ccO}_{X, x}$ is isomorphic to
$\KK[[\fromto{\alpha_1}{\alpha_n}]]$, where 
$\KK = \kappa(x)$ is the residue field of $x$,  $\hat{\ccO}_{X, x} = \varprojlim\ccO_{X,
x}/\gotm^i$ is the completion  of the local ring $\ccO_{X,
x}$ at maximal ideal $\gotm \subset \ccO_{X, x}$, see~\cite[Tag 07NY]{stacks_project} together with
\cite[Theorem~7.7]{eisenbud}.
In the case $\kk =\CC$, this is analogous to the statement that a variety is smooth if every point has an Euclidean open neighbourhood biholomorphic to an open disk.

\subsection{Finite schemes and morphisms}\label{sect_finite_schemes}

A scheme $R$ is said to be \emph{finite} over $\kk$
   if one of the following equivalent conditions holds:
\begin{enumerate}
 \item  \label{item_def_finite_SpecA}
         $R \simeq \Spec A$ for a finite $\kk$-algebra $A$.
 \item  \label{item_def_finite_union_SpecAi}
         $R$ is a finite disjoint union $R = R_1 \sqcup R_2 \sqcup \dotsb \sqcup R_k$, where each $R_i = \Spec A_i$
         and $A_i$ is a finite local $\kk$-algebra.
\end{enumerate}
The \emph{degree} of a finite scheme $R = \Spec A$ is $\dim_{\kk} A$ as in~\ref{item_def_finite_SpecA}, 
the dimension of $A$ as a $\kk$-vector space; see 
Section~\ref{sect_Hilbert_intro} for a motivation for this name.
If $\kk$ is algebraically closed, then $A$ has a submodule isomorphic to $\kk$,
so by an inductive argument the degree of $A$ is equal to its
Jordan-H\"older \emph{length} as an $A$-module. This is why some authors prefer the name
\emph{length}. For $\kk$ non-algebraically closed the notion of length and
degree diverge, as Example~\ref{ex_length_degree} shows.
\begin{example}\label{ex_length_degree}
    If $\kk = \mathbb{R}$ and $A = \mathbb{C}$ viewed as a
    $\mathbb{R}$-algebra, then the length of $A$ is one
    and its degree is two.
\end{example}

If $R = R_1 \sqcup \dotsb \sqcup R_k$ as in \ref{item_def_finite_union_SpecAi}, 
   then the degree is the sum of degrees (over $\kk$) of $A_i$.
   
\begin{example}
   The affine scheme $X=\Spec \kk[[\alpha]]$ consists of $2$ points corresponding to prime ideals $\set{0}$ and $(\alpha)$.
   Nevertheless, $X$ is \emph{not} finite, as the algebra $\kk[[\alpha]]$ is infinite dimensional.
\end{example}
   
\begin{example}
   Up to an isomorphism, there is a unique finite scheme over $\kk$ of degree
   one, namely $\Spec \kk$.
\end{example}

\begin{example}\label{ex_degree_2_schemes}
   Degree two finite schemes can be either $\Spec \kk \sqcup \Spec \kk \simeq \Spec (\kk\times \kk)$, or $\Spec \kk[\epsilon]/(\epsilon^2)$ or $\Spec \KK$,
   where $\kk\subset \KK$ is a degree two field extension.
\end{example}

Every finite $\kk$-scheme $R$ can be embedded into an affine space
$\mathbb{A}^n_{\kk}$. To see this, we take $R  \simeq \Spec A$, choose a basis
$a_1, \ldots ,a_n$ of $A$ viewed a $\kk$-vector space and take an epimorphism $\kk[\alpha_1, \ldots ,\alpha_n]\to A$
sending $\alpha_i$ to $a_i$.

Suppose now $R$ is irreducible, so that $A$ is local; let $\gotm$ be its maximal
ideal. Then $n$ above has to satisfy~$n\geq \dim_{\kk} \gotm/\gotm^2$.
Therefore, we call $\dim_{\kk} \gotm/\gotm^2$ the \emph{embedding dimension}
of $R$ over $\kk$.

If $R$ is finite, then it is smooth if and only
if $R= R_1 \sqcup \dotsc \sqcup R_s$, where each $R_i \simeq \Spec \FF_i$ and
$\kk  \subset \FF_i$ is a finite separable field extension; see \cite[tag 02GL
and tag 02GH]{stacks_project}.  If in addition $\kk$ is algebraically closed,
then $R$ is finite of degree $r$ and smooth if and only if $R= R_1 \sqcup
\dotsc \sqcup R_r$, where each $R_i \simeq \Spec \kk$.
A finite scheme is irreducible if and only if its topological space consists of a single point. 
Irreducible components of a finite scheme are uniquely determined.

A morphism of affine schemes $f\colon X \to Y$, $X = \Spec A$ and $Y =\Spec B$
is finite, if $A$ is a finitely generated $B$-module.
Thus, $R$ is finite if and only if the morphism $R\to \Spec \kk$  is finite.
Any finite morphism is proper.

For further use, we define Gorenstein schemes. Let $R = \Spec A$ be a finite
$\kk$-scheme. Then $A^* = \Hom_{\kk}(A, \kk)$ is an $A$-module by $A \times
A^*\ni (a, f) \to af \to A^*$ where $(af)(b) = f(ab)$.

\begin{defn}[{\cite[Section~21.2]{eisenbud}}]\label{ref:Gorenstein:def}
    We say that $R = \Spec A$ is \emph{Gorenstein} if $A^*$ is a principal
    $A$-module (i.e.,~it is generated by a single element).
\end{defn}

The module $A^*$ is in fact a dualizing module for $A$. Gorenstein rings are
intimately tied with injective dimension of $A$ and dualizing complexes,
as~\cite[Tag~0AWV]{stacks_project}. We will not explicitly use these advanced
devises.

Definition~\ref{ref:Gorenstein:def} is a little cryptic; for $R$ over $\kkbar$
it can be formulated differently.
\begin{lem}\label{ref:Gorchar:eis:lem}
    Let $R = R_1 \sqcup  \ldots \sqcup R_k$ be a finite scheme over $\kkbar$
    and $R_i$ be its irreducible components. Let $R_i = \Spec A_i$ and $\gotm_i
    \subset A_i$ be the maximal ideal. Then
    \begin{enumerate}
        \item\label{it:tmpfirst} $R$ is Gorenstein if and only if each $R_i$ is Gorenstein,
        \item\label{it:tmpsecond} $R_i$ is Gorenstein if and only if $\dim_{\kkbar} (0:\gotm_i) =
            1$.
    \end{enumerate}
\end{lem}
\begin{proof}
    According to Definition~\ref{ref:Gorenstein:def}, the scheme $R_i$ is
    Gorenstein if $H_i = \Hom_{A_i}(A_i, A_i^*)$ contains an epimorphism.
    We have $R = \Spec A$ with $A = \prod A_i$. Let $\gotn_i \subset A$ be the
    maximal ideal corresponding to the closed point of $R_i$.
    The module $H_i$ is the stalk of $\Hom_{A}(A, A^*)$ at $\gotn_i$, so
    $\Hom_A(A, A^*)$ contains an epimorphism if and only if all $H_i$
    do.
    Point~\ref{it:tmpfirst} follows. For Point~\ref{it:tmpsecond}
    see~\cite[Proposition~21.5]{eisenbud}, noting that, in the language of
    this reference, the socle of $A_i$ is exactly $(0:\gotm_i)$.
\end{proof}

\begin{example}\label{ex_Gorenstein_schemes}
    The scheme $R = \Spec \kkbar[x, y]/(x^2, y^2)$ is
    Gorenstein (over $\kkbar$). Indeed the maximal ideal is $\gotm = (x, y)$ and its
    annihilator $(0:\gotm)$ is $\kkbar\cdot xy$. Similarly, $R = \Spec \kkbar$
    is Gorenstein (over $\kkbar$). On the contrary, $R = \Spec
    \kkbar[x, y]/(x^2, xy, y^2)$ is not Gorenstein (over $\kkbar$); the maximal ideal is
    annihilated by $(x, y)$, which is two dimensional.
\end{example}

By the following proposition we may always reduce checking Gorenstein property
to the scheme over $\kkbar$.
\begin{prop}\label{prop_Gorenstein_base_change}
    Let $R$ be a finite scheme over $\kk$. Then the following are equivalent:
    \begin{enumerate}
        \item\label{it:Gorbcfirst} $R$ is Gorenstein (over $\kk$),
        \item\label{it:Gorbcsecond} $R' = R \times_{\kk} \KK$ is Gorenstein for some field
            extension $\KK \supset \kk$,
        \item\label{it:Gorbcthird} $R' = R \times_{\kk} \KK$ is Gorenstein for every field
            extension $\KK \supset \kk$,
    \end{enumerate}
\end{prop}
\begin{proof}
    This is proven in~\cite[tag~0C03]{stacks_project}.
\end{proof}

\section{Smoothability of finite schemes}\label{sec:smoothability}

In this section we define abstract and embedded smoothings and we discuss the interactions and properties of these two notions.
We restrict the presentation to the main case of our interest, that is, smoothability of finite schemes.
Recall that we constantly assume that all schemes are separated and locally Noetherian (Notation~\ref{not_scheme}).

    \begin{defn}[abstract smoothing]\label{ref:abstractsmoothable:def}
        Let $R$ be a~finite scheme over $\kk$. We say that \emph{$R$ is abstractly
        smoothable} (over $\kk$) if there exist an irreducible scheme
        $T$ and a morphism of schemes $Z\to T$
        such that
        \begin{enumerate}
            \item $Z\to T$ is flat and finite,
            \item $T$ has a $\kk$-rational point $t$, such that the fibre $Z_t \simeq R$. We call
                $t$ the \emph{special point} of $T$.
            \item $Z_{\eta}$ is a~smooth scheme over $\eta$, where $\eta$
                is the generic point of $T$.
        \end{enumerate}
        The scheme $Z$ is called an \emph{abstract smoothing} of $R$. We will
        sometimes
        denote it by $(Z, R)\to (T, t)$, which means that $t$ is the $\kk$-rational point of $T$, such that
        $Z_t  \simeq R$.
    \end{defn}

Informally, $R$ is smoothable if and only if it is a \emph{limit} of $r$ distinct points.
In abstract terms, it is not entirely clear what does it mean ``limit'',
  which is why the definition involves many technical terms and depends on the
  field $\kk$. Some articles consider only schemes over an algebraically
  closed field $\kkbar$, so the dependence is less explicit. We will later
  show that $R$ is smoothable is and only its base change $R'$ is smoothable,
  see Corollary~\ref{ref:basechangesmoothings:cor}.

\begin{example}
    Any finite smooth scheme $R$ is smoothable via a trivial smoothing $Z
    =R$, $T = \Spec \kk$.
\end{example}

\begin{example}\label{ex_smoothability_of_degree_2_schemes}
    Any  scheme $R$ of degree $2$ is (abstractly) smoothable.
    Indeed, by Example~\ref{ex_degree_2_schemes},
      the scheme $R$ is isomorphic to one of the following:
      \begin{itemize}
       \item $\Spec \kk \sqcup \Spec \kk$, which is smooth, or
       \item $\Spec \kk[\epsilon]/(\epsilon^2)$, 
             which has a smoothing $(Z, R) \to (\AA_{\kk}^1,0)$, where $Z = \Spec \kk[\epsilon,t]/(\epsilon^2 - t\epsilon)\simeq \AA_{\kk}^1 \cup \AA_{\kk}^1$, or       
       \item $\Spec \KK$, where $\kk\subset \KK$ is a degree two separable field extension, so that $R$ is smooth, or
       \item $\Spec \KK$, where $\kk\subset \KK$ is a degree two nonseparable field extension,
              thus $\cchar \kk = \cchar \KK =2$, and $\KK \simeq \kk[\epsilon]/(\epsilon^2 -a)$ for $a \in \kk$, which is not a square.
             In this case, the smoothing is 
               $(Z, R) \to (\AA_{\kk}^1,0)$, where $Z = \Spec \kk[\epsilon,t]/(\epsilon^2- a - t\epsilon)\simeq \AA_{\kk}^1 \cup \AA_{\kk}^1$.
             It is straightforward to check that for a constant $t_0\in \kk \setminus\set{0}$, the algebra $\kk[\epsilon]/(\epsilon^2- t_0\epsilon- a)$ 
               is either isomorphic to $ \kk \times \kk$ or a separable field extension of $\kk$.
      \end{itemize}
      More generally for every finite field extension $\kk\subset \KK$ the
      $\kk$-scheme $R := \Spec \KK$
    is smoothable. For $\kk \subset \KK$ separable this is immediate,
    but for inseparable requires a little work, see
    Example~\ref{ex:fieldExtensions}.
\end{example}

\begin{example}\label{ex:fieldExtensions}
    For every finite field extension $\kk\subset \KK$ the $\kk$-scheme $R =
    \Spec \KK$ is smoothable.  Suppose first that $\kk \subset
    \KK$ is  a separable extension.  Then $R = \Spec \KK$ is smooth over
    $\kk$, see Section~\ref{sect_finite_schemes}, so it is trivially smoothable.
    Suppose now that $\kk \subset \KK$ is not separable.
    Every extension of a finite field is separable, so we may assume $\kk$ is infinite.
    We may decompose $\kk \subset \KK$ as a chain of one-element extensions
    $\kk \subset \kk(t_1) \subset \kk(t_1,t_2) \subset  \ldots
    \subset \kk(t_1, \ldots ,t_n) = \KK$.
    Then $\KK= \kk[\alpha_1, \ldots ,\alpha_n]/(f_1, \ldots ,f_n)$ where $f_i$ is the lifting of minimal polynomial of $t_i$; in
    particular $f_i = \alpha_i^{d_i} + a_{i,d_i-1}\alpha_i^{d_i-1} +  \ldots  + a_{i, 0}$, where
    $a_{i,j}\in \kk[\alpha_1, \ldots ,\alpha_{i-1}]$.
    We now inductively construct, for $i=1, \ldots ,n$, polynomials $F_i\in \kk[\alpha_1, \ldots
    ,\alpha_n][t]$ such that:
    \begin{enumerate}
        \item\label{it:cond1} The family $Z = \Spec \kk[\alpha_1, \ldots ,\alpha_n,t]/(F_1, \ldots ,F_n)\to
            \Spec \kk[t]$ is flat and finite,
        \item\label{it:cond2} $F_i(0) = f_i$, so that $Z_{|t=0} = \Spec \KK$,
        \item\label{it:cond3} The fibre $Z_{|t=1}$ is a disjoint union of copies of $\kk$.
    \end{enumerate}
    First, we construct polynomials $g_i$ such that $g_i$ has degree $d_i =
    \deg(f_i)$ and
    \begin{equation}\label{eq:formofgi}
        g_i = \alpha_i^{d_i} + b_{i,d_i-1}\alpha_i^{d_i-1} +  \ldots  + b_{i, 0}\quad
        \mbox{where}\quad
        b_{i,j}\in \kk[\alpha_1, \ldots ,\alpha_{i-1}]
    \end{equation}
    and $\kk[\alpha_1, \ldots ,\alpha_n]/(g_1,
    \ldots ,g_n)$ is a product of $\kk$. This is done inductively. We choose $g_1$ as any
    polynomial of degree $d_1$ having $d_1$ distinct roots in $\kk$, then
    $\kk[\alpha_1]/g_1  \simeq \kk^{\times d_1}$.
    We choose $g_2$ as a polynomial of degree $d_2$ in $\kk[\alpha_1,\alpha_2]/g_1 \simeq
    (\kk[\alpha_2])^{\times d_1}$ such that after projecting to each factor $g_i$
    has $d_2$ distinct roots in $\kk$, then $\kk[\alpha_1,\alpha_2]/(g_1, g_2)  \simeq
    \kk^{\times d_1d_2}$ and so we continue.

    Define $F_i = (1-t)f_i + t g_i$, then $F_i =
    \alpha_i^{d_i}+((1-t)a_{i,d_{i}-1}+tb_{i,d_i-1})\alpha_i^{d_i-1} + \ldots
    +((1-t)a_{i,0}+tb_{i,0})$ where $(1-t)a_{j,d_{i}-1}+tb_{j,d_i-1}\in
    \kk[\alpha_1, \ldots ,\alpha_{i-1}][t]$.
    From this ``upper-triangular'' form of $F_i$ we see that the quotient
    \[
        \kk[\alpha_1, \ldots ,\alpha_n,t]/(F_1, \ldots ,F_n)
    \]
    is a free $\kk[t]$-module with basis consisting of monomials
    $\alpha_1^{s_1} \ldots \alpha_n^{s_n}$ such that $s_i < d_i$ for all $i$.
    Hence, condition~\ref{it:cond1} is satisfied;
    conditions~\ref{it:cond2},\ref{it:cond3} are satisfied by construction. In
    particular, a fibre of $Z$ is smooth, so the generic fibre is also smooth, thus
    the family $Z$ is a smoothing of $\kk$-scheme $\Spec \KK$.
\end{example}

It is nontrivial to provide examples of schemes that are nonsmoothable, however these are known to exist in abundance, 
at least in sufficiently large degree~\cite{iarrobino__reducibility, iarrobino_compressed_artin}.
We recall their existence in
Example~\ref{ref:nonsmoothableOfHighDegree:ex}; for now on we just assume they exist over any base field $\kk$.

We now introduce \emph{embedded smoothings}.
The difference between the previous setting is in the presence of the ambient
scheme $Y$, where the whole smoothing must be embedded. We will
see in Section~\ref{sect_smoothings_abstract_embedded} that when $Y$ is
smooth, the notions are equivalent.

    \begin{defn}[embedded smoothing]\label{ref:smoothable:def}
        Let $Y$ be a~scheme and $R$ be a~finite closed subscheme of $Y$. We say that
        $R$ is \emph{smoothable in $Y$} if there exist an irreducible scheme
        $T$ and a~closed subscheme $Z \subseteq Y \times T$ such that $Z\to T$ is
        an abstract smoothing of $R$.
        The scheme $Z$ is called an \emph{embedded
        smoothing} of $R \subseteq Y$.
    \end{defn}

\begin{example}
   Let $Y= \Spec \kk[\alpha,\beta,\gamma]/(\alpha\beta,\alpha\gamma,\beta\gamma)$, i.e.~$Y$ is the union of three coordinate lines 
     in the three dimensional affine space $\AA_{\kk}^3$.
   Let $R= \Spec \kk[\alpha,\beta,\gamma]/(\alpha-\beta, \alpha-\gamma, \alpha^2) \simeq \kk[\epsilon]/\epsilon^2$
   be the degree two subscheme of $Y$,
     which is the intersection of $Y$ with the affine line $\alpha=\beta=\gamma$. 
   Then $R$ is abstractly smoothable, but $R$ is not smoothable in $Y$.
\end{example}

    In the embedded setting, if in addition $\kk$ is algebraically closed, we
    can be more explicit about the ``limit'' mentioned above; it is the limit
    inside the Hilbert scheme of points, see
    Proposition~\ref{ref:embeddedvsHilbert:prop}.

\subsection{Changing base for smoothings}

Smoothings benefit from a nice base-change property, which implies interesting results, 
  involving modifications of the base $T$ of a smoothing.
The first trivial case is to extend the basis by multiplying it by another scheme $T'$.
  
\begin{example}\label{ex_change_base_by_product}
    If $(Z,R) \to (T,t)$ is an abstract smoothing with affine $T$, 
      and $T'$ is another affine irreducible scheme with $\kk$-rational point $t' \in T'$, 
      then $(Z\times T', R) \to (T\times T',t\times t')$ is also a smoothing with modified base.
\end{example}   
    
More generally, we have the following base change.
    
    \begin{lem}[Base change for smoothings]\label{ref:basechange:lem}
        Let $T$ be an irreducible scheme with the generic point $\eta$ and a
        $\kk$-rational point $t$.  Let $(Z, R)\to (T, t)$ be an abstract~smoothing of a~finite scheme $R$.
        
        Suppose $T'$ is another irreducible scheme with a~morphism $f:T' \to T$ such that
        $\eta$ is in the image of $f$ and there
        exists a~$\kk$-rational point of $T'$ mapping to $t$. Then
        the base change $Z' = Z \times_T T' \to T'$ is an abstract~smoothing
        of $R$.
        Moreover, if $R$ is embedded into some $Y$ and
        $Z\subseteq Y\times T$ is an embedded smoothing, then $Z' \subseteq
        Y\times T'$ is also an embedded smoothing of $R\subseteq Y$.
    \end{lem}

    \begin{proof}
        $Z'\to T'$ is finite and flat. The generic point $\eta'$ of $T'$ maps
        to $\eta$ under $f$ so that $Z'_{\eta'}\to \eta'$ is a~base change of a
        smooth morphism $Z_{\eta}\to \eta$. In particular it is smooth, so
        that $Z'\to T'$ is a~smoothing of $R$. If $Z \subseteq Y\times T$ was
        a~closed subscheme, then $Z'\subseteq Y\times T'$ is also a~closed
        subscheme.
    \end{proof}
    
    The next lemma can be informally summarised as follows: if $U$ is an open subset of a scheme $T$, and $t$ is a point in the closure of $U$, 
       then there exists a curve in $T$ through $t$ intersecting $U$.
    \begin{lem}\label{lem_finding_a_curve_through_point_and_open_subset}
        Suppose $T$ is a scheme, $U \subset T$ is an open subset and $t\in T$ is a point in $\overline{U}$ (for example, any point if $U$ is dense).
        Suppose the residue field of $t$ is $\kappa$. 
        Then there exists a one-dimensional Noetherian complete local domain
        $A'$ with residue field $\kappa$, 
           and a morphism $T'= \Spec A' \to T$, such that the closed point $t'
           \in T'$ is $\kappa$-rational and it is mapped to $t$ and the generic point $\eta' \in T'$ is mapped into $U$.
       
        If in addition $\kappa$ is algebraically closed, we may furthermore
        assume that $A' = \kappa[[\alpha]]$.
   \end{lem}
    \begin{proof}
        First, we may replace $T$ with $\Spec A_1 := \Spec \mathcal{O}_{T, t}$, and $U$ with the preimage under $\Spec A_1 \to T$.
        The ring $A_1$ is Noetherian by our global assumption (Notation~\ref{not_scheme}).
        Let $A$  be the completion of $A_1$ at the maximal ideal $\mathfrak{m}$ of $A_1$.
        Then $A$ is Noetherian and flat over $A_1$, so that
        $\Spec A\to \Spec A_1$ is surjective
        \cite[Tag~0316, Tag~0250, items~(6)\&(7)]{stacks_project}.
        Let any prime ideal $ \gotp \subset A$ mapping to the generic point of $T$.
        Then $\Spec A/\mathfrak{p}$ is integral and satisfies the assertions on $T'$
        except, perhaps, one-dimensionality.
        Replace $T$ with $\Spec A/\mathfrak{p}$ and $U$ with the preimage under $\Spec A/\mathfrak{p} \to T$.

        If $\dim T =0$, then $T = \set{t} = \Spec \kappa$ and $U = T$, and $T'
        = \Spec \kappa[[\alpha]]$ with a morphism $T' \to T$ corresponding to
        $\kappa \to \kappa[[\alpha]]$ will satisfy the claim of the lemma.
        So suppose $\dim T>0$.

        Let $\eta$ be a generic point of $T$.
        If $U = \{\eta\}$, then $T$ is at most one-dimensional by the Theorem of
        Artin-Tate, see~\cite[Cor~B.62]{gortz_wedhorn_algebraic_geometry_I}.
        If not, then we may take an irreducible closed subset $V \subsetneq T$ such that the generic point of $V$ is in $U$ and again replace $T$ with $V$. 
        Since $\dim V < \dim T <\infty$, after a finite number of such replacements we obtain that $\dim T = 1$.
        Thus $T$ is a spectrum of a Noetherian complete local domain with
        quotient field $\kappa$ and we may take the identity $T'=T$ to finish the proof of the first part.

        Suppose now that $\kappa$ is algebraically closed. 
        We may assume that $T = \Spec A$, is as above. 
        Let $\mathfrak{m}$ be the maximal ideal of $A$.
        The normalisation $\tilde{A}$ of $A$ is a finite $A$-module, see
        e.g.~\cite[Appendix 1, Cor 2]{nagata_algebraic_geo_over_Dedekind_two}.
        Then $\tilde{T} = \Spec \tilde{A} \to T$ is finite and dominating, thus it is onto.
        Since $\kappa$ is algebraically closed, any point in the preimage of the
        special point is a $\kappa$-rational point, thus $\tilde{T}\to T$ satisfies claim of the lemma.
        Now $\tilde{A}$ is a one-dimensional normal Noetherian domain which is a finite $A$-module.
        By \cite[Cor 7.6]{eisenbud} the algebra $\tilde{A}$ 
          is a finite product $\tilde{A} = \prod B_i$,
          where each $B_i$ is local and complete, and the residue field of
          each $B_i$ is $\kappa$.
        From the first part of the proof it follows, that
        we may replace $\tilde{A}$ by one of the factors $B_i$, 
        which is a one-dimensional
        Noetherian normal local complete domain with quotient field $\kappa$.
        Thus $B_i$ is regular by Serre's criterion~\cite[Thm 11.5]{eisenbud}, so from
        the Cohen Structure Theorem~\cite[Thm 7.7]{eisenbud}, it follows
        that $B_i$ is isomorphic to $\kappa[[\alpha]]$.	
    \end{proof}

 \begin{example}
    Suppose $\kk=\RR$ and consider the $\RR$-algebra
    $A:=\RR \oplus \alpha \CC[[\alpha]] \subset \tilde{A} := \CC[[\alpha]]$.
    Then the normalisation of $A$ is $\tilde{A}$, and $\Spec A$ has no $\RR$-points. 
    This illustrates that in the proof of the final part of
    Lemma~\ref{lem_finding_a_curve_through_point_and_open_subset} we really
    use the assumption that $\kappa$ is algebraically closed.
 \end{example}

    The following Theorem~\ref{ref:goodbaseofsmoothing:thm} is the key
    result of this section. It allows us to shrink the base of smoothing to an
    algebraic analogue of a small one-dimensional disk.
    \begin{thm}\label{ref:goodbaseofsmoothing:thm}
        Let $Z\to T$ be an abstract or embedded smoothing of some scheme.
        Then, after a~base change, we may assume that $T  \simeq \Spec A$, where $A$
        is a one-dimensional Noetherian complete local domain with quotient field $\kk$.

        If $\kk$ is algebraically closed, we may furthermore
        assume that $A = \kk[[\alpha]]$.
    \end{thm}
    \begin{proof}
        Since $Z\to T$ is finite, the relative differentials sheaf is coherent
        over $T$, so that there exists an open neighbourhood $U$ of the generic point $\eta$
        such that $Z_{u}$ is smooth for any $u\in U$.      
        Thus the claim is a combination of Lemmas~\ref{ref:basechange:lem} and \ref{lem_finding_a_curve_through_point_and_open_subset}.
    \end{proof}

    Now we recall the correspondence between the smoothings of $R$ and of its
    connected components. Intuitively, by Theorem~\ref{ref:goodbaseofsmoothing:thm} we
    may choose such a small basis of the smoothings, that smoothings of connected components are
    connected components of the smoothing.
    \begin{prop}\label{ref:smoothingcomponents:prop}
        Let $R = R_1\sqcup R_2\sqcup  \ldots  \sqcup R_k$ be a~finite scheme.
        If $(Z_i, R_i)\to (T, t)$ are abstract smoothings of $R_i$ over some
        base $T$, then $Z =
        \bigsqcup Z_i\to T$ is an abstract~smoothing of $R$.

        Conversely, let $(Z, R)\to (T, t)$ be an abstract smoothing of $R$
        over $T=\Spec A$, where $A = (A,\mathfrak{m}, \kk)$ is a~local
        complete $\kk$-algebra. Then $Z = Z_1\sqcup  \ldots \sqcup Z_k$, where
        $(Z_i, R_i)\to (T, t)$ is an abstract smoothing of $R_i$.
    \end{prop}
    \begin{proof}
        The first claim is clear, since we may check that $Z = \bigsqcup Z_i$
        is flat and finite locally on \emph{connected components} of $Z$.
        Let $\eta$ be the generic point of $T$, then
        $Z_{\eta} = \bigsqcup (Z_i)_{\eta}$ is smooth over $\eta$ since
        $(Z_i)_{\eta}$ are all smooth.

        \def\nn{\mathfrak{n}}
        \def\mm{\mathfrak{m}}
        For the second part, note that $Z$ is affine by definition of finite
        morphism~\cite[Tag 01WH]{stacks_project}.
        Let $Z = \Spec B$, then
        $B$ is a finite $A$-module and $R = \Spec B/\mm B$.
        Let $\nn_i$ be the maximal ideals in $B$ containing $\mm$. They
        correspond bijectively to maximal ideals of $B/\mm B$ and thus to
        components of $R$. Namely, $R_i = (B/\mm B)_{\nn_i}$ for appropriate indexing of $\nn_i$.
        Since $A$ is a complete Noetherian
        $\kk$-algebra, by~\cite[Thm 7.2a, Cor 7.6]{eisenbud} we get that $B =
        B_{\nn_1} \times  \ldots \times B_{\nn_n}$.
        Then $B_{\nn_i}$ is a flat $A$-module, as a localisation of $B$,
        and also a finite $A$-module, since it may be regarded as a quotient
        of $B$. The fibre of $B_{\nn_i}$ over the generic point of $\Spec A$ is a
        localisation of the fibre of $B$. Therefore $\Spec B_{\nn_i} \to \Spec
        A$ is a smoothing of $\Spec (B_{\nn_i})/\mm B_{\nn_i} = \Spec (B/\mm
        B)_{\nn_i} = R_i$.
    \end{proof}

    \begin{cor}\label{ref:smoothingcomponentsresult:cor}
        Let $R = R_1\sqcup R_2\sqcup  \ldots  \sqcup R_k$ be a~finite scheme.
        Then $R$ is abstractly smoothable if and only if each $R_i$ is
        abstractly smoothable.
    \end{cor}

    \begin{proof}
        If each $R_i$ is abstractly smoothable, then we may choose smoothings
          over the same base $T$, for instance by taking the product of the all bases of the individual smoothings,
          compare with Example~\ref{ex_change_base_by_product}.
        The claim follows from
        Proposition~\ref{ref:smoothingcomponents:prop}.
        
        Conversely, if $R$ is
        smoothable, then we may choose a smoothing over a one-dimensional
        Noetherian complete local domain by
        Theorem~\ref{ref:goodbaseofsmoothing:thm}. Again the result is implied by Proposition~\ref{ref:smoothingcomponents:prop}.
    \end{proof}

\subsection{Comparing abstract and embedded
smoothings}\label{sect_smoothings_abstract_embedded}

    Now we will compare the notion of abstract smoothability and embedded
    smoothability of a scheme $R$. First, we need a technical lemma.
    \begin{lem}\label{ref:closedimmersions:lem}
        Let $(A, \gotm, \kk)$ be a local $\kk$-algebra and $T = \Spec A$ with a
        $\kk$-rational point $t$ corresponding to $\gotm$.
        Let $Z$ be a scheme with a unique closed point and $Z\to T$ be a finite flat morphism.
        Let $Y$ be another scheme and $f:Z\to Y \times T$ be a
        morphism such that the following diagram is commutative:
        \[
          \begin{tikzcd}
             Z \arrow{rr}{}\arrow{dr}{} & & Y \times T\arrow{dl}{}\\
             & T 
          \end{tikzcd}
        \]
        If $f_t:Z_t\to Y$ is a closed
        immersion, then $f$ is also a closed immersion.
    \end{lem}

    \begin{proof}
        Since $Y\times T\to T$ is separated and $Z\to T$ is finite, from the
        cancellation property (\cite[Ex.~II.4.8]{hartshorne}) it follows that
        $f:Z\to Y\times T$ is
        finite, in particular proper, thus the image of $Z$ in $Y\times T$ is closed.
        Then it is enough to prove that $Z\to Y\times T$ is a~locally closed
        immersion.

        Let $U\subseteq Y$ be an open affine neighbourhood of $f_t(p)$, where
        $p$ is the unique
        closed point of $Z$. Since the preimage of $U \times T$ in $Z$ is open
        and contains $p$, the morphism $Z\to Y\times T$ factors through
        $U\times T$.
        We claim that $Z\to U\times T$ is a~closed immersion.
        Note that it is a morphism of affine schemes. 
        Let $B$, $C$ denote the coordinate rings of $Z$ and $U\times T$, respectively.
        Then the morphism of schemes $Z\to Y\times T$ corresponds to a~morphism of $A$-algebras $C\to B$. 
        Since the base change $A\to A/\gotm$
        induces an isomorphism $C/\gotm C\to B/\gotm B$, we have $B = \gotm B + C$,
        thus $C\to B$ is onto by Nakayama Lemma and the fact that $B$ is
        a finite $A$-module. Hence, the morphism $f:Z\to U \times T\to Y\times T$ is
        a~locally closed immersion.
    \end{proof}

    The following Theorem~\ref{ref:abstractvsembedded:thm} together with its
    immediate Corollary~\ref{ref:smoothingeverytwhere:cor} is a generalisation
    of \cite[Lemma~2.2]{casnati_notari_irreducibility_Gorenstein_degree_9} and \cite[Prop.~2.1]{nisiabu_jabu_cactus}. 
    Similar ideas are mentioned in \cite[Lemma~4.1]{cartwright_erman_velasco_viray_Hilb8} and in \cite[p.~4]{artin_deform_of_sings}.
    A sketch of this argument was also communicated to the first named author by V.~Srinivas.
    The theorem uses the notion of \emph{formal smoothness}, see \cite[Def.~17.1.1]{ega4_4}, which we recall briefly.
    A scheme $X$ is formally smooth if and only if for every affine scheme $Y$ and every closed subscheme $Y_0 \subset Y$ defined by a nilpotent ideal of $\ccO_Y$,
    every morphism $Y_0\to X$ extends to a morphism $Y\to X$.
    A scheme is smooth (in the sense Section~\ref{sec_smooth_regular_flat}) 
       if and only if it is locally of finite type and formally smooth, see \cite[17.3]{ega4_4}.

    \begin{thm}[Abstract smoothing versus embedded smoothing]\label{ref:abstractvsembedded:thm}
        Let $R$ be a~finite scheme over $\kk$ which is embedded into a~formally smooth scheme $X$.
        Then $R$ is smoothable in $X$ if and only if it is abstractly smoothable.
    \end{thm}

    \begin{proof}
        Clearly from definition, if $R$ is smoothable in $X$, then it is
        abstractly smoothable. It remains to prove the other implication.

        Let $(Z, R)\to (T, t)$ be an abstract smoothing of $R$. Using
        Theorem~\ref{ref:goodbaseofsmoothing:thm} we may assume that $T$ is
        a~spectrum of a~complete local ring $(A, \gotm, \kk)$. Since $Z\to T$ is
        finite, $Z \simeq \Spec B$, where $B$ is a~finite $A$-algebra. 
        Let us first consider the case when $R$ is irreducible.
        In particular, since $Z$ is irreducible, it is complete  by~\cite[Cor.~7.6]{eisenbud},
          and by~\cite[Thm~7.2a]{eisenbud}, the algebra $B$ is the inverse limit of Artinian $\kk$-algebras $B/(\gotm B)^n$, where $n\in \mathbb{N}$.

        By definition of $X$ being formally smooth, the morphism $R = \Spec
        B/\gotm B \to X$ lifts to $\Spec B/(\gotm B)^2\to X$
          and subsequently $\Spec B/(\gotm B)^n\to X$ lifts to $\Spec B/(\gotm B)^{n+1}\to X$ for every $n \in \mathbb{N}$.
        Since $B$ is complete, together these morphisms give a morphism $Z = \Spec B \to X$, which in
        turn gives rise to a~morphism of $T$-schemes $Z\to X\times
        T$. This morphism is a closed immersion by
        Lemma~\ref{ref:closedimmersions:lem}.
        This finishes the proof in the case of
        irreducible $R$.

        Now consider a not necessarily irreducible $R$. Let $R = R_1 \sqcup
        \ldots  \sqcup R_k$ be the decomposition into irreducible (or
        connected) components. By
        Proposition~\ref{ref:smoothingcomponents:prop}, the smoothing $Z$
        decomposes as $Z = Z_1 \sqcup  \ldots  \sqcup Z_k$, where $(Z_i,
        R_i)\to (T, t)$ are smoothings of $R_i$. The schemes $R_i$ are
        irreducible, so by the previous case, these smoothings give
        rise to embedded smoothings $Z_i \subseteq X\times T$. 
        In particular, each subscheme $Z_i$ is closed.
        Moreover, the images of closed points of $Z_i$ are pairwise different in $X\times T$, 
          thus $Z_i$ are pairwise disjoint and 
          we get an embedding of $Z = Z_1\sqcup  \ldots  \sqcup Z_k \subset X\times T$, which is the required embedded smoothing.
    \end{proof}

    \begin{cor}\label{ref:smoothingeverytwhere:cor}
        Suppose that $R$ is a finite scheme and $X$ and $Y$ are two smooth schemes. 
        If $R$ can be embedded in $X$ and in $Y$, then $R$ is
        smoothable in $X$ if and only if $R$ is smoothable in $Y$.
    \end{cor}
    \begin{proof}
        Follows directly from Theorem~\ref{ref:abstractvsembedded:thm}.
    \end{proof}

    \begin{proof}[Proof of Theorem~\ref{thm_equivalence_of_abstract_and_embedded_smoothings}]
        Corollary~\ref{ref:smoothingcomponentsresult:cor} gives the
        equivalence of \ref{it:Rabssm} and \ref{it:Rconabssm}.
        Smooth variety is formally smooth by definition, so
        Theorem~\ref{ref:abstractvsembedded:thm} implies equivalence of
        \ref{it:Rabssm} and \ref{it:Rembsm} as well as
        \ref{it:Rconabssm} and \ref{it:Rconembsm}.
    \end{proof}

    \subsection{Embedded smoothability depends only on singularity type}

    While the comparison between abstract and embedded smoothings
    given in Theorem~\ref{ref:abstractvsembedded:thm} above is
    satisfactory, it is natural to ask what is true without the assumption of formal smoothness.
    This assumption cannot be removed altogether: for a projective
    curve $C$ all its tangent
    vectors are smoothable in $C$ if and only if all tangent spaces are
    contained in tangent stars,
    see~\cite[Section~3.3]{jabu_ginensky_landsberg_Eisenbuds_conjecture}.
    However we will see that this assumption may be removed entirely if we
    have an appropriate morphism, see
    Corollary~\ref{ref:pushingsmoothings:cor} and that the existence of
    smoothings depends only on the formal geometry of $X$ near the support of
    its subscheme, see
    Proposition~\ref{prop_smoothability_depends_only_on_sing_type}.

    \begin{cor}\label{ref:pushingsmoothings:cor}
        Let $R$ be a finite scheme embedded in a scheme $X$ and smoothable in $X$.
        Let $Y$ be a scheme with a morphism $X\to Y$ which induces an isomorphism of $R$
           with its scheme-theoretic image $S \subseteq Y$. Then $R  \simeq S$
        is smoothable in $Y$.
    \end{cor}
    \begin{proof}
        Let $Z\subseteq X \times T\to T$ be an embedded smoothing of $X$ over
        a base $(T, t)$.  The morphism $X\to Y$ induces a morphism $Z\to Y \times
        T$ which, over $t$, induces a closed embedding $R \subseteq Y$. Then
        we need to prove that $Z\to Y \times T$ is a closed immersion.
        By Theorem~\ref{ref:goodbaseofsmoothing:thm} and
        Proposition~\ref{ref:smoothingcomponents:prop} we may reduce to the
        case when $R$ is irreducible.
        Then the claim follows from Lemma~\ref{ref:closedimmersions:lem}.
    \end{proof}

    Using Corollary~\ref{ref:pushingsmoothings:cor} we may strengthen
    Corollary~\ref{ref:smoothingeverytwhere:cor} a bit, obtaining a direct
    generalisation of \cite[Prop 2.1]{nisiabu_jabu_cactus}.
    \begin{cor}
        Let $X$ be a finite type scheme and $R \subseteq X$ be a finite closed
        subscheme of $X$, supported in the smooth locus of $X$. If $R$ is
        abstractly smoothable, then $R$ is smoothable in $X$.
    \end{cor}

    \begin{proof}
        Let $X^{sm}$ be the smooth locus of $X$. By
        Theorem~\ref{ref:abstractvsembedded:thm} the scheme $R$ is smoothable in
        $X^{sm}$ and by Corollary~\ref{ref:pushingsmoothings:cor} it is also
        smoothable in $X$.
    \end{proof}

    We now show that the possibility of smoothing a given $R$ inside $X$
    depends only on $R$ and the formally local structure of $X$ near $R$. This
    is the strongest result in this direction we could hope for; it implies
    that smoothability depends only on Zariski neighbourhoods of $R$ in $X$
      (and even \'etale local neighbourhoods: if $X'\to X$ is \`etale and $R
      \subset X' \to X$ are two embeddings, then the residue fields at
      corresponding points are isomorphic, hence also formal neighbourhoods
      are isomorphic, see~\cite[Tag~0257, discussion
      after~40.11.3]{stacks_project}).

    \begin{prop}\label{prop_smoothability_depends_only_on_sing_type}
        Let $R \subset X$ be a finite scheme, supported at points $x_1, \ldots
        ,x_k$ of $X$. Then $R$ is smoothable in $X$ if
        and only if $R$ is smoothable in $\bigsqcup \Spec \hat{\ccO}_{X, x_i}$.
    \end{prop}
    The proof of this proposition relies on more terminology than introduced so far.
    See~\cite[\S5.5]{vakil_FoAG} for the definition of \emph{associated points} of a scheme (see also Section~\ref{sec_Castelnuovo_Mumford_reg}), 
    and~\cite[Chapter~10]{atiyah_macdonald} for the $\gotm$-\emph{adic topology} and properties of \emph{Cauchy sequences} in this topology.
    \begin{proof}
        \def\YY{{Z_i}}%
        The ``only if'' part follows from
        Corollary~\ref{ref:pushingsmoothings:cor} applied to the map
        $\bigsqcup\Spec
        \hat{\ccO}_{X, x_i} \to X$. We will now prove the ``if'' part, so we
        assume that $R$ is smoothable in $X$.
        By Theorem~\ref{ref:goodbaseofsmoothing:thm} we may take
        a smoothing of $R$ over $T = \Spec A$ where $(A, \gotm)$ is local and
        complete; this is a family $Z \subset X \times T$ cut out of
        $\ccO_{X} \tensor_{\kk} A$ by an ideal sheaf $\mathcal{I}$.
        By Proposition~\ref{ref:smoothingcomponents:prop} we have $Z =
        \bigsqcup Z_i$ where $Z_i$ is a smoothing of $R_{x_i}$.
        We will show that $Z_i \to X \times T$ can be factorised as follows:
        \begin{equation}\label{eq:trimming}
            Z_i \to \Spec \hat{\ccO}_{X, x_i} \times T \to X \times T.
        \end{equation}
        Fix $i$. Let $x := x_i$ and $\gotn \subset \ccO_X$ be the ideal sheaf
        of $x\in X$. 
        Since $\YY \to T$ is finite, the algebra $H^0(\YY, \ccO_\YY)$ is
        $\gotm$-adically complete. Since $R$ is finite, say of degree $d$, we
        have $\gotn^d \ccO_\YY \subset \gotm \ccO_\YY$.
        This means that each $\gotn \ccO_\YY$-adic Cauchy sequence is also
        an $\gotm \ccO_\YY$-adic Cauchy sequence and hence has a unique limit in
        $\ccO_\YY$.
        Thus the algebra $H^0(\YY, \ccO_\YY)$ is complete in $\gotn \ccO_\YY$-adic topology. By
        universal property of completion, the map $\YY \to X \times T$ factors
        through $\Spec \hat{\ccO}_{X, y} \times T$. The map $\Spec
        \hat{\ccO}_{X, x} \to X$ is separated, hence from~\eqref{eq:trimming}
        it follows that $\YY \to \Spec\hat{\ccO}_{X, x} \times T$ is a closed
        immersion; this gives a deformation $\YY$ embedded into
        $\Spec \hat{\ccO}_{X, x}$, taking the union over all components we get the desired
        embedding $Z \subset T \times \bigsqcup \Spec \hat{\ccO}_{X, x_i}$.
    \end{proof}

    \begin{cor}
        Let $X$ and $Y$ be two schemes and $x\in X$, $y\in Y$ be points with
        isomorphic completions of local rings; let
        $\varphi:\Spec\hat{\ccO}_{X,x}\to \Spec\hat{\ccO}_{Y,y}$
        be an isomorphism.

        Suppose that $R$ is an irreducible finite scheme
        with embeddings $i_X:R \to X$ and $i_Y:R\to Y$ such that $i_X(R)$,
        $i_Y(R)$ are supported at $x$, $y$ respectively. 
        Suppose that $\varphi$ maps isomorphically $i_X(R)$ onto $i_Y(R)$.
        Then $R =i_X(R)$ is smoothable in $X$ if and only if $R=i_Y(R)$ smoothable in $Y$.
    \end{cor}
    \begin{proof}
        By Proposition~\ref{prop_smoothability_depends_only_on_sing_type} the
        scheme $i_X(R)$ is smoothable in $X$ if and only if it is smoothable in
        $\hat{\ccO}_{X,x}$. By the assumptions $i_X(R) \subset \Spec
        \hat{\ccO}_{X, x}$ is isomorphic to $i_Y(R) \subset \Spec
        \hat{\ccO}_{Y, y}$ via $\varphi$. By
        Proposition~\ref{prop_smoothability_depends_only_on_sing_type} again,
        $i_Y(R)$ is smoothable in $\hat{\ccO}_{Y, y}$ if and only if it is
        smoothable in $Y$.
    \end{proof}

\subsection{Subschemes of smoothable schemes}\label{sec_intersect_gives_nonsmoothable}

    Proposition~\ref{ref:smooth_subvar_and_pushing:prop} proves that if $Y
    \subset X$ are smooth varieties and $R\subset X$ is smoothable, then
    $R\cap Y$ can be embedded into a smoothable subscheme of degree at most $\deg
    R$. This is an algebraic version of
    \cite[Lemma~2.8]{jabu_ginensky_landsberg_Eisenbuds_conjecture}. Note that
    $R\cap Y$ needs not to be smoothable, see
    Corollary~\ref{cor_nonsmoothable_intersection_of_smoothable_and_smooth}.

    \begin{lem}\label{ref:bubulem:lem}
        Let $X$ be a scheme and $T = \Spec \kk[[\alpha]]$ with $t\in T$ the closed point corresponding to the ideal $(\alpha)$.
        Suppose that $(Z, R) \subseteq X \times (T, t)$ is an embedded
        smoothing of a finite scheme $R$. 
        Let $Y$ be a  scheme with a quasi-compact
        morphism $f:X\to Y$. Suppose that $f$ maps different points from support of
        $R$ to different points.

        Let $Z'$ be the
        scheme-theoretic image of $Z$ under $f\times \id_T$. Then $(Z', R') \subseteq Y \times (T,
        t)$ is an embedded smoothing of a finite scheme $R'$, such that $\deg R'
        \leq \deg R$. Furthermore, the
        scheme-theoretic image of $R$ in $Y$ is contained in $R'$ as a closed
        subscheme.
    \end{lem}
    \begin{proof} 
        Since $Z\to T$ is proper and $Y \times T\to T$ is
        separated, we see that $Z\to Y\times T$ is proper.
        Every associated point of $Z'$ (see~\cite[\S5.5]{vakil_FoAG} for
        definition) is an image of an associated point of $Z$
        by~\cite[8.3B]{vakil_FoAG}, thus maps to
        the generic point of $T$.
        Therefore $Z'\to T$ is flat, see~\cite[Exercise 24.4.J ``Criterion for
        flatness over a regular curve'']{vakil_FoAG} or \cite[Prop.~III.9.7]{hartshorne}.

        Since $\kk[[\alpha]]$ is complete, using Proposition~\ref{ref:smoothingcomponents:prop} 
          and the assumption on $f$ we may reduce to the case of $R$ irreducible.
        We argue similarly as in the proof of Theorem~\ref{ref:abstractvsembedded:thm}.
        The scheme $R$ has a unique closed point, so also $Z$ and $Z'$ have unique closed points.
        Let $p \times t$ be the closed point of $Z'$, and pick an affine open $U \subset Y$ containing $p$.
        Then $U\times T$ is open affine subset of $Y \times T$ which contains $Z'$:
          indeed, the complement $Z'\cap ((Y \setminus  U)
          \times T)$ is finite over $T$ and contains no closed points, hence
          is empty.
        Thus $Z\to Y \times T$ factors through $U\times T$ and $Z'$ is contained in $U\times T$.

        Let $Z = \Spec B$ and $Z' = \Spec C$, so that $C\to B$ is an injection.
        Note that $C$ is a $T$-submodule of $B$.  Since $T$ is local, $C$ is a
        free $T$-module of rank not higher than $B$. Then 
        \[
           \deg R' =\deg Z'_t = \rk_k C/\alpha C = \rk_T C \leq \rk_T B = \rk_k B/\alpha B = \deg Z_t = \deg R.
        \]
        For the last claim of the lemma, note that we get a morphism $R = Z_t \to Z_{t}'$,
        which by definition factors through the scheme-theoretic image of $R$.
    \end{proof}

    \begin{prop}\label{ref:smooth_subvar_and_pushing:prop}
        Let $\kk$ be algebraically closed.
        Let $X$ be a smooth subvariety of a smooth variety $Y$ and $R$ be a
        finite smoothable closed subscheme of $Y$ supported in $X$. Then there exists
        a finite scheme $Q\subseteq X$ containing $R\cap X$ and smoothable in $X$,
        Such that $\deg Q \le  \deg R$.
        If $\dim X >0$, we may assume $\deg Q = \deg R$.
    \end{prop}

    \begin{proof}
        \def\locY{\hat{\mathcal{O}}_{Y, x}}
        \def\locX{\hat{\mathcal{O}}_{X, x}}
        Suppose that $R$ is irreducible, supported at $x\in X$.
        Since $R$ is smoothable, it is also smoothable in
        $\Spec \locY$, because it is a power series ring, thus it is formally
        smooth (Theorem~\ref{ref:abstractvsembedded:thm}). Moreover by~Theorem~\ref{ref:goodbaseofsmoothing:thm} we may
        assume that the base of the smoothing is equal to $\Spec\kk[[\alpha]]$. Now
        $\locX$ is another power series ring with an
        epimorphism $i^*\colon \locY \onto \locX$. This epimorphism has a section $\pi^*:\locX\into \locY$ 
        (coming from any splitting of the cotangent bundle), 
        which gives a morphism $\pi:\Spec \locY \to \Spec \locX$. Applying
        Lemma~\ref{ref:bubulem:lem} to $\pi$, we obtain a smoothing of
        a finite scheme $Q \subseteq \Spec \locX$. Moreover $\deg Q\leq \deg R$ and $Q$
        contains the scheme-theoretic image $S$ of $R$ under $\pi$.

        If $R$ is cut out by an ideal $I\subseteq \locY$, then $R\cap X$ is
        defined by $i^*(I)$ and $S$ is defined by $(\pi^*)^{-1}(I)$.
        Since $\pi^*$ is a section of
        $i^*$, we see that $(i^*)^{-1}(I) \subseteq \pi^*(I)$, so that $R\cap X
        \subseteq S\subseteq Q$.
        Finally, Corollary~\ref{ref:pushingsmoothings:cor} implies that
        $Q$ is smoothable in $X$.  For the general case, we consider the
        irreducible components of $R$.

        The obtained $Q$ satisfies $\deg Q\leq \deg R$. Adding points to $Q$,
        we obtain the equality $\deg Q = \deg R$.
    \end{proof}
    
    It is a natural question, if the scheme $R\cap X$ in the setting of Proposition~\ref{ref:smooth_subvar_and_pushing:prop}
      is necessarily smoothable. We answer this question in negative in Corollary~\ref{cor_nonsmoothable_intersection_of_smoothable_and_smooth}.
    In fact, any finite scheme $R'$ can be realised as intersection of a smoothable scheme and a smooth variety.

We construct an example of a smooth variety $Y$, a smoothable scheme $R \subset Y$ and a smooth variety $X \subset Y$,
such that $R \cap X$ is not smoothable.
This construction is based on the unpublished note \cite{jabu_example_to_BGL},
   where it illustrates a delicacy in the course of the proof of
   \cite[Theorem~1.1]{jabu_ginensky_landsberg_Eisenbuds_conjecture}.

   First, we show that for any finite scheme (perhaps non-smoothable)
   there exists a finite smoothable superscheme.
\begin{lem}\label{lem_exists_smoothable_superscheme}
   Let $R'$ be any finite $\kk$-scheme. Then there exists a finite smoothable
   $\kk$-scheme $R$, and an inclusion $R' \subset R$.
\end{lem}
\begin{proof}
   Let $A = H^0(R, \mathcal{O}_R)$, then $A$ has finite $\kk$-dimension and
   can be written as a quotient $A = \kk[\alpha_1, \ldots ,\alpha_n]/I$, see~Section~\ref{sect_finite_schemes}.
   For any fixed $i$, the powers $\{\alpha_i^j\ |\ j\geq 0\}$ are
   linearly dependent over $\kk$: there exist polynomials $F_i\in \kk[\alpha_i]$
   such that $F_i\in I$.

   Take the ideal $J = (F_1, \ldots , F_n)$.
   Every power of $\alpha_i$ can be reduced modulo $J$ to $\alpha_i^j$ with $j < \deg
   F_i$. Thus every monomial can be reduced modulo $J$ to $\alpha_1^{j_1} \ldots
   \alpha_n^{j_n}$ with $j_i < \deg F_i$ for all $i$. Therefore $\kk[\alpha_1, \ldots
   ,\alpha_n]/J$ is finite; we choose $R' = V(J)$.
   The proof of its smoothability is a special case of the argument of
   Example~\ref{ex:fieldExtensions}.
\end{proof}

It is an interesting problem, how to find $R$ as in Lemma~\ref{lem_exists_smoothable_superscheme} 
    which has a minimal possible degree and what is this minimal degree.
This question is related to a topological problem of constructions of $k$-regular maps
  (see the final paragraph of \cite[Sect.~5.2]{jabu_januszkiewicz_jelisiejew_michalek_k_reg}), 
  and to the problem of finding the border rank of multiplication tensor in finite commutative algebras 
  \cite{balser_lysikov_on_degs_of_tensors_and_algebras}.

Second, we show, that any embedding of finite schemes can be extended to an embedding of smooth varieties in a compatible way.

\begin{prop}\label{prop_Rprime_subset_R_can_be_extended_to_X_subset_Y}
   Let $R'$ be any finite scheme and $R$ be any finite scheme containing $R'$. 
   Then there exists a smooth variety $Y$, a smooth subvariety $X \subset Y$, and an embedding $R\subset Y$, 
      such that $X \cap R = R'$.
   Moreover, we can assume $\dim X$ is not more than the embedding dimension
   of $R$ (see~Section~\ref{sect_finite_schemes} for definition), and $X$ and $Y$ are affine spaces.
\end{prop}
\begin{proof}
Let $R = \Spec\kk[\fromto{\alpha_1}{\alpha_n}]/I$, where $I=I(R)$ is an ideal defining $R$ in the affine space $\AA_{\kk}^n$.
Without loss of generality, $n$ is equal to the embedding dimension of $R$.
Let polynomials $f_1,\dotsc, f_m\in \kk[\fromto{\alpha_1}{\alpha_n}]$ be such that the ideal of $R'$
   is generated by $I(R)$ together with $f_1,\dotsc, f_m$.
In the affine space 
\[
  Y:=\AA_{\kk}^n \times \AA_{\kk}^m \simeq \AA_{\kk}^{n+m} = \Spec\kk[\fromto{\alpha_1}{\alpha_n}, \fromto{\beta_1}{\beta_m}] 
\]
  consider the subscheme $X$ defined by $(f_1-\beta_1, \dotsc, f_m-\beta_m)$,
  where $f_i$ are the polynomials as above on the $\AA_{\kk}^n$ factor, whereas $\beta_1, \dotsc, \beta_m$
  are the coordinates on the $\AA_{\kk}^m$ factor.
That is, $X$ is a graph of the map $\AA_{\kk}^n \to \AA_{\kk}^m$ given by $(f_1, \dotsc, f_m)$.
In particular, $X$ and $Y$ are smooth varieties of dimensions $n$ and $m$ respectively.
Note also that $X\cap (\AA_{\kk}^n \times \set{0})$ is defined in
$\AA_{\kk}^n$ by $(f_1, \ldots ,f_n)$. Therefore, $X \cap R = R'$.
\end{proof}

\goodbreak
Finally, we prove the main counter-example of this subsection.

\begin{cor}\label{cor_nonsmoothable_intersection_of_smoothable_and_smooth}
   There exists a finite smoothable scheme $R$ embedded in a smooth variety $Y$ and a smooth subvariety $X$,
     such that $X \cap R$ is not smoothable. 
   Moreover, the intersection $X\cap R$ can be chosen to be isomorphic to any nonsmoothable finite scheme.
\end{cor}
\begin{proof}
   Suppose $R'$ is any nonsmoothable finite scheme over $\kk$, for example as
   in Example~\ref{ref:nonsmoothableOfHighDegree:ex}.
   Construct a smoothable scheme $R$ containing $R'$ as in Lemma~\ref{lem_exists_smoothable_superscheme}.
   Then find smooth varieties $X$, $Y$, such that $X \subset Y$, $R \subset Y$ and $R' = X \cap R$ 
      as in Proposition~\ref{prop_Rprime_subset_R_can_be_extended_to_X_subset_Y}.
\end{proof}

Suppose  $X$, $Y$, $R$ are constructed by the methods presented in
Lemma~\ref{lem_exists_smoothable_superscheme}, 
   Proposition~\ref{prop_Rprime_subset_R_can_be_extended_to_X_subset_Y} 
   and Corollary~\ref{cor_nonsmoothable_intersection_of_smoothable_and_smooth}.
Further let $Q \subset X$ be the finite smoothable subscheme 
  containing $X\cap R$ constructed in Proposition~\ref{ref:smooth_subvar_and_pushing:prop}.
Then $R$ is (abstractly) isomorphic to $Q$.
However it is just a feature of these proofs, and it is not always the case.
That is, in general $Q$ needs not to be (abstractly) isomorphic to $R$.
%

\section{Projective schemes and embedded geometry} \label{sec_projective_schemes_and_embedded_geometry} 

As explained in Section~\ref{sec_preliminaries_and_finite_schemes}, the definition of a scheme is local in nature:
we describe the affine schemes which are then glued to a global object.
This is very useful for studying local properties of schemes as illustrated in Section~\ref{sec:smoothability}.
However, global properties are often difficult to tackle using local affine
geometry. Below we switch to projective geometry.

Projective varieties and schemes are of uttermost importance to algebraic geometry.
They are analogues of compact manifolds in differential geometry, 
  and indeed, many famous examples of holomorphic compact manifolds can be realised as smooth complex projective varieties.
In this section we review some basics of projective geometry: Veronese map, apolarity, homogeneous ideals and corresponding schemes, 
  linear span of schemes and their Veronese reembeddings, Castelnuovo-Mumford regularity.

\subsection{Vector spaces and duality of symmetric powers}\label{sec_vector_spaces_duality}\label{sec_apolarity}\label{sec_Veronese_map}

The purpose of this subsection is to present Macaulay's duality and the Veronese map in a choice free way.
Macaulay's duality relates two polynomial rings $\kk[\fromto{x_1}{x_n}]$ and $\kk[\fromto{\alpha_1}{\alpha_n}]$
  with an action of the latter one on the first one denoted by $\hook$: 
  if $F= x_1^{a_1}x_2^{a_2}\dotsm x_n^{a_n}$, then
  \[
     \alpha_i\hook F = \begin{cases}
                         x_i^{-1}F & \text{if $a_i >0$}\\
                          0 & \text{if $a_i =0$}
                       \end{cases}.
  \]
  The action extends linearly to an action of $\alpha_i$ on
  whole $\kk[\fromto{x_1}{x_n}]$, hence to an action of
  $\kk[\fromto{\alpha_1}{\alpha_n}]$ on $\kk[\fromto{x_1}{x_n}]$.
    Similarly, the $d$-th Veronese map is defined
    by all monomials of degree $d$:
  \begin{equation}\label{equ_define_Veronese}
    (\fromto{a_1}{a_n}) \mapsto(\fromto{a_1^d, a_1^{d-1}a_2}{a_n^d}).
  \end{equation}
  A serious drawback of those definitions is that they depend on the choice of
  bases $\fromto{x_1}{x_n}$ and $\fromto{\alpha_1}{\alpha_n}$. 
  In characteristic zero there are
  alternative definitions which are ``coordinate-free''.

  \begin{remark}
      If $\cchar \kk = 0$ then the Macaulay's duality is, up to change of
      basis, equivalent to the action $\alpha_i\hook F =
      \frac{\partial}{\partial x_i} F$. The Veronese map may be similarly
      defined by $\ell \mapsto \ell^d$.
  \end{remark}

  Here we present equivalent coordinate free definitions in all characteristic, 
    thus proving that the above drawback is only apparent. 
  Our presentation below is based on~\cite[Section~4.2]{jelisiejew_MSc} and mostly on \cite[Appendix]{eisenbud}.

Denote by $\GG_m$ the algebraic group, which as a variety is $\GG_m = \Spec \kk[\alpha, \alpha^{-1}] = \AA^1_{\kk} \setminus \set{0}$.
Thus, informally, $\GG_m$ might be thought of as the group of invertible elements of the base  field $\kk$. 
In particular, the group structure comes from the multiplication and inverse in the field.

Let $V^{alg}$ be a finite dimensional \emph{algebraic vector space} over $\kk$.
More precisely, $V^{alg}$ is an affine space with the additive group structure 
  $\Add\colon V^{alg}\times V^{alg} \to V^{alg}$ and 
  a rescaling action $\GG_m \times V^{alg} \to V^{alg}$ satisfying the usual vector space axioms.
The set of closed $\kk$-rational points of $V^{alg}$ is a $\kk$-vector space in the usual sense, which we denote by $V$.
Then
\[
   V^{alg}= \Spec S^{\bullet} V^* = \Spec \bigoplus_{i=0}^{\infty} S^{i} V^* \simeq \kk[\fromto{\alpha_1}{\alpha_n}]
\]
  where $\fromto{\alpha_1}{\alpha_n}$ is a basis of $V^*$ and $S^{\bullet} V^*$ is the symmetric algebra.
In particular, $V^{alg}$ is uniquely determined by $V$ and vice-versa, 
  and any linear automorphism of $V$ naturally determines an algebraic automorphism of $V^{alg}$ preserving its structure 
    (additivity and rescalings) and vice-versa.

The additive structure of $V^{alg}$ is given by a Hopf algebra structure of $S^{\bullet} V^*$
   with usual multiplication and unity, comultiplication induced by $\Add^*(v) = 1\tensor v +
v\tensor 1$, counity induced from the zero map $V^*\to \kk$ and antipode $S(v) = -v$.
Note that all these maps are homogeneous with respect to the natural grading
on $S^{\bullet} V^*$.
Consider the space of \emph{divded powers} of $V$:
\[
    \DPV{\bullet}  := \bigoplus_{d\geq 0} \left(S^{d} V^*\right)^*.
\]
Note that $\DPV{0} = \kk$ and $\DPV{1} = V$, thus we will regard $V$ as subspace of $\DPV{\bullet}$.
We have an action of $S^{\bullet} V^*$ on
$\DPV{\bullet}$ by precomposition of functionals. 
Namely, for $\sigma\in S^{i} V^*$ and $f\in \DPV{j}$ we get $\sigma\hook f\in \DPV{j-i}$ defined by
\[
    (\sigma\hook f)(\tau) = f(\sigma\tau)
\]
for all $\tau\in S^{j-i} V^*$.
We call it the \emph{contraction action} of $S^{\bullet} V^*$ on
$\DPV{\bullet}$.
Since $\DPV{\bullet}$ is the space of graded functionals on a Hopf algebra it is a Hopf
algebra as well. In particular, it has a multiplication, which is closely
related to the Veronese map, which we define below.

The inclusion $S^d V^* \subset S^{\bullet} V^*$ induces a
homomorphism $\hat{\nu}_d^{*}:S^{\bullet} (S^d V^*) \to S^{\bullet} V^*$ of graded
algebras, which in turn gives a morphism $\hat{\nu}_d$ from $\Spec S^{\bullet} V^* = V^{alg}$ to $\Spec S^{\bullet} (S^d V^*) = (\DPV{d})^{alg}$.
This map is called \emph{Veronese map}, and it is usually considered in the projective setting (see Section~\ref{sec_embedded_proj_geom}).

From the explicit point-wise perspective, if $v \in V$ is any vector and $v^{alg} \in V^{alg}$ is a corresponding $\kk$-rational point, 
   then   the image under $\hat{\nu}_d\colon V^{alg}\to (\DPV{d})^{alg}$ of
   $v^{alg}$ is the closed $\kk$-rational point corresponding to $\hat{\nu}_d(v) \in \DPV{d}$, which is the evaluation at $v$  in the following sense.
\[
  \DPV{d} = \left(S^{d} V^*\right)^*  \ni \hat{\nu}_d(v)  \colon S^d V^* \to \kk, \text{ where } \hat{\nu}_d(v) (\Theta) = \Theta(v).
\]
The map $\hat{\nu}_d$ descends to a projective map 
$\nu_d:\PP V\to \PP(\DPV{d})$, 
which also has an explicit description. Namely, suppose $W \subset V$ is a
one-dimensional space. Then $W^{\perp} \subset V^*$ is of codimension one,
thus the space $W^{\perp} \cdot S^{d-1} V^* \subset S^d V^*$ is also of
codimension one. Hence its perpendicular space $W' \subset (S^d V^*)^* =
\DPV{d}$ is one dimensional and indeed $\nu_d(W) = W'$.

The above definitions are natural, we never used any coordinates.
Now suppose a basis of $V$ is chosen to be $\fromto{x_1}{x_n}$, and $\fromto{\alpha_1}{\alpha_n}$ is the dual basis of $V^*$.
The space $S^d V^*$ has a natural basis of monomials $\alpha_1^{p_1}\dotsm\alpha_n^{p_n}$ (here the sum of the nonnegative integers $p_i$ is $d$),
   and a dual basis of $\DPV{d}$ is 
  \[
    x_1^{(d)}, x_1^{(d-1)}x_2,  \dotsc, x_1^{(p_1)}x_2^{(p_2)} \dotsm x_n^{(p_n)}, \dotsc, x_n^{(d)},
  \]
  where the \emph{divided power monomials} $x_1^{(p_1)}x_2^{(p_2)}\dotsm x_n^{(p_n)}$ take value $1$ on $\alpha_1^{p_1}\dotsm\alpha_n^{p_n}$ and $0$ on all the other monomials.
Note that this is not the same as the monomial basis of $S^dV$, particularly in a positive characteristic.
Consider a closed $\kk$-rational point $a_1x_1+\dotsb + a_n x_n \in V \subset V^{alg}$. 
The Veronese map in these coordinates is defined as:
\[
  a_1x_1+\dotsb + a_n x_n \stackrel{\hat{\nu}_d}{\longmapsto}  \sum a_1^{p_1}\dotsm a_n^{p_n} \cdot x_1^{(p_1)}\dotsm x_n^{(p_n)},
\]
which is the same as suggested in Equation~\eqref{equ_define_Veronese}.

If $g \colon V \to V$ is a linear automorphism, 
  then by the naturality of construction it induces a linear automorphism $g^{(d)}$ of $\DPV{d}$.
Explicitly in coordinates, the automorphism $g^{(d)}$ takes the divided power monomial $x_1^{(p_1)}x_2^{(p_2)}\dotsm x_n^{(p_n)}$ 
  to $\nu_{p_1}(g x_1) \nu_{p_2}(g x_2) \dotsm \nu_{p_n}(g x_n)$.
In particular, the Veronese map is equivariant under these two actions of $\GL(V)$ and it preserves the rescalings.

\subsection{Projective schemes and varieties}

The $n$-dimensional projective space $\PP_{\kk}^N$ is obtained by glueing of $N+1$ copies of affine spaces $\AA_{\kk}^N$.
Alternatively, it is the quotient $\pi\colon \AA_{\kk}^{N+1} \setminus \set{0} \to \PP_{\kk}^N$  by the action of the algebraic group $\GG_m$.
The description we are going to recall here is the algebraic version
of this global description using homogeneous ideals.

Let $V$ be an $(N+1)$-dimensional vector space over $\kk$.
By $\PP V$, or $\PP_{\kk}^N$, we denote the scheme consisting of the set of
homogeneous prime ideals in
\[
   \Sym V^*:=\bigoplus_{i=0}^{\infty} S^{i} V^*  = \kk \oplus V^* \oplus S^2 V^* \oplus \dotsb \simeq \kk[\fromto{\alpha_0}{\alpha_N}],
\]
which are contained in, but not equal to, the ideal $\bigoplus_{i\geq 1} S^i V^*$.
The Zariski topology on $\PP V$ is defined in the same way as in the affine
case, except we restrict attention to homogeneous ideals: a subset $X \subset
\PP V$ is closed if and only if there is a homogeneous ideal $I \subset \Sym
V^*$, such that $X$ is the set of homogeneous prime ideals containing $I$.
The projective space has distinguished open affine subsets, each of them is
the complement $U_{(f)}$ of the closed subset defined by a homogeneous ideal
$(f)$ for some nonzero $f \in S^k V^*$ and the global functions on $U_{(f)}$ are
  \[
    \ccO_{\PP V}(U_{(f)}) = (\Sym V^*)_{(f)} = \set{\frac{g}{f^{i}} \mid g\in S^{ik}V^*, i \in \ZZ_{\ge 0}}.
  \]
In particular, such $U_{(f)}$ is naturally isomorphic to $\Spec (\Sym V^*)_{(f)}$.

Closed subschemes of $\PP V$ are defined by homogeneous ideals in $\Sym V^*$. 
That is, for a homogeneous ideal $I \subset \Sym V^*$,
the  subscheme $X \subset \PP V$ defined by $I$ is locally on each
$U_{(f)}$ defined by $I_{(f)} = \set{\frac{g}{f^{i}} \mid g\in S^{ik}V^* \cap I} \subset  (\Sym V^*)_{(f)}$.
Conversely,
  given a closed subscheme $X$ of $\PP V$, the \emph{(saturated) ideal} of $X$ is the largest homogeneous ideal $I(X) \subset \Sym V^*$, 
  such that the scheme defined by $I(X)$ is $X$.
The ideal $I(X)$ consists of all the polynomial functions on $V$ that vanish identically on the affine cone over $X$, that is on the scheme $\hat{X} := \pi^{-1} (X) \subset \AA_{\kk}^{N+1}$.

A scheme (or a variety) $X$ is \emph{projective}, if there \emph{exists} an embedding $X \hookrightarrow \PP_{\kk}^N$ making $X$ a closed subscheme of $\PP^N_{\kk}$.
We stress that the embedding is far from unique!
In the literature, sometimes the term projective scheme or variety refer to a situation,
   where the embedding $X \hookrightarrow \PP V$ is chosen and fixed. 
We discuss this in more details in Section~\ref{sec_embedded_proj_geom}.

\begin{lem}\label{lem_finite_is_projective}
  Suppose $R$ is a finite scheme. Then $R$ is projective.
\end{lem}
\begin{proof}
    See, for example \cite[17.3.C,~pg.~451]{vakil_FoAG}
    or~\cite[Corollary~13.77]{gortz_wedhorn_algebraic_geometry_I}.
\end{proof}

\begin{lem}
  If $Y$ is any scheme, $X \subset Y$ is a subscheme, and $X$ is a projective
  scheme, then $X$ is a closed subscheme of $Y$.
\end{lem}
\begin{proof}
    Subject to Notation~\ref{not_scheme}, $Y$ is separated.
    The claim follows for example from~\cite[Proposition~10.3.4e]{vakil_FoAG},
    where we should take $Z = \Spec \kk$ so that $X$ is proper over $Z$.
\end{proof}

\subsection{Embedded projective geometry}\label{sec_embedded_proj_geom}

In this subsection we discuss several notions related to embeddings of projective varieties into projective spaces.
Particularly, we are interested in the notions of Veronese embedding, linear span and how does linear span behave with respect to Veronese reembeddings.

Recall from Subsection~\ref{sec_Veronese_map} the Veronese map $\hat{\nu}_d\colon V^{alg} \to (\DPV{d})^{alg}$. 
Since it preserves the dilation action of $\GG_m$, it descents to an algebraic morphism of projective spaces $\nu_d \colon\PP V \to \PP(\DPV{d})$.
In the projective setting, the map $\nu_d$ is an embedding, thus we call it
the \emph{Veronese embedding}.

A \emph{linear subspace} of $\PP V$ is a subscheme defined by an ideal $I$ generated by linear functions in $V^*$.
Equivalently, it is subscheme, such that its affine cone is a linear subspace of $V^{alg}$.
In particular, a subscheme $X \subset \PP V$ is a linear subspace if and only if there exists a linear subspace $W \subset V$, 
  such that $X = \PP W$ and the ideal of $X$ is $\ker (V^* \to W^*)$.

From now on, throughout this subsection, we assume $X$ is a projective scheme with a fixed embedding $X \subset \PP V$,
  i.e.~$X$ is a closed subscheme of $\PP V$. 
Let $I(X) = \bigoplus_{i=0}^{\infty} I(X)_i \subset \Sym V^*$ be the homogeneous ideal of $X$.
The function assigning to an integer $i$ the codimension of $I(X)_i \subset S^i V^*$ is called the \emph{Hilbert function} of 
  $X \subset \PP V$ and denoted:
\[
   H_{X\subset \PP V}(i) := \codim \left(I(X)_i \subset S^i V^*\right) = \dim \left( S^i V^*/ I(X)_i \right).
\]

The \emph{linear span} $\langle X \rangle$ of $X$ is the smallest linear
subspace of $\PP V$ that contains $X$.
Equivalently, $\langle X \rangle$ is the scheme defined by the ideal generated by $I(X)_1 = I(X) \cap V^*$.
We have $\dim \langle X \rangle = H_{X \subset \PP V}(1) -1$, the modification by $-1$ comes from the projectivisation.

Consider also the Veronese reembedding of $X$. 
That is, fix a positive integer $d$, and consider $\nu_d(X) \subset \PP( \DPV{d})$.
The scheme $\nu_d(X)$ is isomorphic to $X$, but its embedding into a projective space is different,
  in particular, its ideal $I(\nu_d(X))$ is different, in fact, it is an ideal in a different ring.
More precisely, we have the following composition of ring homomorphisms:
  \[
        \Sym(S^d V^*)  = \bigoplus_{k=0}^{\infty} S^{k} (S^d V^*)
               \twoheadrightarrow  \bigoplus_{k=0}^{\infty} S^{kd} V^*
               \hookrightarrow  \bigoplus_{k=0}^{\infty} S^{d} V^*
                   =   \Sym(V^*)
               \twoheadrightarrow   \Sym(V^*) / I(X).
  \]
The kernel of the composed map is equal to $I(\nu_d(X))$.
In particular, 
\[
  \begin{array}{ccc}
             I(\nu_d(X))_1 & = & I(X)_d \\
               \cap && \cap\\
             S^1(S^d V^*) &=& S^d V^*
  \end{array}
\]
and thus the linear span $\langle \nu_d(X)\rangle$ is determined by the degree $d$ equations of $X$ 
and $\dim \langle \nu_d(X) \rangle = \dim \left(\Sym V^*/I(X)\right)_d -1 = H_{X \subset \PP V}(d) -1$.

From the point of view of sheaf theory, the algebra $\Sym V^*/I(X)$ is determined by global sections of twisted sheaves.
Specifically, denote by $\ccO_{\PP V}(-1)$ the sheaf of sections of the
tautological line bundle;
  informally, the line bundle is a subset of $\PP V \times V$, determined by $\set{(x, v) : x \in v}$.
Then denote by $\ccO_{\PP V}(1):= \ccO_{\PP V}(-1)^*$ the sheaf of sections of a dual line bundle, 
  and for a positive integer $d$:
\[
   \ccO_{\PP V}(d) :=  \ccO_{\PP V}(1)^{\otimes d},  \text{ and } \ccO_{\PP
   V}(-d) :=  \ccO_{\PP V}(-1)^{\otimes d}, \text{ and, for consistence, }\ccO_{\PP V}(0) :=  \ccO_{\PP V}.
\]
As a consequence, for an open affine subset $U_{(f)}$ given by $f \in S^{k} V^*$ (for $k>0$, $f \ne 0$) we have:
  \[
    \ccO_{\PP V}(d)\left(U_{(f)}\right) = \set{\frac{g}{f^i} \mid g\in S^{ik + d} V^*, i \in \ZZ}.
  \]
Glueing together the local sections to global sections we have:
\[
    H^0(\ccO_{\PP V}(d)) = S^{d} V^* \text{ (where $S^d V^* = 0$ if $d < 0$)}.
\]

Let $\ccI_X \subset \ccO_{\PP V}$ be the ideal sheaf of $X$.
Denote
  \[
    \ccI_X(d) := \ccI_X \otimes \ccO_{\PP V}(d), \text{ so that }
    \ccI_{X}(d)\left(U_{(f)}\right) = \set{\frac{g}{f^i} \mid g\in I(X)_{ik + d}, i \in \ZZ},
  \]
  and $\ccO_X(d) := \ccO_{\PP V}(d) /  \ccI_{X}(d) = \ccO_X \otimes \ccO_{\PP V}(d)$. 
Now, from the definitions and the short exact sequence of sheaves $ 0 \to \ccI_{X}(d)  \to \ccO_{\PP V}(d) \to \ccO_X(d) \to 0 $ we conclude:
\begin{align}
 S^d V^*/I(X)_d &=  H^0\left(\ccO_{\PP V}(d)\right) / H^0\left(\ccI_{X}(d)\right) \text{ and }\nonumber \\
   \dim \langle \nu_d(X) \rangle +1 = H_{X \subset \PP V}(d)&= \dim H^0 \left(\ccO_X(d)\right)  - \dim H^1 \left( \ccI_X(d)\right).     \label{eq:dimensionofspan}
\end{align}
In particular, the linear span of Veronese reembedding is strongly connected to cohomological properties of $\ccI_X$ and related sheaves.

\begin{example}\label{exam_Hilbert_function_of_subscheme_drops_by_at_most_r_minus_r_prime}
   Suppose $R \subset \PP V$ is a finite subscheme of degree $r$ and $R' \subset R$ is a subscheme of degree $r'$.
   Then $0 \le H_{R \subset \PP V}(d)- H_{R' \subset \PP V}(d) \le r-r'$.
\end{example}
\begin{proof}
   We have $H^0(\ccO_R(d)) \simeq H^0(\ccO_R)\simeq \CC^r$ and  $H^0(\ccO_{R'}(d)) \simeq H^0(\ccO_{R'})\simeq \CC^{r'}$.
   The following diagram is commutative:
   \[
     \begin{tikzcd}
        S^d V^*/I(R)_d \arrow{r}{}\arrow{d}{}& S^d V^*/I(R')_d \arrow{d}{} \\ 
        H^0(\ccO_R(d)) \arrow{r} &  H^0(\ccO_{R'}(d)).
     \end{tikzcd}
   \]
   The horizontal maps are surjective while the vertical are injective.
   Hence the induced map of kernels of horizontal maps is also injective.
   The kernel of the lower map has dimension $r-r'$, thus the kernel of the upper map has dimension at most $r-r'$.
\end{proof}

We say that the  embedding of $X \subset \PP V$ is \emph{linearly
normal}, if $H^1 \left( \ccI_X(1)\right) =0$ or equivalently, $\dim \sspan{X} = \dim H^0 \left(\ccO_X(1)\right) -1$.
So by above discussion, the Veronese reembedding $\nu_d(X)$ is linearly normal if and only if $H^1 \left( \ccI_X(d)\right)=0$,
  or $\dim \sspan{\nu_d(X)} = \dim H^0 \left(\ccO_X(d)\right) -1$.
Linearly normal embeddings share some properties of linearly independent subsets from linear algebra:
\begin{lem}\label{lem_properties_of_lin_indep_emb}
  Fix a projective space $\PP V$ and its closed subschemes $R, X, Y  \subset \PP V$. 
  \begin{enumerate}
      \item\label{it:linindeppoints} If $R$ is a finite union of distinct reduced $\kk$-rational points in $\PP V$, then the embedding of $R$ 
           is linearly normal if and only if the $\kk$-rational points lifted to $V$ are linearly independent in $V$.
       \item\label{it:linindepintersect} \emph{(Grassmann lemma)} If $X$ and $Y$ are subschemes of $\PP V$, 
           and the embeddings of $X\cap Y$, $X$, $Y$, and $X \cup Y$ into $\PP
           V$  are  linearly normal, then:
         \[
            \sspan{X} \cap \sspan{Y} = \sspan{X \cap Y}.
         \]
    \item \label{item_R_finite_lin_indep_then_any_subscheme_lin_indep}
          If $R \subset \PP V$ is finite, whose embedding into
          $\PP V$ is linearly normal, then any subscheme $R' \subset R$ is
          linearly normal.
          In particular, in \ref{it:linindepintersect}, if $X$ and $Y$ are finite, it is enough to assume $X\cup Y$ is linearly normal.
  \end{enumerate}
\end{lem}
\begin{proof}
    In~\ref{it:linindeppoints} the scheme $R$ is a union of $r = \dim
    H^0(\ccO_X(1))$ points. Choose lifts of those points to $v_1, \ldots
    ,v_r\in V$, then $\dim \sspan{R} = \dim \sspan{v_1, \ldots ,v_r} - 1$ is
    equal to $r-1$ if and only if lifts are linearly independent.
    To prove~\ref{it:linindepintersect} note that $\sspan{X \cap Y} \subset
    \sspan{X} \cap \sspan{Y}$; it is enough to check that these spaces have
    the same dimension.
    By a slight abuse of notation, let $\sspan{X} + \sspan{Y} \subset \PP V$
       be the linear subspace spanned by $\sspan{X}$ and $\sspan{Y}$, which is equal to $\sspan{X \cup Y}$.
    To prove the equality of dimensions, we note that
    \begin{equation}\label{eq:tmpsspans}
        \dim \sspan{X}\cap \sspan{Y} = \dim \sspan{X} +
        \dim \sspan{Y} - \dim (\sspan{X} + \sspan{Y}),
    \end{equation}
    where $\sspan{X} + \sspan{Y} = \sspan{X \cup Y}$.
    Linear normality and Equation~\eqref{eq:tmpsspans} imply that
    \begin{equation}\label{eq:tmpsspanstwo}
        \dim (\sspan{X}\cap \sspan{Y}) = \dim H^0(\ccO_X(1)) + \dim
        H^0(\ccO_Y(1)) - \dim H^0(\ccO_{X\cup Y}(1)) - 1.
    \end{equation}
    Using the long exact sequence of cohomologies of the following short exact sequence:
    \[
       0 \to \ccO_{X\cup Y}(1) \to \ccO_{X}(1) \oplus \ccO_Y(1) \to \ccO_{X\cap Y}(1) \to 0,
    \]
    and exploiting the linear normality of $X\cup Y$ we transform
    Equation~\eqref{eq:tmpsspanstwo} to
    \[
        \dim (\sspan{X}\cap \sspan{Y}) = \dim H^0(\ccO_{X\cap Y}(1)) -1 = \dim
        \sspan{X\cap Y},
    \]
    which concludes proof of~\ref{it:linindepintersect}.
    To prove~\ref{item_R_finite_lin_indep_then_any_subscheme_lin_indep}, note
    that $\dim H^0(R, \ccO_R(1)) = \deg R$ for any finite scheme. If $\deg R'
    = \deg R - k$, then $R'$ is cut out of $R$ by at most $k$ equations, so
    $\dim\sspan{R'} \geq \dim \sspan{R} - k = \deg R' - 1 = 
     \dim H^0(R',\ccO_{R'}(1))-1$.
    By~\eqref{eq:dimensionofspan} this implies
      $\dim \sspan{R'} = \dim H^0(R',\ccO_{R'}(1))-1$, 
      so that $R'$ is linearly normal.
\end{proof}
Note however that Property~\ref{item_R_finite_lin_indep_then_any_subscheme_lin_indep} fails to be true if $R$ is not finite.
For example, if $R=\PP V$ and $R'$ is finite but very long (degree of $R'$ is more than $\dim V$),
  then $R$ is linearly normal, but $R'$ is not.

It might be difficult to control the linear normality of embeddings, 
  but there is a stronger notion of Castelnuovo-Mumford regularity which is easier to control due to its inductive properties, 
  and it implies linear normality of Veronese reembeddings.
In the following subsection we explain how to use it.

\subsection{Castelnuovo-Mumford regularity}\label{sec_Castelnuovo_Mumford_reg}

Let $\ccF$ be a \emph{coherent sheaf} on $\PP V$. 
By this we mean a sheaf corresponding to a graded $\Sym V^*$-module $M$ in a similar way
  as $\ccI_X$ above corresponds to $I(X)$, and the twists $\ccO_{\PP V}(d)$ correspond to shifts of $\Sym V^*$ by $d$. 
Here the main coherent sheaves we will consider are $\ccO_{X}(d)$ and $\ccI_X(d)$ for a subscheme $X\subset \PP V$.
The Euler characteristic of the twists $\ccF(d) = \ccF \otimes \ccO(d)$ is a polynomial in $d$, called the \emph{Hilbert polynomial}
  of $\ccF$. 
  In the case $\ccF = \ccO_{X}$, we denote this polynomial by $h_{X \subset \PP
V}(d)$, see~\cite[Section~III.3.1]{eisenbud_harris},
or~\cite[Section~18.6]{vakil_FoAG}.
For sufficiently large $d$, the Euler characteristic is equal to the
Hilbert function by Serre's vanishing theorem~\cite[Theorem~18.1.4]{vakil_FoAG}.
How large $d$ we have to take to obtain equality of the Hilbert polynomial and Hilbert function 
   is measured by the notion of Castelnuovo-Mumford regularity.
We say $\ccF$ is \emph{Castelnuovo-Mumford $\delta$-regular}\footnote{%
Many authors simply say \emph{$\delta$-regular}. 
This is a different notion than the regularity of schemes discussed in Section~\ref{sec_smooth_regular_flat},
  which is why we will keep mentioning the names of Guido Castelnuovo and David Mumford.}
if $H^i(\ccF(\delta-i)) =0$ for all $i >0$.
If $\ccF$ is Castelnuovo-Mumford $\delta$-regular, 
  then it is also Castelnuovo-Mumford $d$-regular for all $d\geq \delta$ \cite[Lemma 5.1(b)]{fantechi_et_al_fundamental_ag}.

\begin{cor}\label{cor_regular_implies_independent_embedding}
   If $X\subset \PP V$ is a subscheme such that $\ccI_X$ is Castelnuovo-Mumford $\delta$-regular, 
     then $\nu_d(X) \subset \PP (\DPV{d})$ is a linearly normal embedding for all $d\ge \delta -1$.
\end{cor}

We will need a technical lemma (Lemma~\ref{ref:mumfordreg:lem}) relating the Castelnuovo-Mumford regularities of sheaves appearing in a short exact sequence, 
  under an assumption that the quotient sheaf is ``very small''.
For this purpose we introduce the notions of \emph{support} and \emph{associated points} of a coherent sheaf on $\PP V$.

Castelnuovo-Mumford regularity does not depend on the base field change, as it
is defined in terms of cohomologies, which base-change well.
In particular, in the proofs we can always assume that the base field is infinite.
This is important for inductive procedures, as whenever the base field is infinite, we can pick a hyperplane $W \subset V$, 
  which avoids finitely many (not necessarily closed) points in $\PP V$.
The points we want to avoid are the \emph{associated points} of sheaves in question.
To define them, suppose $\ccF$ is a coherent sheaf and $x\in \PP V$ is any point.
Let $\ccF_x$ be  the stalk of $\ccF$ at $x$, which is a $\ccO_{\PP V, x}$-module.
Let $\gotp_x$ be maximal ideal of the local ring $\ccO_{\PP V, x}$.
We say that $x$ is the \emph{associated point} of $\ccF$ if there exists an element of $\ccF_x$, whose annihilator is $\gotp_x$.
A crucial observation is that every coherent sheaf on $\PP V$ has only
finitely many associated points \cite[Tag~05AF]{stacks_project},
or~\cite[Section~5.5]{vakil_FoAG}. Intuitively, as Ravi~Vakil vividly explains, associated points are the ones
``that control everything'' or the ones, where ``all the action happens'' (\cite[Section~5.5, title and p.~168]{vakil_FoAG}).

The \emph{support} of $\ccF$ consists of all points
$x\in \PP V$ such that $\ccF_x \ne 0$.
It is a closed subset of $\PP V$ by \cite[Tag~01BA]{stacks_project}
  and it is equal to the closure of all the associated points of $\ccF$ by \cite[Tag~05AH]{stacks_project}.
It can be also considered as a scheme with reduced structure.

\begin{example}\label{exam_finite_support_regular}
   Suppose $\ccF$ is a coherent sheaf whose support, considered as a reduced scheme,  is finite.
   (For instance $\ccF = \ccO_R$ for a finite subscheme $R \subset \PP V$.)
   Then $\ccF$ is Castelnuovo-Mumford $\delta$-regular for all integers $\delta$. 
   Indeed, $\ccF \simeq \ccF(\delta)$ for any $\delta$, 
     and $\ccF$ is  Castelnuovo-Mumford $\delta$-regular for some $\delta$.
\end{example}
  
\begin{example}
   Suppose $X\subset \PP^3_{\kk}$ is a closed subscheme defined by the homogeneous ideal 
      $I_X = (\alpha^3, \alpha^2\beta, \alpha\beta^2, \alpha^2\gamma, \alpha\beta\gamma)$.
   Then the support of the sheaf $\ccI_X$ is $\PP V$, it has only one associated point, which is the generic point of $\PP V$, 
      i.e.~the point corresponding to the prime ideal $(0)$.
   On the other hand, the structure sheaf $\ccO_X$ has its support on the plane $(\alpha)$,
     while the associated points of $\ccO_X$ are the points corresponding to the prime ideals $(\alpha)$, $(\alpha, \beta)$ 
     and $(\alpha, \beta, \gamma)$. 
\end{example}

\begin{example}
   Suppose $\ccF$ is a non-zero locally free sheaf on $\PP V$, or equivalently, $\ccF$ is the sheaf of local sections of a vector bundle.
   Then the support of $\ccF$ is $\PP V$ and the unique associated point is the generic point of $\PP V$.
\end{example}

For a hypersurface $Y\subset \PP V$ defined by an ideal generated by 
  a homogeneous polynomial $f \in S^d V^*$ and for every integer $s$
  there is a short exact sequence:
\begin{equation}\label{equ_ideal_structure_sheaf_sequence_for_hypersurface}
   0 \to \ccO_{\PP V}(s-d) \to \ccO_{\PP V}(s) \to \ccO_{Y} (s) \to 0.
\end{equation}
The first nontrivial map $\ccO_{\PP V}(s-d) \to \ccO_{\PP V}(s)$ is the multiplication by 
  $f$ interpreted as a global section of $\ccO_{\PP V}(d)$.
An important observation, which is ubiquitous in algebraic geometry, is:
\begin{prop}\label{prop_ideal_structure_sheaf_seq_tensored_with_F}
  If $\ccF$ is a coherent sheaf on $\PP V$ and $Y \subset \PP V$ is a hypersurface 
    that does not contain any of the associated points of $\ccF$,
    then the sequence \eqref{equ_ideal_structure_sheaf_sequence_for_hypersurface} tensored with $\ccF$ is still exact:
    \[
        0 \to \ccF(s-d) \to \ccF(s) \to \ccF \otimes \ccO_{Y} (s) \to 0.
    \]
\end{prop}
\begin{proof}
    It is enough to check that $i:\ccF(s-d) \to \ccF(s)$, the multiplication
    by the defining equation of $Y$, is injective. Its kernel would be
    supported at a subset of associated points of $\ccF$
    by~\cite[5.5L]{vakil_FoAG}, thus it is enough to check that $i$ is
    injective near each associated
    point of $\ccF$. But this is the assumption: every such point $x$ does not lie on $Y$
    so the map $\ccO(s-d)\to \ccO(s)$ is an isomorphism near $x$. Then
    $\ccF(s-d)\to \ccF(s)$ is an isomorphism near $x$ as well.
\end{proof}
We stress, that in
Proposition~\ref{prop_ideal_structure_sheaf_seq_tensored_with_F} we only
assume that $Y$ does not contain the associated points of $\ccF$; in contrast
$Y$ can intersect
the support of $\ccF$.
In particular, if $\ccF$ has only one associated point, which is the generic point of $\PP V$,
  then any $Y$ satisfies the assumption.
In general, if $\ccF$ is arbitrary, and $\kk$ is infinite,
  then there is a hypersurface $Y$ of any degree that avoids all associated points of $\ccF$.
  
We note the following dependence between the Castelnuovo-Mumford regularities of sheaves in a short exact sequence.
    \begin{lem}\label{ref:mumfordreg:lem}
        Let $0\to \ccF'\to \ccF\to \ccF''\to 0$ be a short exact sequence of sheaves on
        $\PP V$ and pick integers $\delta, r \geq 0$. Suppose that
        \begin{enumerate}
            \item $\ccF$ is Castelnuovo-Mumford $\delta$-regular,
            \item The support of $\ccF''$, as scheme with reduced structure, is finite,
            \item $\dim H^{0}(\ccF'') \leq r$.
        \end{enumerate}
        Then $\ccF'$ is Castelnuovo-Mumford $(\delta+r)$-regular.
    \end{lem}

    \begin{proof}
        The sheaf $\ccF''$ is Castelnuovo-Mumford $(\delta-1)$-regular by Example~\ref{exam_finite_support_regular}.
        Looking at the piece 
        \[
          H^{i-1}(\ccF''(s)) \to H^{i}(\ccF'(s))\to H^{i}(\ccF(s))
        \]
          of the long exact sequence of cohomologies one sees that $H^i(\ccF'(s)) = 0$
          whenever $i > 1$ and $s\geq \delta -i$, in particular
        $H^i(\ccF'(\delta+r-i)) = 0$ for all $i > 1$. Moreover we deduce that if $H^1(\ccF'(s-1)) = 0$ for
        some $s\geq \delta$ then $\ccF'$ is Castelnuovo-Mumford $s$-regular.

        By a base-change we may assume that the underlying field $\kk$ is infinite.
        Take a general hyperplane $W \subset V$.
        Then $\PP W$ is disjoint from the associated points of all three sheaves 
        $\ccF'$, $\ccF$, $\ccF''$. 
        Hence $\PP W$  is disjoint from the support of $\ccF''$ and $\ccF'' \otimes \ccO_{\PP W}$ is the zero sheaf by 
           a trivial instance of Proposition~\ref{prop_ideal_structure_sheaf_seq_tensored_with_F}.
        Thus we have the isomorphism $\ccF'\otimes \ccO_{\PP W} = \ccF\otimes \ccO_{\PP W}$ 
           and the latter is Castelnuovo-Mumford  $\delta$-regular by a remark before \cite[Lemma~5.1]{fantechi_et_al_fundamental_ag}.
        Fix any $s \geq \delta$.
        The long exact sequence of cohomologies of 
        $0\to \ccF'(s-1)\to \ccF'(s)\to \ccF'\otimes \ccO_{\PP W}(s)\to 0$ 
          (Proposition~\ref{prop_ideal_structure_sheaf_seq_tensored_with_F}) gives
        \[
            0\to H^0(\ccF'(s-1))\to H^0(\ccF'(s)) \overset{\rho_s}\to H^0(\ccF'\otimes \ccO_{\PP W}(s))\to H^1(\ccF'(s-1))\to
            H^1(\ccF'(s))\to 0.
        \]
        Suppose that $\dim H^{1}(\ccF'(s-1)) = \dim H^1(\ccF'(s))$. 
        Then $\rho_s$ is onto. 
        But also 
        \[
          H^0(\ccO(1))\tensor H^0(\ccF'\otimes \ccO_{\PP W}(s))\to H^{0}(\ccF'\otimes \ccO_{\PP W}(s+1))
        \]
        is onto \cite[Lemma~5.1a]{fantechi_et_al_fundamental_ag},
           which implies that $\rho_{s+1}$ is onto, 
           and by induction $\rho_{t}$ is onto for $t\geq s$.
        Then $H^{1}(\ccF'(s-1)) = H^1(\ccF'(t)) = 0$ for $t\gg 0$.
        We deduce that $H^1(\ccF'(s-1)) \neq 0$ implies $\dim H^1(\ccF'(s)) < \dim H^1(\ccF'(s-1))$. 
        Moreover  $H^1(\ccF'(s-1)) = 0$ implies $H^1(\ccF'(s)) = 0$.

        Now $H^0(\ccF''(s))\to H^{1}(\ccF'(s))$ is surjective for $s\geq \delta-1$, 
          so $\dim H^1(\ccF'(\delta-1)) \leq \dim H^0(\ccF''(\delta-1)) = \dim H^0(\ccF'') \leq r$ by assumption.
        Then $\dim H^1(\ccF'(\delta+r-1)) = 0$, thus $\ccF'$ is Castelnuovo-Mumford $(\delta+r)$-regular.
    \end{proof}

\begin{cor}\label{cor_finite_schemes_r_regular}
   Suppose $R \subset \PP V$ is a finite scheme of degree $r$.
   Then $\ccI_R$ is Castelnuovo-Mumford $r$-regular.
\end{cor}
\begin{proof}
   The sequence $0 \to \ccI_R \to \ccO_{\PP V} \to \ccO_{R} \to 0$ satisfies assumptions of Lemma~\ref{ref:mumfordreg:lem} 
     with $\delta =0$.
\end{proof}

As a conclusion, we obtain a strengthening of
\cite[``Main Lemma'', Lemma~1.2]{jabu_ginensky_landsberg_Eisenbuds_conjecture}.
    \begin{lem}\label{ref:bglmainlemma:lem}
        Let $X\subseteq \mathbb{P}_{\kk}^n$ be a subscheme and suppose that $\ccI_X$ is Castelnuovo-Mumford $\delta$-regular for some 
        $\delta \ge 0$. 
        Let $R$ be a $0$-dimensional subscheme of $\mathbb{P}^n$ of degree at most $r = \deg R$.
        Then
        \begin{equation}\label{eq:spanintersections}
            \sspan{\nu_d(R)} \cap \sspan{\nu_d(X)} = \sspan{\nu_{d}(X\cap R)}
        \end{equation}
        for all integers $d$ satisfying $d \geq \delta + r - 1$.
    \end{lem}
Our lemma is slightly stronger than in the original statement \cite[Lemma~1.2]{jabu_ginensky_landsberg_Eisenbuds_conjecture}: 
   instead of $\delta$ there is the Gotzmann number of the Hilbert polynomial of $X$, 
   which is a possibly larger integer.
Moreover, we state it over any base field $\kk$.

    \begin{proof}
        First we claim that $\ccI_{X\cup R}$ is Castelnuovo-Mumford $(\delta+r)$-regular.
        Indeed, in the exact sequence
        \[
           0\to \ccI_{X\cup R} \to \ccI_X \to \ccJ\to 0.
        \]
        The quotient sheaf $\ccJ$ has a finite support contained in $R$.
        Thus Lemma~\ref{ref:mumfordreg:lem} implies $\ccI_{X\cup R}$ is Castelnuovo-Mumford $(\delta+\dim H^0(\ccJ))$-regular.
        Thus it sufficies to prove $\dim H^0(\ccJ) \leq r$.
        But the composition morphism $\ccI_X\to \ccO_X\to \ccO_R$ descends to an injection $\ccJ\to \ccO_R$, 
        thus $\dim H^0(\ccJ) \leq \dim H^0(\ccO_R) = \deg R = r$,
        which finishes the proof of the claim.
        
        Moreover, since $\delta \ge 0$, the ideal sheaves $\ccI_{R}$ and $\ccI_{X \cap R}$ are $\delta + r$ 
           regular by Corollary~\ref{cor_finite_schemes_r_regular}.
        By Corollary~\ref{cor_regular_implies_independent_embedding} 
           all schemes $\nu_d(X)$,  $\nu_d(R)$, $\nu_d(X\cup R)$, $\nu_d(X\cap R)$  
           are embedded in a linearly normal way for $d\ge \delta +r -1$.
        By Grassmann Lemma (Lemma~\ref{lem_properties_of_lin_indep_emb}), the equality in \eqref{eq:spanintersections} holds.
   \end{proof}
   
  As a side remark we also mention a classical fact that sufficiently high degree Veronese reembeddings 
    of a projective scheme is linearly normal.
  \begin{lem}\label{lem_regular_then_complete_linear_system}
     Suppose $X\subset \PP V$ is a projective embedded scheme (possibly embedded by a non-complete linear system),
        and  $\ccI_X$ is Castelnuovo-Mumford $\delta$-regular.
     Then for all $d\ge \delta -1$ the Veronese reembedding $\nu_d(X) \subset \sspan{\nu_d(X)}\subset \DPV{d}$ 
        is embedded by a complete linear system.
  \end{lem}
  \begin{proof}
     The map $H^0(\ccO_{\PP (\DPV{d})}(1)) \to H^0(\ccO_{\nu_d(X)}(1))$ is equal to  $H^0(\ccO_{\PP V}(d)) \to  H^0(\ccO_{X}(d))$.
     Since $H^{1}(\ccI(d)) =0$, the latter (and hence also the first) is surjective. 
  \end{proof}

  \section{Hilbert scheme of finite subschemes}   \label{sec_Hilbert}
  
Hilbert scheme of a scheme $X$ is a moduli space for all flat families of subschemes of $X$.  
There are many excellent resources about categorical properties and existence (for projective scheme $X$) of such scheme,
see for instance~\cite{fantechi_et_al_fundamental_ag, miller_sturmfels_combinatorial_commutative_algebra,
    haiman_sturmfels_multigraded_Hilb,
nakajima_lectures_on_Hilbert_schemes}.
In this section we briefly summarise the properties for the Hilbert scheme of finite subschemes.

\subsection{Definition and constructions of the Hilbert
scheme}\label{sect_Hilbert_intro}

Suppose $Z$ and $T$ are any schemes and $\pi\colon Z\to T$ is a flat finite morphism.
Then the push forward sheaf $\pi_*\ccO_Z$ is locally
free~\cite[Proposition~12.19]{gortz_wedhorn_algebraic_geometry_I}, so the rank of the structure ring of
the fibre $\ccO_{Z_t}$ (as a vector space over the residue field of $t$) is
locally constant; in other words the degree of $\ccO_{Z_t}$ is locally constant.
We say that $Z\to T$ has \emph{degree} $r$ if this degree is constant and equal
to $r$ globally, i.e.~constant over all connected components of $T$.
In particular, if $T$ is connected, and $\pi\colon Z\to T$ is a flat finite morphism, 
then $Z \to T$ has some degree $r$.

Fix a scheme $X$ and suppose $r\ge 1$ is an integer.
Define $\Hilb_r(X)$ to be the \emph{Hilbert scheme} of finite degree $r$
subschemes of $X$. It is also called the \emph{Hilbert scheme of $r$ points on
$X$}.
It comes
with a subscheme
$\ccU_r\subset X \times \Hilb_r(X)$ satisfying the following universal property.
If $T$ is any scheme, and  $Z \subset X \times  T$ a subscheme such that
  the projection $Z \to T$ is flat and finite of degree $r$, 
then there exists a unique morphism $T \to \Hilb_r(X)$ such that $Z = \ccU_r \times_{\Hilb_r(X)} T$.
  
If such Hilbert scheme exists, then its $\kk$-rational points are in $1$-to-$1$ correspondence with
  finite subschemes of $X$ of degree $r$. 
If $R \subset X$ is a finite subscheme of degree $r$, 
  then by $[R] \in \Hilb_r(X)$ we denote the corresponding $\kk$-rational point of the Hilbert scheme.

The existence of Hilbert schemes is established for several important classes of schemes $X$.

If $X=\PP V$ is a projective space, then $\Hilb_r(\PP V)$ is constructed as a closed subscheme of a Grassmannian.
Very roughly, a finite subscheme $R$ is mapped to the linear span $\sspan{\nu_{2r-1}(R)}$, which uniquely determines $R$ by 
  Lemma~\ref{ref:bglmainlemma:lem}.
In particular, $\Hilb_r(\PP V)$ is projective.

If $X \subset \PP V$ is a projective scheme, then $\Hilb_r(X)$ is constructed as a closed subscheme of $\Hilb_r(\PP V)$.
Roughly, the closed condition is given by the containment in $X$. The scheme $\Hilb_r(X)$ for projective $X$ is also projective.
In particular, it has finitely many irreducible components.

If $X$ is an open subset of a projective scheme $\overline{X}$ (for instance,
$X$ is an affine space), 
  then $\Hilb_r(X)$ is an open subset of $\Hilb_r(\overline{X})$; again the open condition is given by containment in $X$.

If $X = \Spec \kk[[\fromto{\alpha_1}{\alpha_n}]]$ or more generally $X = \Spec A$ for any $\kk$-algebra
  $A$, then $\Hilb_r(X)$ exists by~\cite[Theorem~5.4]{gustavsen_laksov_skjelnes__existence_for_affine}.


\begin{example}
   Suppose $r=1$. 
   Then $\Hilb_1(X) = X$ and $\ccU_1 \simeq X$ embedded as the diagonal $\Delta \subset X \times X$ with a
   projection map to $X = \Hilb_1(X)$.
\end{example}

\begin{example}
   Suppose $X=\PP^1_{\kk} = \PP V$ for two a dimensional vector space $V$. 
   Then $\Hilb_r(X) = \PP(S^r V^*) \simeq \PP^r_{\kk}$, and 
   \[
     \ccU_r = (F = 0) \subset \PP(S^r V^*) \times \PP V,
   \]
   where $F$ is the universal form of degree $r$; if we fix homogeneous coordinates $a_i$
   on $\PP(S^r V^*)$ corresponding to a basis $m_i$ of $S^r V^*$, then $F = \sum
   a_i m_i$, where $a_i$ is a function on $\PP(S^rV^*)$ and $m_i$ is a form
   on $\PP V$. Thus $F\in H^0(\ccO_{\PP(S^r V^*)}(1) \boxtimes \ccO_{\PP V}(r))$.
   Over a point $(b_i)\in \PP(S^r V^*)$ the corresponding fibre is a subscheme
   of $X = \PP_{\kk}^1$, cut out by the equation $\sum b_i m_i$.
   This example vastly generalises; the only assumption we need is that $X$ is
   a smooth curve, see~\cite{fantechi_et_al_fundamental_ag}.
\end{example}

\begin{example}\label{ex_Hilb2_of_a_line_with_an_embedded_pt}
   Suppose $r=2$ and $X \subset \PP^2_{\kk}$ is a subscheme defined by $I(X) = (\alpha^2, \alpha\beta)$.
   Then $\Hilb_2 (X)$ consists of two irreducible components: 
      one of them, isomorphic to $\PP^2_{\kk}$, represents schemes of degree
      two contained in the line $\alpha =0$,
      and the other is a $\PP^1_{\kk}$ and represents points of degree two supported at the point $x_0$ defined by $\alpha=\beta=0$.
   The two components intersect in a $\kk$-rational point which corresponds to the scheme whose ideal is $(\alpha, \beta^2)$.
   In addition there is an embedded component supported at the line in $\PP^2_{\kk}$.
   This line corresponds to schemes whose support contains $x_0$, and the nonreduced structure comes from an infinitesimal deformations of $x_0$, 
      i.e.~from the nonreduced structure of $X$ at $x_0$.
\end{example}

\subsection{Base change of Hilbert scheme}\label{sec_Hilb_base_change}

The categorical properties of the Hilbert scheme imply that it is a very natural object
  and it is invariant under base changes.
To explain a special case of this phenomena, 
  denote $\Hilb_{r}(X/\kk):=\Hilb_{r}(X)$ to stress that the construction concerns schemes over $\kk$.
Now if $\KK$ is another field containing $\kk$ as a subfield, then let $X_{\KK}:=X \times \Spec \KK$.
We can now consider $\Hilb_{r}(X_{\KK}/\KK)$, where the role of $\kk$ is replaced by $\KK$.
Similarly, denote the universal families by $\ccU_{r,\kk}$ and $\ccU_{r, \KK}$ respectively.

\begin{prop}\label{prop_base_change_for_Hilb_r}
   Suppose $X$ is a scheme such that $\Hilb_r(X/\kk)$ exists.
   Then 
   \begin{align*}
      \Hilb_r(X_{\KK}/ \KK) &= \Hilb_{r}(X/\kk) \times \Spec \KK \text{ and } \\
      \ccU_{r, \KK}  = \ccU_{r, \kk} \times \Spec \KK &=  \ccU_{r, \kk} \times_{\Hilb_{r}(X/\kk)} \Hilb_{r}(X_{\KK}/\KK).
   \end{align*}
   Moreover, all the relevant maps are commutative:
   \[
     \begin{tikzcd}
        \ccU_{r, \KK} \arrow{r}{}\arrow{d}{}& \ccU_{r, \kk}\arrow{d}{} \\
        X_{\KK} \times_{\Spec \KK} \Hilb_r(X_{\KK}/ \KK) \arrow{r}{}\arrow{d}{} & 
          X \times \Hilb_r(X_{\kk}/ \kk)\arrow{d}{} \\
        \Hilb_r(X_{\KK}/ \KK) \arrow{r} & \Hilb_r(X/ \kk)  
     \end{tikzcd}
   \]
\end{prop}
\begin{proof}
    See~\cite[(5) pg. 112]{fantechi_et_al_fundamental_ag}.
\end{proof}


\subsection{Comparing embedded smoothings and Hilbert schemes}\label{sect_embedded_smoothings_and_Hilb}

    We now compare the abstract notion of smoothability of a finite scheme $R$ to
    the geometry of Hilbert scheme around $[R]$. 
    This enables us to give the notion of 
    smoothability a more direct touch; prove that there exist nonsmoothable
    schemes and that for a given family the set of smoothable fibres is
    closed.

    Fix a scheme $X$ of finite type over $\kk$ and such that the Hilbert
    scheme $\Hilb_r(X)$ of $r$ points on $X$ exists.
    Consider the subset $\Hilb_r^{\circ}(X) \subset \Hilb_r(X)$ consisting of points $t \in \Hilb_r(X)$ such that the fibre of $\ccU_{r} \to \Hilb_r(X)$ over $t$ is smooth.
    The subset $\Hilb_r^{\circ}(X)$ is open in $\Hilb_r(X)$ by~\cite[top of
    p.~580, item (18)]{gortz_wedhorn_algebraic_geometry_I}. Endow it with a
    reduced scheme structure.
    Let $\Hilb^{sm}_r(X) := \overline{\Hilb_r^{\circ}(X)} \subset \Hilb_r(X)$
    be its Zariski closure in $X$ with reduced scheme
    structure; this is a
    reduced scheme and a reduced union of components of $\Hilb_r(X)$.
    Let $\kk \subset \KK$ be a field extension and $X_{\KK} = X \times \Spec \KK$. 
    Since smoothness is defined on (geometric) fibres, it is faithfully
    preserved by a base change, so we have $\Hilb_{r}^{\circ}(X_{\KK} / \KK) =
    \Hilb_r^{\circ}(X)\times \Spec\KK$.  Therefore also
    \begin{equation}\label{eq:basechangeforsmoothablecomponent}
    \Hilb_{r}^{sm}(X_{\KK}/\KK) = \reduced{\left(\Hilb_{r}^{sm}(X) \times \Spec \KK\right)}.
    \end{equation}
    (Recall, that $\reduced{(\cdot)}$ denotes the maximal reduced subscheme.)
    The scheme $\Hilb_{r}^{sm}(X) \times \Spec \KK$ on the right hand side of \eqref{eq:basechangeforsmoothablecomponent}
      possesses a nonreduced structure only over an imperfect field~$\kk$~\cite[Corollaries~5.56, 5.57]{gortz_wedhorn_algebraic_geometry_I}.
    Therefore, if, for instance, $\cchar \kk =0$  or $\kk$ is algebraically closed, or $\kk$ is a finite field,
      then Equation~\eqref{eq:basechangeforsmoothablecomponent} 
      may be strengthened to 
      \[
        \Hilb_{r}^{sm}(X_{\KK}/\KK) = \Hilb_{r}^{sm}(X) \times \Spec \KK.
      \]
      By Equation~\eqref{eq:basechangeforsmoothablecomponent}
      the scheme $\Hilb_r^{sm}(X)$ is empty only in the degenerate case where $X$ is
      zero-dimensional --- otherwise we may always find infinitely many
      morphisms $\Spec \kkbar \to X$ with pairwise different images, and in particular a tuple of $r$ of them; such
      a tuple gives a map $\Spec \kkbar  \to \Hilb_r(X)$, whose image  is contained in $\Hilb^{sm}_r (X)$.
    \begin{prop}\label{ref:embeddedvsHilbert:prop}
        Let $R \subset X$ be a finite subscheme of degree $r$. Then the
        following are equivalent
        \begin{enumerate}
            \item\label{it:1smooth} $R$ is smoothable in $X$,
            \item\label{it:2smooth} $[R]\in \Hilb_r^{sm}(X)$.
        \end{enumerate}
    \end{prop}
    \begin{proof}
        To show \ref{it:1smooth}$\implies$\ref{it:2smooth}, pick an embedded smoothing of $R$ in $X$, 
          which is a family $Z \subset X \times T$ flat over an irreducible base $T$,
        such that a fibre over a $\kk$-rational point $t\in T$ is $Z_{t} = R \times \set{t}$.
        In particular, the degree of $Z \to T$ is $r$, 
          and hence it gives a map $\varphi:T\to \Hilb_r(X)$. 
        The base $T$ is irreducible and the fibre of $Z\to T$ over the generic point $\eta \in T$ is smooth.
        Thus the image of the generic point  $\varphi(\eta)$ is contained in $\Hilb_r^{0}(X)$, 
          and the image of any point of $T$ is contained in its closure $\Hilb_r^{sm}(X)$.
        In particular, $\varphi(t) = [R] \in \Hilb_r^{sm}(X)$.

        To show \ref{it:2smooth}$\implies$\ref{it:1smooth},
          pick an irreducible component $T$ of $\Hilb_r^{sm}(X)$ containing $[R]$ 
          and let $Z \subset X \times T$ be the restriction of the universal family $\ccU_r$ to $T$.
        The map $f\colon Z\to T$ is flat, finite and by definition of $\Hilb_r^{sm}(X)$ there exists an open dense $U$ such that $f:f^{-1}(U)\to U$ is smooth.
        Hence in particular the fibre over the generic point of $T$ is smooth, so $Z\to T$ gives an embedded smoothing of $R$.
    \end{proof}

    The following corollary enables us to reduce questions of smoothability to
    schemes over an algebraically closed field.
    \begin{cor}\label{ref:basechangesmoothings:cor}
        Let $R$ be a finite scheme over $\kk$ and $\kk \subset \KK$ be a field extension. 
        Then $R$ is smoothable if and only if the scheme over $\KK$
        \[
            R_{\KK} = R\times \Spec \KK
        \]
        is smoothable over $\KK$ (so with all definitions in
        Section~\ref{sec:smoothability} having $\kk$ replaced with $\KK$).
    \end{cor}
    \begin{proof}
        Suppose $R$ is smoothable and take its smoothing $(Z, R)\to (T, t)$,
        then $Z_t  \simeq R$. 
        The point $t$ is $\kk$-rational, so it gives a
           $\KK$-rational point $t_{\KK} \in T_{\KK} = T \times \Spec \KK$
           just by a product of the composition $\Spec\KK\to \Spec\kk\to T$ and the identity $\Spec \KK \to \Spec \KK$. 
        Moreover $Z_{t_{\KK}} = Z_t \times \Spec \KK = R_{\KK}$.

        Suppose now $R_{\KK}$ is smoothable as a scheme over $\KK$. 
        Since $R$ is finite, we can embed $R$ into a $\kk$-projective space $\PP^N_{\kk}$ by Lemma~\ref{lem_finite_is_projective}. 
        Since $\PP^N_{\kk}$ is smooth and $\PP^N_{\KK} = \PP^N_{\kk} \times \Spec \KK$,
          by Theorem~\ref{ref:abstractvsembedded:thm} the scheme $R_{\KK}$ is smoothable (over $\KK$)
          in $\PP^N_{\KK} = \PP^N_{\kk} \times \Spec \KK$.
        By Proposition~\ref{ref:embeddedvsHilbert:prop} this means that the point $[R_{\KK}]$ lies in 
        $\Hilb_r^{sm}(\PP^N_{\KK}/\KK)= \reduced{\left(\Hilb_r^{sm}(\PP^N_{\kk}) \times \Spec \KK\right)}$.
        The image of the projection of this point to $\Hilb_r^{sm}(\PP^N_{\kk})$ is equal to $[R]$ 
          by the commutativity of the diagram in Proposition~\ref{prop_base_change_for_Hilb_r}.
        Using Proposition~\ref{ref:embeddedvsHilbert:prop} again, we get that $R$ is smoothable in $X$.
    \end{proof}
    \begin{proof}[Proof of Proposition~\ref{prop_smoothability}]
        Using Corollary~\ref{ref:basechangesmoothings:cor} we may reduce to
        the case $\kk = \kkbar$. In this case the proof was given in
        \cite[Theorem~1.1]{cartwright_erman_velasco_viray_Hilb8}
        and~\cite[Theorem~A]{casnati_jelisiejew_notari_Hilbert_schemes_via_ray_families}.
    \end{proof}
    \begin{remark}\label{remark_smoothability_depends_only_on_char}
        Let $\kk$ be any field and $\kk_0 \subset \kk$ be its prime subfield,
        i.e.,~$\kk_0 = \mathbb{Q}$ or $\kk_0 = \mathbb{F}_p$. In particular
        $\kk_0$ is perfect.
        From~\eqref{eq:basechangeforsmoothablecomponent} it follows that
        $\Hilb^{sm} \mathbb{A}^n_{\kk} = \Hilb^{sm} \mathbb{A}^{n}_{\kk_0}
        \times_{\kk_0} \Spec \kk$.
        Hence the equality 
        $\Hilb^{sm} \mathbb{A}^n_{\kk} = \Hilb \mathbb{A}^{n}_{\kk}$  
         (or $\Hilb^{sm} \mathbb{A}^n_{\kk} = \reduced{(\Hilb \mathbb{A}^{n}_{\kk})}$)
        holds if
        and only if $\Hilb^{sm} \mathbb{A}^{n}_{\kk_0} = \Hilb \mathbb{A}^{n}_{\kk_0}$ 
         (or $\Hilb^{sm} \mathbb{A}^n_{\kk_0} = \reduced{(\Hilb \mathbb{A}^{n}_{\kk_0})}$)
        holds. 
        The latter equality depends only on the characteristic of $\kk$.
        From Proposition~\ref{ref:embeddedvsHilbert:prop} it follows that the
        validity of the statement \emph{all $\kk$-schemes of degree $r$ are
        smoothable} 
           is implied by the same statement with $\kk$ replaced with the algebraic closure $\overline{\kk_0}$, 
           which depends only on $r$ and the characteristic of $\kk$.
        Also the statement \emph{all Gorenstein $\kk$-schemes of degree $r$ are
        smoothable} is analogously implied by a statement over $\overline{\kk_0}$,
          because the Gorenstein property is faithfully preserved under base change
          (Proposition~\ref{prop_Gorenstein_base_change}).
    \end{remark}

\begin{example}[nonsmoothable schemes of high degree over any base field
    $\kk$]\label{ref:nonsmoothableOfHighDegree:ex}
    \newcommand{\mm}{\mathfrak{m}}%
    \newcommand{\Flarge}{\mathcal{F}}%
    We will now recall that nonsmoothable subschemes of
    $\mathbb{A}^n_{\kk}$ for every $n\geq 3$ and every field $\kk$ exist. The
    proof is nonconstructive. We follow~\cite{iarrobino__reducibility}.

    Consider finite subschemes $R \subset \mathbb{A}^n$ of fixed degree $r$. By
    Theorem~\ref{ref:abstractvsembedded:thm} and
    Proposition~\ref{ref:embeddedvsHilbert:prop} every such an $R$ is smoothable if and
    only if $[R]$ belongs to $\Hilb_r^{sm}(X)$.
    Suppose we constructed a family $\Flarge \subset \Hilb_r(X)$ such
    that $\dim \Flarge > d\cdot \dim X$. Then $\Flarge \not\subset
    \Hilb_r^{sm}(X)$, hence a general element of $\Flarge$ is nonsmoothable.
    We will construct such an $\Flarge$ below.

    We fix $n$ and consider a integer varying parameter $s$. For a function $f(s)$ with real values
    we say that $f$ is asymptotically $\Omega(s^a)$ if $\lim_{s\to \infty}
    f(s)/s^a$ exists and is a finite positive number.
    Denote by $\mm$ the ideal of the origin.
    We have $\dim_{\kk} \mm^s/\mm^{s+1} = \binom{n+s-1}{n-1}$, which
    asymptotically is $\Omega(s^{n-1})$.
    For an integer $s$ consider ideals $I$
    satisfying $\mm^{s+1} \subset I \subset \mm^{s}$.
    The set of ideals $I$ is parametrised by
    union of finitely many Grassmannians and the largest of them has dimension asymptotically
    $\Omega(s^{2(n-1)})$.
    On the other hand, $\dim_{\kk} \kk[\mathbb{A}^n]/\mm^{s}$
    asymptotically equal to $\Omega(s^n)$, so also the dimension of the
    smoothable component of the Hilbert scheme containing $I$ above is
    $\Omega(s^n)$.

    For every $n \geq 3$ we have $2(n-1) >
    n$, so that for $s\gg
    0$ the Grassmannian above is not contained in the smoothable component, so the
    Hilbert scheme is reducible.
    Replacing asymptotic bounds by a careful calculation, one can check that,
    for example, $\Hilb_r \mathbb{A}^3$ is reducible for $r\geq 96$.
\end{example}

    From Proposition~\ref{ref:embeddedvsHilbert:prop} we can also deduce that
    smoothability is a closed property.
    \begin{cor}\label{ref:limitOfSmoothableIsSmoothable:cor}
        Let $T$ be an irreducible scheme and $\eta$
        be its generic point. Let $\pi:Z\to T$ be a finite flat morphism such that
        the generic fibre $Z_{\eta}$ is smoothable over $\kappa(\eta)$ (where $\kappa(\eta)$ is the residue field of $\eta$). 
        Then every fibre $\pi^{-1}(t)$ of $\pi$ is smoothable over $\kappa(t)$ (where $\kappa(t)$ is the residue field of $t$).
    \end{cor}
    \begin{proof}
        Let $r = \deg \pi$.
        Choose a point of $t\in T$. We will show that $Z_t$ is smoothable.
        Choose an affine neighbourhood of $t$ in $T$ and restrict to it.
        After a restriction $Z$ and $T$ are affine. Moreover, $Z$ is finitely generated as
        $\ccO_T$-algebra by at most $r$ elements, so we have a closed embedding
        \begin{equation}\label{eq:embed}
            Z\subset \mathbb{A}_{\kk}^r \times T,
        \end{equation}
        which induces a map $\varphi\colon T\to \Hilb_r \mathbb{A}^r_{\kk}$.
        Restriction of~\eqref{eq:embed} gives embeddings $Z_{\eta} \subset
        \mathbb{A}^r_{\kappa(\eta)}$ and $Z_t \subset \mathbb{A}^r_{\kappa(t)}$.
        By Theorem~\ref{ref:abstractvsembedded:thm} and
        Proposition~\ref{ref:embeddedvsHilbert:prop} we have 
        $[Z_{\eta}]\in \Hilb^{sm}_r(\mathbb{A}^{r}_{\kappa(\eta)}/\kappa(\eta)) 
          = \Hilb^{sm}_r(X) \times\Spec \kappa(\eta)$. 
        The projection of $[Z_{\eta}]$ onto the first coordinate is
        equal to $\varphi(\eta)$, hence $\varphi(\eta)\in \Hilb^{sm}_r(X)$.
        Consequently, $\im \varphi \subset \Hilb^{sm}_r(X)$.
        The point $\varphi(t)$ of $\Hilb^{sm}_r(X) \times_{\kk} \Spec
        \kappa(t)$ is equal to the point $[Z_t]\in \Hilb^{sm}_r(X\times
        \kappa(t)/\kappa(t))$.
        By Proposition~\ref{ref:embeddedvsHilbert:prop} again, we have $Z_t$
        smoothable in $\mathbb{A}^{r}_{\kappa(t)}$, hence smoothable.
    \end{proof}

\subsection{Composition of smoothings}

We apply the base change property to show that the product of smoothable schemes is smoothable, and that schemes smoothable over larger fields are also smoothable over subfields.

\begin{prop}[product of smoothable is smoothable]\label{prop_product_is_smoothable}
   If $R_1$ and $R_2$ are two finite smoothable $\kk$-schemes, then $R_1
   \times R_2$ is a smoothable $\kk$-scheme.
\end{prop}
\begin{proof}
    By base change of smoothability
    (Corollary~\ref{ref:basechangesmoothings:cor}) we may
    assume $\kk = \kkbar$.
    Let $B = \Spec \kk[[\alpha]]$ and $(Z_i, R_i)\to (\Spec \kk[[\alpha]], 0)$
    be smoothings of $R_i$, which exist by
    Theorem~\ref{ref:goodbaseofsmoothing:thm}.
   Consider the fiber product $\pi:Z_1 \times_B Z_2 \to B$.
   The special fibre of $\pi$ is $R_1 \times R_2$,
     while the general fibre is $(Z_1)_{\eta} \times_{\eta} (Z_2)_{\eta}$,
     which is smooth as product of schemes smooth over $\eta$. Thus $\pi$ is a smoothing of $R_1 \times R_2$.
\end{proof}

\begin{cor}[composition of smoothings]\label{cor_composition_of_smoothings}
   Suppose $\kk\subset \KK$ is a finite field extension and $R$ is a finite scheme over $\KK$, smoothable over $\KK$.
   Then $R$ (considered as a scheme over $\kk$) is smoothable (over $\kk$).
\end{cor}
\begin{proof}
   Let $\kkbar$ be an algebraic closure of both $\kk$ and $\KK$.
   For clarity we avoid writing $\Spec$ next to fields.
   The scheme $R \times_{\KK} \kkbar$ is smoothable over $\kkbar$
   by assumption and Corollary~\ref{ref:basechangesmoothings:cor}.
   The scheme $\KK$ is smoothable
   over $\kk$ by Example~\ref{ex:fieldExtensions}, so the $\kkbar$-scheme $\KK
   \times_{\kk} \kkbar$ is smoothable over $\kkbar$ by the same corollary.
   The product
   \[
       (R \times_{\KK} \kkbar) \times_{\kkbar} (\KK
       \times_{\kk} \kkbar) = R
       \times_{\KK} (\kkbar \times_{\kkbar} (\KK
       \times_{\kk} \kkbar)) = (R
       \times_{\KK} \KK) \times_{\kk} \kkbar = R \times_{\kk}
       \kkbar
   \]
   is smoothable over $\kkbar$ by
   Proposition~\ref{prop_product_is_smoothable}. Therefore $R$ is smoothable
   over $\kk$ by Corollary~\ref{ref:basechangesmoothings:cor}.
\end{proof}

\begin{example}
   Let $\kk \subset \KK$ be a finite field extension and let $I \subset \KK[\fromto{\alpha_1}{\alpha_n}]$ be a monomial ideal.
   If the quotient ring $A:=\KK[\fromto{\alpha_1}{\alpha_n}]/I$ has finite $\KK$-dimension,
   then $\Spec A$ is smoothable over $\kk$.
   Indeed, finite monomial schemes over algebraically closed fields are
   smoothable by Hartshorne's distraction method
   (\cite[p.~34]{hartshorne_connectedness},
   ~\cite[Proposition~4.15]{cartwright_erman_velasco_viray_Hilb8}), hence they
   are smoothable over any field by
   Corollary~\ref{ref:basechangesmoothings:cor}.
   Using Proposition~\ref{prop_product_is_smoothable} the scheme $\Spec A$ is smoothable over $\kk$.
\end{example}

\subsection{Gorenstein locus}
The \emph{Gorenstein locus} of the Hilbert scheme. It is the set
$\Hilb^{Gor}_r(X) \subset \Hilb_r(X)$ consisting of points $[Z]$ such that $Z$
is Gorenstein. Gorenstein subschemes are important for our point of view because they are precisely
those which have few codegree one subschemes, see
Lemma~\ref{lem_Gorenstein_is_equivalent_to_finite_Hilb}. Later in
Lemma~\ref{lem_Gorenstein_are_enough_for_cactus} we will see that spans of
Gorenstein subschemes are enough to define cactus varieties.

As proven below, the Gorenstein locus is an open subset containing all smooth
schemes, so it may be thought of as an intermediate object between the set of
smooth schemes and the whole Hilbert scheme.
\begin{lem}\label{ref:GorensteinLocus:lem}
    Let $X$ be a scheme, and for a field extension $\kk \subset \KK$ denote by $\Hilb^{Gor}(X_{\KK}/\KK)$
      the Gorenstein locus of $\Hilb_r(X_{\KK}/\KK)$. Then:
    \begin{enumerate}
        \item\label{it:Gorlocfirst}  $\Hilb^{\circ}_r(X) \subset \Hilb^{Gor}_r(X)$,
        \item\label{it:Gorlocsecond} The subset $\Hilb^{Gor}_r(X)$ is open in~$\Hilb_r(X)$,
        \item\label{it:Gorlocthird} $\Hilb^{Gor}_r(X_{\KK}/\KK) = \Hilb^{Gor}_r(X) \times \Spec \KK$.
    \end{enumerate}
\end{lem}
\begin{proof}
    Point~\ref{it:Gorlocsecond} is proven as a special case in~\cite[tags 0C09
    and 0C02]{stacks_project}, it is also discussed in~\cite[Introduction,
    p.~2056]{casnati_notari_irreducibility_Gorenstein_degree_9}.
    Point~\ref{it:Gorlocfirst} follows, since~both subsets are open,
    and every $\kkbar$-point of $\Hilb^{\circ}_r(X)$ is isomorphic to $\left(
    \Spec \kkbar \right)^{\times r}$ which is Gorenstein by
    Example~\ref{ex_Gorenstein_schemes}.
    Point~\ref{it:Gorlocthird} follows from
    Proposition~\ref{prop_Gorenstein_base_change}.
\end{proof}

By Lemma~\ref{ref:GorensteinLocus:lem} the locus $\Hilb^{Gor}_r(X)$ has a
natural open subscheme structure; we will use it without further comments.
The Gorenstein locus is not necessarily dense; the first example is
$\Hilb_{8} \PP_{\kk}^4$, see~\cite{cartwright_erman_velasco_viray_Hilb8}.

\begin{lem}\label{lem_Gorenstein_is_equivalent_to_finite_Hilb}
   Suppose $R$ is a finite scheme of degree $r$.
   Consider the Hilbert scheme $\Hilb_{r-1}(R)$.
   Then $R$ is Gorenstein if and only if $\dim \Hilb_{r-1}(R) = 0$.
\end{lem}

\begin{proof}
   By invariance of the notions of \emph{Gorenstein}
   (Proposition~\ref{prop_Gorenstein_base_change}), of
   \emph{dimension}, and of \emph{Hilbert scheme}
   (Proposition~\ref{prop_base_change_for_Hilb_r})
     under the base change from $\kk$ to $\kkbar$, we may assume the base
     field $\kk$ is algebraically closed.
     For $\kkbar = \mathbb{C}$ the claim is proven
     in~\cite[Lemma~3.5]{jabu_ginensky_landsberg_Eisenbuds_conjecture};
     the proof generalises word by word to arbitrary algebraically closed fields.
\end{proof}

    It is intuitively clear, that if the Hilbert scheme of $r$ points is
    reducible, also the Hilbert scheme of $r' > r$ points is reducible. For
    further use, we prove it formally below.
    Denote by $\Hilb^{Gor, sm}_r(X) := \Hilb^{Gor}_r(X) \cap \Hilb^{sm}_r(X)$.
    In particular, since $\Hilb^{Gor}_r(X)$ is open and $\Hilb^{sm}_r(X)$ is reduced, 
      it follows that $\Hilb^{Gor, sm}_r(X)$ is reduced.
    \begin{cor}\label{ref:persistence_of_smoothability:cor}
        Let $X$ be a projective scheme such that $\Hilb^{sm}_r(X) =
        \reduced{(\Hilb_{r}(X))} \neq \emptyset$. Then $\Hilb^{sm}_k(X) =
        \reduced{(\Hilb_{k}(X))}$ for all $k\leq r$.
        Also, if $\Hilb^{Gor, sm}_r(X) =
        \reduced{(\Hilb^{Gor}_{r}(X))} \neq \emptyset$ then $\Hilb^{Gor,sm}_k(X) =
        \reduced{(\Hilb^{Gor}_{k}(X))}$ for all $k\leq r$.
    \end{cor}
    \begin{proof}
        By Proposition~\ref{prop_base_change_for_Hilb_r} and Lemma~\ref{ref:GorensteinLocus:lem}\ref{it:Gorlocthird} it is enough to prove the statement 
          for algebraically closed $\kk=\kkbar$.
        From $\Hilb_r^{sm}(X) \neq\emptyset$ we see that $X$ has at least $r$ distinct closed points.
        Take any $k \leq r$, any $\kk$-rational point $[R]\in \Hilb_k(X)$ 
          and a superscheme $R'\supset R$ consisting of $R$ 
          and a tuple of $r-k$ closed points distinct from the support of $R$ 
          (which are all $\kk$-rational, since $\kk$ is algebraically closed).
        By the assumption $[R']\in \Hilb^{sm}_{r}(X)$, hence $R'$ is smoothable. 
        Then  Corollary~\ref{ref:smoothingcomponentsresult:cor} implies that $R$
        is also smoothable, hence $[R]\in \Hilb^{sm}_k(X)$. The schemes
        $\Hilb^{sm}_k(X) \subset \reduced{(\Hilb_k(X))}$ are reduced, finite type
        and have the same $\kkbar$-points, so they are equal.
        The proof for
        Gorenstein case is the same, just note that adding points to
        a Gorenstein scheme $R$ gives a \emph{Gorenstein} scheme $R'$.
    \end{proof}

\subsection{Relative linear span}\label{sec_relative_linear_span}

Now we discuss linear spans of various points of the Hilbert scheme. They
enable us to sensibly define secant varieties of $X$ when $X$ does not have
enough $\kk$-rational points, see Section~\ref{sect_defining_secant_and_cactus}.

Assume $X \subset \PP V\simeq \PP_{\kk}^n$ is an embedded projective scheme and $r\ge 1$ is an integer.
Consider the Hilbert scheme $\Hilb_r(X)$.
If $[R] \in \Hilb_r(X)$ is a $\kk$-rational point, then it corresponds to $R
\subset \PP V$ and its linear span $\sspan{R} \subset \PP V$ is the smallest
projective subspace containing $R$, as in the construction described in Section~\ref{sec_embedded_proj_geom}.

More generally, let $y \in \Hilb_r(X)$ be a point with residue field $\kappa(y) = \KK$.
It corresponds to a closed subscheme $Y \subset \PP V_{\KK} = \PP V \times \Spec \KK$.
We can take its $\KK$-linear span i.e.~the smallest projective subspace $\PP^k_{\KK}$
containing $Y$. We denote this $\KK$-linear span by $\sspan{Y}_{\KK} \subset
\PP V_{\KK}$.
Note that the $\KK$-dimension of $\sspan{Y}_{\KK}$ is at most $r-1$.
In fact, this dimension can be read directly from the structure of the Hilbert scheme at $y$,
   see Section~\ref{sect_universal_ideal_and_graded_pieces}.
Finally, we take the images of $\KK$-linear spans inside our fixed space $\PP V$.
\begin{defn}\label{defn_cactus_of_point}
For $r$, $X$, $y$, and $Y$ as above, 
  the \emph{relative linear span} $\cactus{r}{X,y}$
  is the scheme theoretic image of $\sspan{Y}_{\KK}$ under the projection $\PP (V \otimes_{\kk} \KK) \to \PP V$.
\end{defn}

Note that $\cactus{r}{X,y}$ is not necessarily a linear subspace of $\PP V$.
Informally, one can think of $\cactus{r}{X,y}$ as the closure of the union of
linear spans of schemes in the family $\overline{\set{y}} \subset \Hilb_{r}(X)$.

\begin{remark}
    In the initial case, when $y=[Y] \in \Hilb_r(X)$ is a $\kk$-rational point, we
have $\cactus{r}{X, y} = \sspan{Y}$.
More generally, if $y = [Y]\in\Hilb_r(X)$ is a \emph{closed} point, then it is a $\KK$-rational point for some \emph{finite} field extension $\kk \subset \KK$.
Therefore, the projection morphism $\PP V_{\KK} \to \PP V$ is also finite, hence closed.
It follows that $\cactus{r}{X, y}$ is the set-theoretic image of $\sspan{Y}_{\KK}$; no closure is necessary.
\end{remark}

\begin{remark}
    There are is another possible way of defining something resembling a linear span of a
    $\KK$-rational point $y$. Namely, we first project $Y \subset \PP V_{\KK}$
    to $\PP V$ obtaining $Y'$ and then take the smallest linear subspace
    $\sspan{Y'} \subset \PP V$, containing $Y'$. We call this the \emph{naive
    linear span}.
  The linear space obtained in this way tends to be very large.
  For example, if $r = 1$ and $y\in \Hilb_1(X) = X$ is the generic point,
  then the image of $Y$ in $\PP V$ is equal to $X$ and so $\sspan{Y'} =
  \sspan{X}$.
\end{remark}

We stress that the notions of linear span heavily depends on the chosen embedding of $X \subset \PP V$.
If we consider also a different embedding of $X$, say the $i$-th Veronese reembedding $\nu_i(X) \subset \PP (\DPV{d})$, 
  then the $\KK$-linear span and the relative linear span are denoted, respectively:
\[
   \sspan{\nu_i(Y)}_{\KK}, \qquad \cactus{r}{\nu_i(X),y}.
\]

\begin{example}\label{ex:sspans}
   Suppose $r=1$ and $X \subset \PP^n_{\kk}$ is a subscheme.
   Then $\Hilb_1 (X) = X$.
   Pick a point $y \in \Hilb_1(X)$ and consider its closure $\overline{Y}$ as a closed reduced irreducible subvariety in $X$.
   Let $\KK$ be the residue field of $y$. Then
   \begin{enumerate}
    \item The naive linear span of $y$ is just the linear span of $\overline{Y} \subset \PP V$.
    \item The $\KK$-linear span of $y$ is the single $\KK$-rational point of $X\times\KK$ corresponding to $y$.
    \item The relative linear span of $y$ is $\overline{Y}$.
   \end{enumerate}
\end{example}

\begin{example}
   Suppose $r=2$ and $X \subset \PP^2_{\kk}$ is a subscheme defined by $I(X) =
   (\alpha^2, \alpha\beta)$. As in
   Example~\ref{ex_Hilb2_of_a_line_with_an_embedded_pt},
   the reduced subscheme of $\Hilb_2 (X)$ has two irreducible components
   isomorphic to $\PP^2_{\kk}$ and $\PP^1_{\kk}$ respectively.
   \begin{enumerate}
    \item  Pick a point $y \in \PP^2_{\kk}$ from the first component.
           Then its naive and relative linear span coincide and are equal to
           $\PP_{\kk}^1 \subset X$, the unique line in $X$.
           If $\KK$ is the residue field of $y$, then the $\KK$-linear span is always equal to $\PP_{\KK}^1 = \PP_{\kk}^1 \times \Spec \KK$.
    \item  Now pick a point $y\in \PP^1_{\kk}$ in the latter component of $\Hilb_2(X)$.
   If $y$ is a closed point with residue field $\KK$ (which is a finite extension of $\kk$), then its naive, $\KK$-, or relative linear span is a $\PP^1_{\KK}$ through the point $\alpha=\beta=0$
     and pointing out in the direction represented by $y$.
   If $y$ is a generic point of $\PP^1_{\kk}$ with residue field $\kappa(y) =
   \kk(t)$ with $t = \alpha/\beta$, then its naive linear span and relative linear span are the whole $\PP^2_{\kk}$.
   Its $\kk(t)$-linear span is a $\PP^1_{\kk(t)}\subset \PP^2_{\kk(t)}$ defined by the linear equation $\alpha -  t \beta =0$.
   \end{enumerate}
\end{example}

We also define the relative linear span for a family $T \to \Hilb_r(X)$, where $T$ is a reduced scheme.
\begin{defn}\label{defn_cactus_of_family}
    Let $T$ be a reduced scheme with a morphism $\phi\colon T \to \Hilb_r(X)$.
    Then the \emph{relative linear span of $T$} (or, strictly speaking, of
    $\phi:T\to \Hilb_r(X)$) is
    \begin{equation}   \label{eq:cactus_of_family}
        \cactus{r}{X, \phi(T)} := \overline{\bigcup \left\{ \cactus{r}{X,
        \phi(\eta)}\ |\ \eta\text{ is a generic point of a component of } T \subset \PP V\right\}}.
    \end{equation}
\end{defn}
If $T$ is affine or projective (or more generally quasi-compact), there are
finitely many components of $T$ and thus there are finitely many $\eta$,
so that the union in \eqref{eq:cactus_of_family} is closed in $\PP V$, no closure is needed.
If $T= \Spec \KK$ corresponds to a point $t\in \Hilb_r(X)$, then $\cactus{r}{X, \phi(T)} =
\cactus{r}{X, t}$, i.e.~Definitions~\ref{defn_cactus_of_point} and~\ref{defn_cactus_of_family} agree.
Note that in the definition we only use generic points of the components of $T$.
We will see in Proposition~\ref{prop_relative_span_of_a_point_contained} that  $\cactus{r}{X, \phi(T)}$ 
  contains also the relative linear span of all points in $T$.
Example~\ref{ex_cacti_for_families_beware} illustrates that this is nontrivial, as a naive approach gives an incorrect answer.
Also reducedness of $T$ is essential, for reasons analogous to  Example~\ref{ex_cacti_for_families_beware}.
Defining spans of nonreduced families will be an object of future research.

\subsection{Graded pieces of universal ideal and relative linear span}\label{sect_universal_ideal_and_graded_pieces}

In this section we relate the universal ideal sheaf on the Hilbert scheme 
  to the notion of linear span from Section~\ref{sec_relative_linear_span}. 

Assume $X \subset \PP V\simeq \PP_{\kk}^n$ is an embedded projective scheme.
Then the universal family $\ccU_r \subset \Hilb_r(X) \times X \subset \Hilb_r(X) \times \PP V$ 
  determines a homogeneous ideal sheaf $\ccI_{\ccU_{r}} \subset \ccO_{\Hilb_r(X)} \otimes \Sym V^*$ on $\Hilb_r(X)$.
Here \emph{homogeneous} means that 
\[
  \ccI_{\ccU_{r}} = \bigoplus_{i=0}^{\infty} \ccI_{\ccU_r} \cap \left(\ccO_{\Hilb_r(X)} \otimes S^i V^*\right) = \bigoplus_{i=0}^{\infty} (\ccI_{\ccU_r})_{i}.
\]
Each $(\ccI_{\ccU_r})_i$ is a coherent sheaf on $\Hilb_r(X)$.
These sheaves lead to an extension of the Hilbert function to the Hilbert scheme.
More precisely, let $p\in \Hilb_r(X)$ be a point (not necessarily a $\kk$-rational point).
Suppose the residue field of $p$ is $\kappa(p)$.
Consider $(\ccI_{\ccU_r})_i \otimes \kappa(p)$, which is an abbreviation for
   restricting the sheaf $(\ccI_{\ccU_r})_i$ to a module over the local ring
   $\ccO_{\Hilb_r(X), p}$ and then tensoring it with $\kappa(p)$ as a $\ccO_{\Hilb_r(X), p}$-module.
This is a $\kappa(p)$-vector subspace in $S^i V^*\otimes \kappa(p)$. 
Denote its codimension  by $H_{p}(i)$ and call the function 
\[
   H_{\bullet}(\bullet) \colon \Hilb_{r}(X) \times \ZZ^{\ge 0} \to \ZZ^{\ge 0}
\]
the \emph{Hilbert function}.
The function $H_p$ is, by definition (Section~\ref{sec_embedded_proj_geom}),
the Hilbert function of the scheme $R \subset X \times \Spec \kappa(p) \subset \PP V \times
\Spec \kappa(p)$ corresponding to $p$.
In particular:
\begin{prop}
   Suppose $R \subset \PP V$ is a finite subscheme of degree $r$.
   Then $H_{R \subset \PP V}(i) = H_{[R]}(i)$.\qed
\end{prop}

More generally, the Hilbert function determines the dimension of the
$\kappa(p)$-linear span discussed in Section~\ref{sec_relative_linear_span}:
\[
   \dim_{\kappa(p)} \sspan{\nu_i(P)}_{\kappa(p)} = H_{p}(i)  -1,
\]
where $P\subset \ccU_r$ is the fibre over $p$ and where $\nu_i(P)$ is the
Veronese reembedding of $P$ into $\PP (\DPV{i}\otimes_{\kk} \kappa(p)) =\PP
(\operatorname{DP}^i(V\otimes_{\kk} \kappa(p)))$.
In particular, 
\[
   \dim_{\kappa(p)} \sspan{P}_{\kappa(p)} = H_{p}(1)  -1.
\]

Note that $H_p(i)$ is the codimension of $(I_P)_i$. The dimension of $(I_P)_i =
(\mathcal{I}_{\mathcal{U}_r})_{i|p}$ is upper semicontinuous
by~\cite[28.1.1]{vakil_FoAG}. Therefore, the function $H_{p}(i)$ is lower
semicontinuous, i.e.~the subsets
\begin{equation}\label{eq:openness_for_max_Hf}
   \set{p\in \Hilb_r(X) \mid H_{p}(i) \ge k}
\end{equation}
are open for all integers $i$ and $k$.
In particular, for every irreducible component $\ccH\subset \Hilb_r(X)$, 
  we have a unique value $k_{\max}(\ccH)$ such that
  $\ccH^{nondeg} := \set{p\in \ccH \mid H_{p}(i) = k_{\max}(\ccH)}$ is open and nonempty (hence also dense in $\ccH$).
We always have $ H_{p}(i) \le k_{\max}:=\min \set{r , H_{X \subset \PP V} (i)}$, 
  hence $k_{\max}(\ccH) \le  k_{\max}$, but the inequality might be strict. 
Moreover, if $i \ge r-1$, then $H_{p}(i) = r$ for all points $p\in \Hilb_r(X)$, i.e.~the function $H_{\bullet} (i)$ is constant, by Corollary~\ref{cor_finite_schemes_r_regular}. 

In Section~\ref{sec_relative_linear_span} the relative linear span of a point $p \in \Hilb_r(X)$ was defined using $\KK$-linear span, 
  and the latter comes from a single grading of a restriction of the ideal sheaf.
We would like to generalise this approach to families $\phi\colon T \to \Hilb_r(X)$.
The family from Example~\ref{ex_cacti_for_families_beware} shows that this is not always possible.
In the following, we show that it \emph{is} possible in a special case 
   when $\dim_{\kappa(t)} \sspan{\phi(t)}_{\kappa(t)}$ is  constant,
        independent of $t\in T$.
Using the semicontinuity of the Hilbert function we may shrink any $T$ 
to a dense open subset $T'$ satisfying this assumption, see the discussion before~\eqref{eq:openness_for_max_Hf}.

To explain this, we fix an integer $i>0$ and a family $\phi\colon T \to \Hilb_r(X)$ where $T$ is a reduced scheme.
We are going to look at $\cactus{r}{\nu_i(X), \phi(T)}$.
In some statements below we might need another family $\phi' \colon T'\to \Hilb_r(X)$.
All the constructions for $T'$ will be decorated with a ``prime''~$'$: for instance $\ccE$ will become $\ccE'$ etc.

Consider the pullback $\ccE:=\phi^*(\ccI_{\ccU_r})_i$ of the sheaf
$(\ccI_{\ccU_r})_i$ to $T$.
It is a subsheaf of a trivial sheaf $\ccE\subset \ccO_{T} \otimes S^i V^*$.
It generates an ideal sheaf in $\ccO_{T} \otimes \Sym(S^i V^*)$, and dually, 
  it determines a subscheme $\ccE^{\perp} \subset T \times \DPV{i}$.
The scheme $\PP(\ccE^{\perp})$ could be seen as a generalisation of the $\KK$-linear span.
Indeed, if $T = \Spec \kk$ and $[R] = \phi(T)$ is a $\kk$-rational point corresponding to the finite subscheme $R \subset X$,
   then the fibre over $[R]$ is the linear span $\sspan{\nu_i(R)} \subset \PP
   (\DPV{i})$ of the Veronese reembedding of $R$.
More generally, if $T = \Spec \KK$  and $y = \phi(T)$ is a $\KK$-rational
point, then $\PP(\ccE^{\perp}) = \sspan{\nu_i(Y)}_{\KK}$.

\begin{prop}\label{prop_relative_lin_span_for_locally_constant_Hf_bundle}
    Suppose $T$ is a reduced scheme with a morphism $\phi\colon T \to \Hilb_r(X)$, 
       such that the Hilbert function $t\to H_{\phi(t)}(i)$
       is constant (independent of $t$).
    Assume $\ccE$ and $\ccE^{\perp}$ are as above.
    Then:
    \begin{enumerate}
     \item for every $y \in \phi(T)$ we have $\cactus{r}{\nu_i(X), y} \subset \cactus{r}{\nu_i(X), \phi(T)}$,
     \item \label{item_rel_lin_span_for_locally_constant_Hf_closure_of_the_image}
       $\cactus{r}{\nu_i(X), \phi(T)}$ is the closure of the image of $\PP(\ccE^{\perp})$ under the projection
       $T \times \PP(\DPV{i})\to  \PP(\DPV{i})$.
    \end{enumerate}
\end{prop}
\begin{proof}
    The sheaf $\ccE$ is locally free, since the dimensions over any point are constant by our choice of $T$.
    Hence the point-wise perpendicular subset $\ccE^{\perp} \subset T \times \DPV{i}$ agrees with the kernel of the dual vector bundle map $\ccO_{T} \otimes \DPV{i} \to \ccE^*$, 
    and, in particular, is a vector bundle itself.
    Note that since $T$ is reduced, also the total space of the vector bundle 
      $\ccE^{\perp}$ and its projectivisation $\PP(\ccE^{\perp})$ is reduced.
    
    For $y\in \phi(T)$ let $t \in T$ be such that $\phi(t) = y$. Let $T' =
    \Spec \kappa(y)$ with a map $\phi'\colon T'\to \Hilb_r(X)$.
    Then $\phi|_{\set{t}}$ factorises as $T \to T' \to \Hilb_r(X)$ and $\sspan{Y}_{\kappa(y)} = \PP ((\ccE')^{\perp})$ 
       pulls back to the fibre $\PP(\ccE^{\perp}|_{t})$ over $t$
       (here we heavily exploit that $\ccE^{\perp}$ is a vector bundle, which is a consequence of the constant Hilbert function assumption; 
       otherwise the proof would fail here, see Example~\ref{ex_cacti_for_families_beware}).
    Thus the image of  $\PP(\ccE^{\perp})$ under the projection contains the image of $\PP(\ccE^{\perp}|_{t})$,
       which is equal to the image of  $\PP ((\ccE')^{\perp})$ under the projection $T' \times \PP(\DPV{i})\to  \PP(\DPV{i})$.
    That is, $\cactus{r}{\nu_i(X), y} \subset \cactus{r}{\nu_i(X), \phi(T)}$,
    as claimed in the first item.
    
    Assume $S \subset T$ is the set of generic points of components of $T$.
    Hence, restricting the first item to the images of  points from $S$,
        we obtain that  $\cactus{r}{\nu_i(X), \phi(T)}$  is contained in the image of $\PP(\ccE^{\perp})$.
    To conclude, the points from $S$ are dense in $T$, hence the union of $\PP(\ccE^{\perp}_{\eta})$ for $\eta \in S$ is dense in 
        $\PP(\ccE^{\perp})$ and the image of this union is dense in the image of  $\PP(\ccE^{\perp})$.
    But the closure of $\bigcup_{\eta \in S} \PP(\ccE^{\perp}_{\eta})$ is by the above arguments equal to $\cactus{r}{\nu_i(X), \phi(T)}$,
        which finishes the proof.
\end{proof}

\begin{prop}\label{prop_properties_of_relative_linear_span}
Assume $X \subset \PP V\simeq \PP_{\kk}^N$ is an embedded projective scheme, 
  and suppose $T$ and $T'$ are two reduced schemes, both with a morphism $\phi$ or $\phi'$ mapping them into a subset of $\Hilb_r(X)$ with constant $H_{\bullet}(i)$.
  \begin{enumerate}
   \item \label{item_comparing_relative_spans_inclusion}
         If $\phi'$ factorises through $T' \to T\to \Hilb_r(X)$, 
              then $\cactus{r}{\nu_i(X),\phi'(T')} \subset \cactus{r}{\nu_i(X),\phi(T)}$.
   \item \label{item_comparing_relative_spans_dense}
         If $\phi'$ factorises through $T' \to T\to \Hilb_r(X)$, and the image of $T'$ is dense in $T$, 
              then $\cactus{r}{\nu_i(X),\phi'(T')} = \cactus{r}{\nu_i(X),\phi(T)}$.
   \item \label{item_comparing_relative_spans_dense_subset_point}
         If $S\subset T$ is a dense subset of $T$
           then
           \[
             \overline{\bigcup_{y\in S} \cactus{r}{\nu_i(X),\phi(y)}} = \cactus{r}{\nu_i(X),\phi(T)}.
           \]
   \item \label{item_comparing_relative_spans_closed}
         If the map $\phi\colon T\to \Hilb_r(X)$ is proper (for example, $T$ is a closed subset of $\Hilb_r(X)$),
            then the map 
            \[
               \PP(\ccE^{\perp}) \to \cactus{r}{\nu_i(X),\phi(T)}
            \]
            is surjective, i.e. no closure is needed in 
            the second item of Proposition~\ref{prop_relative_lin_span_for_locally_constant_Hf_bundle}.
            (We stress, that we still assume $\phi(T)$ is contained in the locus with a constant Hilbert function.)
  \end{enumerate}
\end{prop}
\begin{proof}
   To see point \ref{item_comparing_relative_spans_inclusion}, note that the factorisation also holds on the level of respective bundles  $\PP(\ccE^{\perp})$ and $\PP((\ccE')^{\perp})$: 
     they are pullbacks of the same bundle to $T$ or $T'$ respectively.
   Hence the inclusion persists to the images, i.e. $\cactus{r}{\nu_i(X),\phi'(T')} \subset \cactus{r}{\nu_i(X),T}$ 
     by Proposition~\ref{prop_relative_lin_span_for_locally_constant_Hf_bundle}.
   Also, if the image of  $T'$ is dense in $T$, then the respective bundle is also dense, hence \ref{item_comparing_relative_spans_dense} holds.
   Point~\ref{item_comparing_relative_spans_dense_subset_point} follows from \ref{item_comparing_relative_spans_dense}
     by taking $T' = \bigsqcup_{y \in S} \set{y}$.
   Point~\ref{item_comparing_relative_spans_closed} is also immediate, as $\PP(\ccE^{\perp})$ is proper, 
     thus the dominant morphism $\PP(\ccE^{\perp}) \to \cactus{r}{\nu_i(X),\phi(T)}$ is also proper, hence closed.
\end{proof}

Now we turn our attention back to a general situation, when the Hilbert function is not necessarily constant.
The next statement can be seen informally as a statement that \emph{the linear span of a limit is contained in the limit of linear spans.}

\begin{prop}\label{prop_relative_span_of_a_point_contained}
    Let $T$ be a reduced scheme with a morphism $\phi\colon T \to \Hilb_r(X)$. 
    Then for every $y \in \phi(T)$ we have $\cactus{r}{\nu_i(X), y} \subset \cactus{r}{\nu_i(X), \phi(T)}$.
\end{prop}
\begin{proof}
    By Definition~\ref{defn_cactus_of_family}, for any irreducible component $T'$ of $T$ we have  
        $\cactus{r}{\nu_i(X), \phi(T')} \subset \cactus{r}{\nu_i(X), \phi(T)}$.
    Thus we can replace $T$ with its irreducible component, whose image contains $y$.
    That is, we assume $T$ is irreducible.
    
    Suppose $U \subset T$ is the open subset of $T$ such that the Hilbert function $H_{\bullet}(i)$ on $\phi(U)$ is constant.
    By Proposition~\ref{prop_properties_of_relative_linear_span}\ref{item_comparing_relative_spans_inclusion} 
       for any $u\in U$ there is an inclusion $\cactus{r}{\nu_i(X), \phi(u)} \subset \cactus{r}{\nu_i(X), \phi(T)}$.
    Thus by Lemma~\ref{lem_finding_a_curve_through_point_and_open_subset},
       we can replace $T$ with a germ of a curve, i.e.~$T:=\Spec A$ for a one dimensional Noetherian complete local domain $A$.
       
    Let $t\in T$ be the closed point and let $\eta\in T$ be the generic point.
    We must show, that $\cactus{r}{\nu_i(X), \phi(t)} \subset \cactus{r}{\nu_i(X), \phi(\eta)}$.
    Let $T'\to T$ be the normalization of $T$. 
    The map $T'\to T$ is finite~\cite[Appendix 1, Cor
    2]{nagata_algebraic_geo_over_Dedekind_two} and dominant, so there are
    points $t'$ and $\eta'$ mapping to $t$ and $\eta$ respecitvely.
    By Proposition~\ref{prop_properties_of_relative_linear_span}\ref{item_comparing_relative_spans_dense}       
       we have equalities of relative linear spans
       $\cactus{r}{\nu_i(X), \phi(t)}= \cactus{r}{\nu_i(X), \phi'(t')}$ 
       and $\cactus{r}{\nu_i(X), \phi(\eta)}=\cactus{r}{\nu_i(X), \phi'(\eta')}$.
    Thus we may replace $T$ with $T'$, i.e.~assume $T$ is normal.
    
    The sheaf $\ccE$ on $T$ is torsion free, hence locally
    free~\cite[Corollary~6.3]{eisenbud}.
    Thus $\PP(\ccE^{\perp}) \subset T\times \PP(\DPV{i})$ is a vector bundle.
    The relative linear span $\cactus{r}{\nu_i(X), \phi(\eta)}$ is the closure of the image of the generic fibre $\PP(\ccE^{\perp}_{\eta})$.
    Thus it is enough to show that the special fibre $\PP(\ccE^{\perp}_{t})$ contains the pullback of the linear span $\sspan{Y}_{\kappa(y)}$.
    But the latter is defined by the $i$-th grading of the saturation of $\phi^*(\ccI_{\ccU_r})|_{t}$ which contains $\ccE_{t}$, 
      proving the claim.
\end{proof}

We can now generalise Proposition~\ref{prop_properties_of_relative_linear_span}\ref{item_comparing_relative_spans_inclusion}--\ref{item_comparing_relative_spans_dense_subset_point} to the case when the family has non-constant Hilbert function.

\begin{cor}\label{cor_cactus_equal_to_closure_over_points}
    Let $T$ be a reduced scheme with a morphism $\phi\colon T \to \Hilb_r(X)$. 
    Suppose $S \subset T$ is a dense subset. 
    Then
    \[
        \cactus{r}{\nu_i(X), \phi(T)} = \overline{\bigcup_{t\in S} \cactus{r}{\nu_i(X), \phi(t)}}.
    \]
\end{cor}
\begin{proof}
    The inclusion $\supset$ follows from Proposition~\ref{prop_relative_span_of_a_point_contained}.
    To see $\subset$, pick an irreducible component $T'$ of $T$.
    It is enough to show that $\cactus{r}{\nu_i(X), \phi(T')}$ is contained in $\overline{\bigcup_{t\in S} \cactus{r}{\nu_i(X), \phi(t)}}$.
    Replace $S$ with the intersection of $S$, and an open dense subset of $T'$, which has constant Hilbert function $H_{\bullet}(i)$.
    Then the claim follows from Proposition~\ref{prop_properties_of_relative_linear_span}\ref{item_comparing_relative_spans_dense_subset_point}.
\end{proof}

Using Corollary~\ref{cor_cactus_equal_to_closure_over_points} we may in turn
strengthen Proposition~\ref{prop_relative_span_of_a_point_contained} by adding points in the closure of $\phi(T)$.
\begin{cor}\label{prop_relative_span_in_the_limit}
    Let $T$ be a reduced scheme with a morphism $\phi\colon T \to \Hilb_r(X)$. 
    Then for every $y \in \overline{\phi(T)}$ we have $\cactus{r}{\nu_i(X), y} \subset \cactus{r}{\nu_i(X), \phi(T)}$.
\end{cor}
\begin{proof}
    Let $T' \subset \Hilb_r(X)$ be the scheme-theoretic image of $\phi$. Then
    the set $\phi(T)$ is Zariski-dense in $T'$ and $y\in T'$.
    By Proposition~\ref{prop_relative_span_of_a_point_contained} we have
    $\cactus{r}{\nu_i(X), y} \subset \cactus{r}{\nu_i(X), T'}$ and
    Corollary~\ref{cor_cactus_equal_to_closure_over_points} implies that
    $\cactus{r}{\nu_i(X), T'} = \overline{\bigcup_{t\in T}\cactus{r}{\nu_i(X),
    \phi(t)}} = \cactus{r}{\nu_i(X), \phi(T)}$.
\end{proof}

\begin{example}\label{ex_cacti_for_families_beware}
    Fix coordinates $\alpha, \beta, \gamma, \delta$ on $\PP^3_{\kk}$.
    Consider $T = \Spec \kk[s, t]/(st)$ and a morphism $\phi:T\to \Hilb_3 \PP^3_{\kk}$
    which sends the point $(s_0, t_0)\in \Spec T$ to the reduced subscheme
    \[[1,1,0,0] \sqcup [0,1,0,0] \sqcup [1,0,s,t].\]
    Formally, $\phi$ is induced by a family $\mathcal{U}_{T} \subset \PP^3_{\kk}
    \times T$ given by the homogeneous ideal $I = (\alpha-\beta, \gamma, \delta) \cap (\alpha,
    \gamma, \delta) \cap (\beta, \gamma - s\alpha, \delta - t\alpha) \subset
    \kk[\alpha, \beta, \gamma, \delta]\tensor \kk[s, t]/(st)$.
    Now, the sheaf $\ccE$ consists of all linear elements of the saturation of $I$.
    But the only linear elements are $t\gamma$ and $s\delta$. 
    Thus $\PP (\ccE^{\perp}) \subset T \times \PP_{\kk}^3$, 
       the scheme defined by $\ccE = (t\gamma,s\delta)$, maps surjectively onto $\PP^3_{\kk}$.
    Instead, the relative linear span of $\phi(T)$ is a union of two
    $\PP^2_{\kk}$.
    Therefore, unlike in Proposition~\ref{prop_relative_lin_span_for_locally_constant_Hf_bundle}\ref{item_rel_lin_span_for_locally_constant_Hf_closure_of_the_image}, 
      or in the proof of Proposition~\ref{prop_relative_span_of_a_point_contained}, 
      the relative linear span $\kappa_3(\PP V, T)$ is \emph{not} equal to the image of $\PP(\ccE^{\perp})$.
    
    Geometrically, $T$ is a union of lines $s=0$ and $t=0$ intersecting at a
    point $s=t=0$ (in particular $T$ is not smooth at this point). The span of
    any smooth point on the line $s = 0$ is equal to $V(\gamma) = \sspan{[1,1,0,0],
    [0,1,0,0],[1, 0,0,1]}$ and the span of any smooth point on the line $t = 0$
    is equal to $V(\delta) = \sspan{[1,1,0,0], [0,1,0,0], [1,0,1,0]}$. Those
    are different $\PP^2_{\kk}$ and so in the intersection point $s=t=0$ they span
    the whole $\PP^3_{\kk}$. There is no candidate for a fibre of $\ccE^{\perp}$ over $s=t=0$, which
    would make the above family of projective subspaces flat.
\end{example}

The invariance of the Hilbert scheme and its universal family under base
change (Section~\ref{sec_Hilb_base_change}) implies that also
  the notion of relative linear span is invariant under base change in the following sense.

\begin{prop}\label{prop_base_change_for_relative_linear_span}
   Suppose $X \subset \PP V\simeq \PP_{\kk}^n$ is an embedded projective scheme,
      and $\phi\colon T \to \Hilb_r(X)$ is a morphism from a reduced scheme $T$.
   Fix a field extension $\kk \subset \KK$.
   Denote $X_{\KK} := X \times \KK \subset \PP(V \otimes \KK)$,
      $T_{\KK}:= \reduced{(T \times \KK)}$ and $\phi_{\KK}\colon T_{\KK} \to \Hilb_r(X_{\KK}/\KK)$.
   Then
   \[
     \reduced{(\cactus{r}{\nu_i(X),\phi(T)} \times \KK)}  = \cactus{r, \KK}{\nu_i(X_{\KK}),\phi_{\KK}(T_{\KK})},
   \]
     where the right hand side stands for the relative linear span of $\phi_{\KK}(T_{\KK})$ over the field $\KK$, i.e. 
     the subset of $\PP (V\otimes \KK)\simeq \PP_{\KK}^n$ defined as before from  $X_{\KK}$, but with the role of $\kk$ replaced by $\KK$.\qed
\end{prop}

\section{Secant varieties}\label{sect_secants}

In elementary geometry, for a curve $C$ contained in a plane $\RR^2$, a \emph{secant line} is obtained by choosing two distinct points $x_1, x_2 \in C$ 
  and drawing an affine line through $x_1$ and $x_2$.
If the curve is smooth, and the two points converge one to the other, then in the limit we get a \emph{tangent line}, as illustrated on Figure~\ref{fig_secant_and_tangent}.
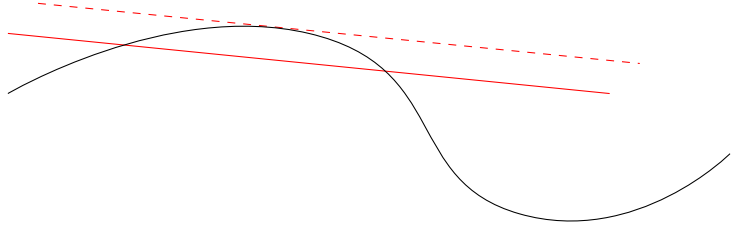
\begin{figure}[htb]
\begin{center}
\begin{tikzpicture}[scale=0.8]
\draw [red] (0, 1) -- (10, 0);
\draw [red, dashed] (0.50, 1.50) -- (10.50, 0.50);
\draw [black] plot [smooth, tension=1] coordinates { (0,0) (5,1)
(8.5,-2) (12,-1)};
\end{tikzpicture}
\end{center}
\caption{Secant and tangent line to a plane curve}\label{fig_secant_and_tangent} 
\end{figure}

The concept of secants in algebraic geometry generalises the above elementary picture.
Instead of working with a curve on a plane $C \subset \RR^2$, we usually work with a subvariety in a projective space $X \subset \PP^n_{\kk}$,
  but affine or linear approaches are also possible, see \cite{jabu_januszkiewicz_jelisiejew_michalek_k_reg}. 

\subsection{Defining secant and cactus varieties}\label{sect_defining_secant_and_cactus}

Fix a base field $\kk$ and an embedded projective variety, or more generally, an embedded projective subscheme $X \subset \PP V$.
The \emph{$r$-th secant variety} of $X$ (strictly speaking, of $X \subset \PP V$) is a well established object in algebraic geometry, provided $\kk$ is algebraically closed.
Roughly, it is a  ``closure'' of the union of all secant linear subspaces $\sspan{\fromto{x_1}{x_r}}$ for points $x_i \in X$.
This is a definition that often appears in articles on secant varieties over the complex field $\CC$, 
   and the authors really mean only closed points $x_i \in X$, which is a standard abuse of notation.
Since in this article we want to consider any base field, we must be more careful.

Define \emph{an $r$-th secant} of $X$ in the following way.
We say that a projective linear subspace $\PP_{\kk}^{r-1} \subset \PP V$ is an $r$-th secant of $X$, if $\PP_{\kk}^{r-1}$ is spanned by $\kk$-rational points of $X$.
In informal discussions, this is sometimes also called an \emph{honest $r$-th secant} or \emph{pure $r$-th secant}, to stress the difference between such secant and the limiting linear spaces (such as the tangent lines above).
In particular, $1$-st secants of $X$ are in one-to-one correspondence with $\kk$-rational points of $X$.

The exact definition of the secant variety  depends on the applications, one has in mind.
However, independent of the applications, the readers will certainly agree that the $r$-th secant variety $\sigma_r(X)$ should contain all (pure) $k$-th secants (for $k\le r$):
\begin{equation}\label{equ_secant_variety_and_secants}
    \sigma_r(X) \supset \bigcup \set{\sspan{\fromto{x_1}{x_r}} :  \text{each $x_i$ is a $\kk$-rational point of $X$}}.
\end{equation}

If the base field $\kk$ is algebraically closed, then the closed points of the right hand side agree with the closed points of a \emph{constructible} set,
   i.e. a set obtained as a finite union and intersection of open and closed subsets of $\PP V$.
In this sense, it is already very close to an algebraic variety and the usual definition of $\sigma_r(X)$ is just the Zariski closure of the right hand side. 
If $\kk=\CC$, then equally well one may take the Euclidean closure of the closed points of the right hand side to obtain the set of closed points of $\sigma_r(X)$.

\begin{remark}
If $\kk=\RR$, then the Euclidean closure of the $\RR$-rational points in the union on the right hand side above is a \emph{semialgebraic subset} 
  of a topological real projective space.
This means, the set is defined as a solution set of a bunch of polynomial \emph{inequalities}.
This closure is relevant in the studies of the notion of real rank (or $X$-rank), an important notion from the point of view of applications.
This is a delicate and complicated subject and very little is known about the structure of such closure even in very special cases.
See~\cite{blekherman_Typical, michalek_spaces_of_sums_of_powers,
bernardi_blekherman_ottaviani} for a sample of recent work in this direction.
We will not discuss further the semialgebraic approach, as it does not fit in the scope of our survey.
\end{remark}

If $\kk$ is not algebraically closed, then the union of secants on the right hand side of \eqref{equ_secant_variety_and_secants} might be very small, even empty, 
  which is not so interesting.
We need a more refined notion of a linear span of general $r$ points of $X$.
It is not a surprise that we use the Hilbert scheme and the notion of relative linear span
  (as in Sections~\ref{sec_relative_linear_span} and \ref{sect_universal_ideal_and_graded_pieces}) for this purpose.

\begin{example}
   Suppose $X=\nu_4(\PP_{\kk}^1) \subset \PP_{\kk}^4$ is the fourth Veronese embedding of $\PP_{\kk}^1$.
   Then let $\ccH:=\Hilb_2(X) \simeq \PP_{\kk}^2$, with $\ccU_2 \simeq X \times X = \PP_{\kk}^1 \times \PP_{\kk}^1$
      with the map $\PP_{\kk}^1 \times \PP_{\kk}^1 \to \PP_{\kk}^2$ equal to the quotient by the $\ZZ_2$-action swapping the factors.
   The map $\ccU_2 \to \PP_{\kk}^4$ is given by the projection on the first factor followed by $\nu_4$.
   The generic point $\eta$ of $\ccH$ has the residue field $\KK$ isomorphic to $\kk(s,t)$ and 
      the fibre $R_{\eta} \subset \ccU_2$ is isomorphic to $\Spec \KK[x]/(x^2 -sx +t)$, 
   where $x$ and $s - x$ are the coordinates on two factors of $\ccU_{2}$.
   The map $ R_{\eta} \to \PP_{\KK}^4$ is given in coordinates by:
   \[
     x \mapsto [1,x, sx-t, (s^2 - t)x -st , (s^3 - 2st)x- s^2t + t^2]. 
   \]
   Denoting the coordinates on $\PP_{\KK}^4$ by $\alpha_0, \alpha_1, \alpha_2, \alpha_3,\alpha_4$,
     the generators of the ideal of the image of $R_{\eta}$ in $\PP_{\KK}^4$ are
   \begin{gather*}
       t^3 \alpha_0 + (-s^3 + 2st) \alpha_3 + (s^2 - t)\alpha_4, \qquad 
       t^2 \alpha_1 + (-s^2 + t) \alpha_3 + s\alpha_4,\qquad
       t \alpha_2 - s\alpha_3 + \alpha_4,\\
       t \alpha_3^2 - s\alpha_3\alpha_4 + \alpha_4^2
   \end{gather*}
   Only the first three equations are linear in the coordinates $\alpha$, hence the $\KK$-linear span of $\iota(R_{\eta})$
     is the $\PP^1_{\KK}$ defined by the three equations in the first line.
   The relative linear span $\cactus{2}{X, \eta}$ is the projection of  $\sspan{\iota(R_{\eta})}_{\KK}$ and it is obtained by eliminating $s$ and $t$ from the ideal.
   In this way we obtain the usual cubic equation defining the second secant variety $\sigma_2(\nu_4(\PP_{\kk}^1))$:
   \[
          \alpha_2^3 - 2\alpha_1\alpha_2\alpha_3 + \alpha_0\alpha_3^2 + \alpha_1^2\alpha_4 - \alpha_0\alpha_2\alpha_4 =0.
   \]
   This cubic can also be seen, up to sign, as the determinant of the matrix
   $\begin{pmatrix}
       \alpha_0& \alpha_1&\alpha_2\\
       \alpha_1& \alpha_2&\alpha_3\\
       \alpha_2& \alpha_3&\alpha_4
    \end{pmatrix}
   $.
\end{example}

\begin{defn}\label{def_secant_and_cactus_varieties}
  Let $X \subset \PP V$ be an embedded projective scheme, $r$ a positive integer.
  \begin{itemize}
   \item We define the \emph{$r$-th secant variety} $\sigma_r(X)$ of $X$ to be
           the union:
           \[
             \sigma_r(X) := \bigcup_{k=1}^{r} \cactus{k}{X, \Hilb_{k}^{sm}(X)}.
           \]
    \item We define the \emph{$r$-th cactus variety} $\cactus{r}{X}$ to be the union:
           \[
             \cactus{r}{X}:=\bigcup_{k=1}^{r} \cactus{k}{X, \reduced{(\Hilb_{k}(X))}}.
           \]
  \end{itemize}
\end{defn}

As an immediate consequence, we have:
\begin{gather}
    \sigma_r(X) = \reduced{\sigma_r(X)} = \sigma_r(\reduced{X}) \subset \cactus{r}{\reduced{X}} \subset \cactus{r}{X} = \reduced{(\cactus{r}{X})} \subset \sspan{X},\nonumber \\
    \text{and } \sigma_r(X) \subset  \sigma_{r+1}(X),\ \ \cactus{r}{X} \subset
    \cactus{r+1}{X}. \label{equ_inclusions_increasing_r}
\end{gather}
The  inclusions~\eqref{equ_inclusions_increasing_r} are immediate from the definitions, since we took the union over all $1\le k \le r$.
However, we will see in Proposition~\ref{prop_enough_to_take_k_equal_r} that it is in fact enough to consider just one relative linear span, typically, for $k=r$.

In particular, both secant and cactus varieties are reduced, which justifies the \emph{variety} part of the name. 
As already mentioned in Section~\ref{sec_relative_linear_span}, in order to give some nontrivial scheme structure to $\cactus{r}{X}$,
  one needs to sensibly define the relative linear span of a family of schemes parametrised by a nonreduced scheme, which does not fit in the scope of this article.
Both inclusions $\sigma_r(\reduced{X}) \subset \cactus{r}{\reduced{X}} \subset \cactus{r}{X}$  can be strict as illustrated by an elementary Example~\ref{ex_second_and_cactus_for_union_of_3_P1_and_double_pt}.
\begin{example}
   Suppose $r=1$.
   Then $\sigma_1(X) = \cactus{1}{X} = \reduced{X}$. 
\end{example}

\begin{example}\label{ex_second_and_cactus_for_union_of_3_P1_and_double_pt}
   Suppose $r=2$ and $X\subset \PP^4_{\kk}$ is the union of three lines
   through $P:=(0,0,0,0,1)$ and a degree two scheme supported at $P$, so that 
   \[
     I(X) =(\alpha^2, \alpha\beta, \alpha\gamma, \beta\gamma, \alpha\delta, \beta\delta, \gamma\delta).
   \]
   That is, the three lines are given by equations $\alpha=\beta=\gamma=0$, $\alpha=\beta=\delta=0$, and  $\alpha=\gamma=\delta=0$, respectively, 
     and the degree $2$ scheme is defined by the ideal $(\alpha^2, \beta, \gamma, \delta)$.
   Then $\sigma_2(X) = \sigma_2(\reduced{X}) = \PP^2_{\kk} \cup \PP^2_{\kk} \cup \PP^2_{\kk}$,
     where the three planes are given by equations $\alpha=\beta=0$, $\alpha=\gamma=0$ $\alpha=\delta=0$, respectively.
   Instead, $\cactus{2}{X} \ne \cactus{2}{\reduced{X}}$ and
   \begin{align*}
     I(\cactus{2}{X}) &= (\alpha\beta, \alpha\gamma, \alpha\delta) = \PP^3_{\kk} \cup \PP^1_{\kk}\\
     I(\cactus{2}{\reduced{X}}) &= (\alpha) = \PP^3_{\kk}
   \end{align*}
\end{example}

We note that our definition of secant variety agrees with the usual one if the
base field is algebraically closed.

\begin{prop}\label{prop_secant_variety_is_the_closure_of_secants}
   Suppose $X \subset \PP_{\kk}^N$ is an embedded projective variety such that its $\kk$-rational points are Zariski-dense in $X$.
   (For instance, this is always true if $\kk$ is algebraically closed.)
   Then the $r$-th secant variety is equal to the Zariski closure of the union
   of all (pure) $k$-th secants, for all $k \le r$:
   \[
      \sigma_r(X) = \overline{\bigcup \set{\sspan{\fromto{x_1}{x_k}}  \mid x_i \text{ is a $\kk$-rational point of $X$}, k\leq r}}.
   \]
\end{prop}

\begin{proof}
   If $\dim X =0$, then the claim is clear, as the assumptions imply that $X$ is a union of finitely many the $\kk$-rational points.
      Otherwise, on the right hand side of the equality claimed in the proposition, we may assume that all points $x_i$ are pairwise distinct.
   For any positive integer $k$ the set $R:=\setfromto{x_1}{x_k}$, where the points $x_i\in X$ are pairwise distinct and $\kk$-rational,
     is a reduced finite smooth degree $k$ subscheme of $X$.
   Consider the subset $S_k \subset \Hilb_k(X)$ consisting of such $[R]$ for all choices of $x_i$.
   We claim $S_k$ is Zariski dense in $\Hilb_k^{sm}(X)$.
   
   Consider $X^{k} = \underbrace{X\times X \times \dotsb\times X}_{k \text{ times}}$.
   Let $U \subset X^{k}$ be the dense open subset, 
      which is the complement of the union of all diagonals $\Delta_{ij}$,
      i.e.~$\Delta_{ij}\simeq X^{k-2} \times \Delta$, where $\Delta \subset X \times X$ 
      is the diagonal embedded on the $i$-th and $j$-th copy of $X$ in the product.
   Intuitively, $U$ is the set of all distinct ordered $k$-tuples of points,
      but one needs to remember that the product has different points, than the topological product 
      (see Section~\ref{sec_product_and_base_change}).
   Consider the natural family $\ccU \to U$, where $\ccU \subset U \times X$ is the union of respective diagonals
      (so that, over a $\kk$-rational $p =(\fromto{x_1}{x_k}) \in U$ we have
      $\setfromto{x_1}{x_k} \subset \set{p}\times X$).

   The family $\ccU \to U$ is finite, flat and all fibres are smooth,
      and thus induces a map $U \to \Hilb_k^{\circ}(X)$.
      This map is surjective on closed points, as every $\kkbar$-point of $\Hilb_k^{\circ}(X)$ is a
      product of $k$ points of $X$ over $\kkbar$ and every such product
      may be embedded into $U$.
   Therefore the map $U \to \Hilb_k^{sm}(X)$ is dominant.
   
   Denote by $S$ be the set of $\kk$-rational points of $X$. By our assumption $S\subset X$ is dense, 
      hence by Lemma~\ref{lem_product_of_dense_is_dense} the subset $S^k = S\times \dotsb \times S \subset X^k$ is dense.
   Thus $S^k \cap U \subset U$ is dense and its image in $\Hilb_r^{sm}(X)$ is dense, 
      and the claim of the proposition follows from 
      Proposition~\ref{prop_properties_of_relative_linear_span}\ref{item_comparing_relative_spans_dense_subset_point}
      together with the fact that $\cactus{r}{X, [R]} = \sspan{x_1, \ldots,x_k}$ 
      if $R = \bigsqcup\{x_i\}$ is a union of $\kk$-rational points.
\end{proof}

Similarly, the definition of cactus variety agrees with~\cite{nisiabu_jabu_cactus}. Note that we assume $\kk = \kkbar$.
\begin{prop}\label{prop_cactus_variety_is_closure_of_spans_of_schemes}
   Suppose $\kk$ is algebraically closed and $X \subset \PP_{\kk}^N$ is an embedded projective variety. 
   Then the $r$-th cactus variety is the Zariski closure of the union of the
   spans of all finite subschemes of $X$ of degree at most $r$:
   \begin{equation}\label{eq:cactus}
      \cactus{r}{X} = \overline{  \bigcup \set{\sspan{R}  \mid R  \subset X
      \text{ is a finite subscheme of degree at most $r$}}}.
   \end{equation}
\end{prop}
\begin{proof}
    First, if $\Hilb_r(X)\neq \emptyset$, then each subscheme of degree at
    most $r$ is embedded into a subscheme of degree $r$. So we may assume
    that the union on the right hand side of~\eqref{eq:cactus} is over $R$ of
    degree $r$.

    The field is algebraically closed, so $R \subset X$ correspond bijectively
    to closed points of $\Hilb_r(X)$. The set of closed points is Zariski-dense in
    each component $\mathcal{H}$ of $\Hilb_{r}(X)$, so
    by
    Proposition~\ref{prop_properties_of_relative_linear_span}\ref{item_comparing_relative_spans_dense_subset_point}
    we have $\cactus{r}{X, \reduced{\mathcal{H}}} = \overline{\bigcup
        \set{\sspan{R}  \mid [R]\in \mathcal{H}}}$. Summing these equalities
        over all (finitely many) components $\mathcal{H}$ we get the required
        equality.
\end{proof}

%
%

\subsection{Properties of secant and cactus varieties}

As usually, let $X \subset \PP V \simeq \PP^N_{\kk}$ be an embedded projective scheme.

We first conclude from analogous property of Hilbert scheme and relative linear span that the secant and cactus varieties behave nicely with respect to the change of base field.
If $\kk \subset \KK$ is any field extension, then denote by $X_{\KK} = X \times \Spec \KK \subset \PP (V \otimes \KK)$, 
  and by $\sigma_{r, \KK}(X_{\KK})$ the $r$-th secant variety over $\KK$.
Analogously, $\cactus{r, \KK}{X_{\KK}}$ is the $r$-th cactus variety over $\KK$, i.e.~in the discussion in Section~\ref{sect_defining_secant_and_cactus} 
  we replace the role of $\kk$ with $\KK$.

\begin{prop}\label{prop_base_change_for_secant_and_cactus}
   In the notation as above, 
     \[
        \sigma_{r, \KK}(X_{\KK}) = \reduced{(\sigma_{r}(X) \times \Spec \KK)} \quad \text{ and } \quad 
        \cactus{r, \KK}{X_{\KK}} = \reduced{(\cactus{r}{X} \times \Spec \KK)}.
     \]
\end{prop}
\begin{proof} 
   This follows from the base change properties for Hilbert scheme (Proposition~\ref{prop_base_change_for_Hilb_r}), 
     for the smoothable components of the Hilbert scheme (Equation~\eqref{eq:basechangeforsmoothablecomponent}), 
     and for the relative linear span (Proposition~\ref{prop_base_change_for_relative_linear_span}).
\end{proof}

As a consequence, one may axiomatically define the secant variety of $X$, as the minimal subvariety of $\PP V$, 
  which contains all pure secants of $X$, even after base change.

\begin{defn}\label{def_contains_all_secants}
   Suppose $X\subset \PP V$ is an embedded projective scheme, and $r>0$ is an integer.
   Let $\sigma\subset \PP V$ be a subscheme.
   For a field extension $\kk\subset \KK$, denote by $X_{\KK} = X\times \Spec \KK$ and by $\sigma_{\KK} = \sigma \times \Spec \KK$, 
     both subschemes of $\PP (V \otimes \KK)$.
   We say that $\sigma$ \emph{contains all $r$-secants of $X$ even after a field extension},
     if for all field extensions $\kk\subset \KK$, and for all choices of $\KK$-rational points $\fromto{x_1}{x_r} \in X_{\KK}$,
     the $\KK$-linear span $ \langle\fromto{x_1}{x_r}\rangle$ is contained in $\sigma_{\KK}$.
\end{defn}

\begin{prop}\label{prop_axiomatic_def_of_secant}
   Suppose $X\subset \PP V$ is an embedded projective scheme, and $r>0$ is an integer.
   Let $\sigma\subset \PP V$ be a subscheme which contains all $r$-secants of $X$ even after a field extension.
   Then $\sigma_r(X) \subset \sigma$ and $\sigma_r(X)$ contains all $r$-secants of $X$ even after a field extension.
\end{prop}

\begin{proof}
   Suppose $\KK$ is algebraically closed. 
   Then by Propositions~\ref{prop_secant_variety_is_the_closure_of_secants} and \ref{prop_base_change_for_secant_and_cactus}:
   \begin{multline*}
      \reduced{(\sigma_{r}(X) \times \Spec \KK)} \ \stackrel{\text{\ref{prop_base_change_for_secant_and_cactus}}}{=} \ \sigma_{r, \KK}(X_{\KK}) \\
      \ \stackrel{\text{\ref{prop_secant_variety_is_the_closure_of_secants}}}{=}\ \overline{  \bigcup \set{\sspan{\fromto{x_1}{x_r}}_{\KK}  \mid x_i \text{ is a $\KK$-rational point of $X_{\KK}$}}}
      \ \stackrel{\text{\ref{def_contains_all_secants}}}{\subset}  \ \sigma_{\KK}.
   \end{multline*}
   Projecting the resulting inclusion $\reduced{(\sigma_{r}(X) \times \Spec \KK)} \subset \sigma \times \Spec \KK$  on the first factor, we obtain:
   \[
      \sigma_r(X) = \reduced{\sigma_{r}(X)} \subset \sigma.
   \]
   
   To see the second claim, let $\kk\subset \KK$ be any field extension.
   Suppose $\fromto{x_1}{x_r}$ are $\KK$-rational points of $X_{\KK}$, and set $R: =\setfromto{x_1}{x_r}$ as a finite (over $\KK$) subscheme of $X_{\KK}$.
   Let $k$ be the degree of $R$ (which might be less than $r$ if there are repetitions among $x_i$).
   Then $[R]\in \Hilb^{sm}_k(X_{\KK}/\KK)$, and suppose $\ccH \subset \Hilb^{sm}_k(X_{\KK}/\KK)$ is an irreducible component containing $[R]$.
   By Proposition~\ref{prop_properties_of_relative_linear_span}:
   \[
      \cactus{r,\KK}{X_{\KK}, [R]} \subset \cactus{r,\KK}{X_{\KK}, \ccH}
      \subset \sigma_{r, \KK}(X_{\KK}).\qedhere
   \]
\end{proof}

Also an analogous statement holds for cactus varieties. 

\begin{prop}\label{prop_axiomatic_def_of_cactus}
   Suppose $X\subset \PP V$ is an embedded projective scheme, and $r>0$ is an integer.
   Let $\gotK\subset \PP V$ be a subscheme satisfying the property:
   \begin{itemize} 
    \item For every field extension $\kk \subset \KK$ and any finite subscheme
        $R \subset X_{\KK}$ of degree (over $\KK$) at most $r$,  
            the $\KK$-linear span $\sspan{R}_{\KK}$ is contained in $\gotK_{\KK} = \gotK \times \Spec \KK$.
   \end{itemize}
   Then $\cactus{r}{X} \subset \gotK$ and $\cactus{r}{X}$ satisfies the above property.
\end{prop}
The proof is analogous to the proof of Proposition~\ref{prop_axiomatic_def_of_secant}, 
  with the role of Proposition~\ref{prop_secant_variety_is_the_closure_of_secants} replaced by Proposition~\ref{prop_cactus_variety_is_closure_of_spans_of_schemes}, 
  and $\Hilb_k^{sm}(X)$ replaced by $\Hilb_k(X)$.

In Definition~\ref{def_secant_and_cactus_varieties} we insist on taking the union over all lower values of $k \le r$. 
Typically, this is redundant, except when $X$ is small and $\Hilb_r^{sm}(X)$ or $\Hilb_r(X)$ is empty.
For this purpose we make a more precise notion of a not so small scheme.

\begin{defn}
   As above let $X \subset \PP V$ be an embedded projective scheme, and $r$ a positive integer. Assume $\bar{\kk}$ is an algebraic closure of $\kk$.
   \begin{itemize}
    \item  We say that \emph{$X$ has at least $r$ points over $\bar{\kk}$} if either $\dim X >0$ or $\dim X=0$ and $\reduced{(X\times \Spec \bar{\kk})}$ consists of at least $r$ distinct points.
           Equivalently, $\Hilb_r^{sm}(X)$ is nonempty.
    \item  We say that \emph{$X$ is longer than $r$} if either $\dim X >0$ or $\dim X=0$ and $\deg X \ge r$.
           Equivalently, $\Hilb_r(X)$ is nonempty.
   \end{itemize}
\end{defn}

\begin{prop}\label{prop_enough_to_take_k_equal_r}
   Suppose $X$, and $r$ are as above.
   \begin{itemize}
    \item  If $X$ has at least $r$ points over $\bar{\kk}$, then 
           \[
             \sigma_r(X) = \cactus{r}{X, \Hilb_{r}^{sm}(X)}.
           \]
    \item  If $X$ is longer than $r$
           \[
              \cactus{r}{X}:= \cactus{r}{X, \reduced{ (\Hilb_{r}(X))}}.
           \]
   \end{itemize}
   Moreover, if either of the assumptions of the items above is not satisfied, 
     then $\sigma_r(X) = \sigma_{r-1}(X)$, or respectively, $\cactus{r}{X}=\cactus{r-1}{X}$.
\end{prop}

\begin{proof}
   By Proposition~\ref{prop_base_change_for_secant_and_cactus} it is enough to prove the statement for an algebraically closed field $\kk$.
   For every $k$, let $\ccH_k$ be either $\Hilb_{k}^{sm}(X)$ (to prove the first item about the secant variety) or $\reduced{(\Hilb_{k}(X))}$ (to prove the second item about the cactus variety).
   We must show that for any $1 \le  k <r$ we have the inclusion $\cactus{k}{X,\ccH_k} \subset \cactus{k+1}{X,\ccH_{k+1}}$.
   Let $S$ be the set of $\kk$-rational points of $\ccH_k$.
   Since $S$ is a dense subset, by Proposition~\ref{prop_properties_of_relative_linear_span}\ref{item_comparing_relative_spans_dense_subset_point}, 
     it is enough to prove that for every $[R] \in S$ the linear span of $R$ is contained in $\cactus{k+1}{X,\ccH_{k+1}}$.

   In the setting of the first item, if $X$ has at least $r$ points over $\bar{\kk}= \kk$, since $k < r$, there is a $\kk$-rational point $x \in X$ not contained in the support of $R$. 
   Let $R' = R \sqcup \set{x}$, then $[R'] \in \ccH_{k+1} = \Hilb_{k+1}^{sm}(X)$ and $\sspan{R} \subset \sspan{R'} \subset \cactus{k+1}{X, \ccH_{k+1}}$. 
   Thus the first item is proved.
   
   In the setting of the second item, if $X$ is longer than $r$, then there is
   a subscheme $R' \subset X$  of degree $k+1$ containing $R$.
   Having such $R'$, the corresponding point in the Hilbert scheme $[R'] \in \ccH_{k+1} = \Hilb_{k+1}(X)$ satisfies $\sspan{R} \subset \sspan{R'} \subset \cactus{k+1}{X, \ccH_{k+1}}$,
      and the second item is proved.
   
   The ``moreover'' part is straightforward from the definitions, since $\ccH_r$ is empty and so is its relative linear span.
\end{proof}

The conclusion of the proposition is that  in Definition~\ref{def_secant_and_cactus_varieties} it is not necessary to consider the union over $1\le  k \le r-1$:
   either by applying the itemised part of the proposition, or by applying its final part and reducing $r$ to a smaller value. 

The next two propositions says that in sufficiently nice situations we can get rid
of the closure in the definitions and statements above, in the sense that
every point of the secant or cactus variety is contained in the relative
linear span of some point of the Hilbert scheme.
\begin{prop}\label{prop_closure_not_needed_secant}
   Assume $X \subset \PP^N_{\kk}$ is an embedded projective scheme and $r$ and
   $d$ are positive integers, such that $X$ has at least $r$ points over
   $\bar{\kk}$.
   Moreover, suppose that the Hilbert function $H_{\bullet}(d)$ is constant on
   $\Hilb_r^{sm}(X)$
     (for example, this holds whenever $d \ge r-1$, see Corollaries~\ref{cor_regular_implies_independent_embedding} and~\ref{cor_finite_schemes_r_regular}).
   Then:
   \begin{enumerate}
      \item \label{item_closure_no_needed_any_field}
             For every point $p$ in $\sigma_r(\nu_d(X))$, there is a point
              $z \in \Hilb_r^{sm}(X)$ with residue field $\kappa(z)$, such
              that the image of the $\kappa(z)$-linear span of $z$ under the
              projection to $\PP (\DPV{d})$ contains $p$.
      \item \label{item_closure_no_needed_alg_closed_field}
            If in addition $\kk$ is algebraically closed, then
             \[
                \sigma_r(\nu_d(X))(\kk) = \bigcup_{[R] \in \ccH(\kk)} \sspan{R}(\kk),
             \]
             i.e.~every $\kk$-rational point of the secant variety
             is in a span of a finite smoothable subscheme of $X$ of degree $r$.
   \end{enumerate}
\end{prop}
\begin{prop}\label{prop_closure_not_needed_cactus}
   Assume $X \subset \PP^N_{\kk}$ is an embedded projective scheme, and $r$
   and $d$ are positive integers such that $X$ is longer than $r$. Moreover,
   suppose that  the Hilbert function $H_{\bullet}(d)$ is constant on
   $\Hilb_r(X)$ (for example, this holds whenever $d \ge r-1$, see
   Corollaries~\ref{cor_regular_implies_independent_embedding}
   and~\ref{cor_finite_schemes_r_regular}).
   Then:
   \begin{enumerate}
      \item
             For every point $p$ in $\cactus{r}{\nu_d(X)}$, 
               there is a point $z \in \Hilb_r(X)$ with residue field $\kappa(z)$, 
               such that the image of the $\kappa(z)$-linear span of $z$ under the projection to $\PP (\DPV{d})$ contains $p$.
      \item
            If $\kk$ is in addition algebraically closed, then 
             \[
                \cactus{r}{\nu_d(X)}(\kk) = \bigcup_{[R] \in \Hilb_r(X)(\kk)} \sspan{R}(\kk), 
             \]
             i.e.~every $\kk$-rational point of the cactus variety
             is in a span of a finite subscheme of $X$ of degree $r$.
   \end{enumerate}

\end{prop}
\begin{proof}[Proof of Propositions~\ref{prop_closure_not_needed_secant}
    and~\ref{prop_closure_not_needed_cactus}]
    Both proofs are identical in structure, we will conduct them using common
    notation. In the setting of
    Proposition~\ref{prop_closure_not_needed_secant} denote
    $\ccH:=\Hilb_r^{sm}(X)$ and $\gotK:=\sigma_r(\nu_d(X))$. To prove
    Proposition~\ref{prop_closure_not_needed_cactus}, denote
    $\ccH:=\Hilb_r(X)$ and $\gotK:=\cactus{r}{\nu_d(X)}$.

   By Proposition~\ref{prop_enough_to_take_k_equal_r} we have $\gotK = \cactus{r}{\nu_d(X), \ccH}$.
   By Proposition~\ref{prop_properties_of_relative_linear_span}\ref{item_comparing_relative_spans_closed} the map $\PP(\ccE^{\perp}) \to \gotK$ is proper and surjective.
   Pick a point $e \in  \PP(\ccE^{\perp})$ mapping to $p$. 
   If $p$ is a closed point we may also assume $e$ is a closed point.
   Define $z\in \ccH$ to be the image of $e$ under the projection $\PP(\ccE^{\perp}) \to \ccH$.
   Note that since the projection is proper, if $e$ is a closed point then also $z$ is a closed point. 
   The point of the Hilbert scheme $z$ corresponds to a subscheme $Z\subset X
   \times \Spec \kappa(z)$ finite over the residue field $\kappa(z)$.
   The fibre of $\PP(\ccE^{\perp})\to \ccH$ over $z$ is $\sspan{Z}_{\kappa(z)}$ and contains $e$.
   Thus the image of $\sspan{Z}_{\kappa(z)}$ in $\PP (\DPV{d})$ contains $p$ as claimed in \ref{item_closure_no_needed_any_field}.

%
   Now suppose $\kk$ is algebraically closed to show~\ref{item_closure_no_needed_alg_closed_field}.
   Assume $p \in \gotK(\kk)$ (in particular, $p$ is a closed point), then $z$ as above is can be chosen to be a closed point of $\ccH$, in particular, it is a $\kk$-rational point, i.e.~$z\in \ccH(\kk)$.
   Hence $z=[R]$, i.e.~it corresponds to a finite subscheme $R \subset X$ of degree $r$, and $p \in \sspan{R}$.
   Thus $\gotK(\kk) \subset \bigcup_{[R] \in \ccH(\kk)} \sspan{R}(\kk)$. 
   The other inclusion follows from Proposition~\ref{prop_relative_span_of_a_point_contained}.
\end{proof}

Note that the equalities as in Proposition~\ref{prop_closure_not_needed_secant} need not to hold if the Hilbert function is not constant.
Examples of this type are exhibited in \cite{nisiabu_jabu_smoothable_rank_example}.

\subsection{Cactus varieties and Gorenstein locus of the Hilbert scheme}

In this subsection we relate the cactus varieties with Gorenstein schemes, as discussed, in particular, in \cite[Lemma~2.4]{nisiabu_jabu_cactus}.

\begin{lem}\label{lem_Gorenstein_are_enough_for_cactus}
   Suppose $R \subset \PP^N_{\kk}$ is an embedded finite non-Gorenstein scheme
   of degree $r$.
   Then 
   $\sspan{R} = \cactus{r-1}{R}.$
   If in addition $\kk$ is algebraically closed, then the set 
   $\sspan{R}(\kk)$ of $\kk$-rational points in the linear span of $R$ is equal to 
   $\bigcup_{R' \subsetneqq R, \deg R'=r-1} \sspan{R'}(\kk)$, 
   where the sum is over the subschemes $R'$ of degree $r-1$.
\end{lem}

\begin{proof}
   By invariance under the base change, it suffices to prove the statement for
   algebraically closed base field $\kk$, see
   Propositions~\ref{prop_Gorenstein_base_change} and~\ref{prop_base_change_for_relative_linear_span}.
   Then the set of $\kk$-rational ($=$closed) points of the form $[R']$ is dense in $\Hilb_{r-1}(R)$,
      so it is enough to show the second version of the statement: $\sspan{R}(\kk) = \bigcup_{R' \subset R} \sspan{R'}(\kk)$.
   By Example~\ref{exam_Hilbert_function_of_subscheme_drops_by_at_most_r_minus_r_prime} every $\sspan{R'} \subset \sspan{R}$ 
   is either equal to $\sspan{R}$ or a hyperplane inside it.
   If for any choice of $R'$ we have an equality, the claim of the lemma is satisfied.
   Thus we may asume that for all $R'$ the linear span $\sspan{R'}$ is a hyperplane in $\sspan{R} \simeq \PP^q_{\kk}$, 
       i.e.~the Hilbert function $H_{\bullet}(1)$ is constant and equal to $q$ on $\Hilb_{r-1}(R)$.   
   
   The Hilbert scheme $\Hilb_{r-1}(R)$ is projective and $\dim \Hilb_{r-1}(R)>0$ by Lemma~\ref{lem_Gorenstein_is_equivalent_to_finite_Hilb}.
   Pick an irreducible component $\ccH$ of $\Hilb_{r-1}(R)$ such that $\dim \ccH >0$.
   The relative linear span is an irreducible reduced subscheme of $\PP^q_{\kk}$ containing at least one hyperplane $ \PP W\simeq \PP^{q-1}_{\kk} \subset \PP^q_{\kk}$.
   Thus it is either equal to the hyperplane $\PP W$, or to whole $\PP^q_{\kk}$.
   Let us first exclude the case of the hyperplane. 
   This would imply that the linear span of each $R'$ (for $[R']\in \ccH$) is equal to $\PP W$.
   But then $R' \subset \PP W \cap R \subsetneqq R$ and thus $R' =  \PP W \cap R$.
   Therefore $[R']$ is uniquely determined by $\ccH$ and $[R']$ is the only closed point of $\ccH$, 
      which is a contradiction with the assumptions that $\kk$ is algebraically closed and $\dim \ccH >0$. 
   
   Thus the relative linear span of $\ccH$ is $\PP^q_{\kk}$. 
   Since the Hilbert function is $H_{\bullet}(1)$ is constant on $\ccH$, 
     the relative linear span $\cactus{r-1}{R, \ccH}$ 
     is dominated by a family of hyperplanes $\PP^{q-1}_{\kk}$ parametrised by $\ccH$.
   Since everything is projective, it follows that the dominant map is projective and surjective.
   In particular, for every $\kk$-rational point $x\in \PP^{q}_{\kk}$ the fibre 
      (consisting of hyperplanes in the family passing through $x$) 
      is nonempty and closed, hence contains a $\kk$-rational point $([R'], x)$.
   That is, $x\in \sspan {R'}$, proving the claim of the lemma.
\end{proof}

\begin{cor}\label{cor_Gorenstein_loci_enough_for_cactus}
   Suppose $X \subset \PP^N_{\kk}$ is an embedded projective scheme.
   Then 
   \[
             \cactus{r}{X}=\bigcup_{k=1}^r \cactus{k}{X, \reduced{\Hilb^{Gor}_{k}(X)}}.
   \]
   If $\kk$ is algebraically closed, then
   \[
        \cactus{r}{X}=\overline{\bigcup \set{\sspan{R} :  [R] \in \Hilb^{Gor}_{k}(X)(\kk), 1\le k \le r}}.
   \]
   If in addition $X$ has at least $r$ points over $\bar{\kk}$, then in the equalities above it is enough to consider $k=r$, i.e., respectively:
   \[
             \cactus{r}{X}= \cactus{r}{X, \reduced{\Hilb^{Gor}_{r}(X)}} \text{ or } \cactus{r}{X}=\overline{\bigcup \set{\sspan{R} :  [R] \in \Hilb^{Gor}_{r}(X)(\kk)}}.
   \]
\end{cor}

\begin{proof}
   By invariance under a base change, we may assume that $\kk$ is algebraically closed.
   The set of $\kk$-rational points $[R]$ is dense in
   $\Hilb^{Gor}_{k}(X)$, so by
   Corollary~\ref{cor_cactus_equal_to_closure_over_points} the variety
   $\cactus{r}{X}$ is equal to the closure of $\sspan{R}$ over all subschemes $R$.
   Lemma~\ref{lem_Gorenstein_are_enough_for_cactus} implies that spans of all
   non-Gorenstein schemes are covered by spans of their Gorenstein
   subschemes, so we may reduce to Gorenstein subschemes.
   The final claim follows in the same way as the proof of Proposition~\ref{prop_enough_to_take_k_equal_r}, 
     the case of secant varieties, because if $R$ is a finite Gorenstein scheme, then also $R \sqcup \Spec\kk$ is Gorenstein (for instance, by Lemma~\ref{lem_Gorenstein_is_equivalent_to_finite_Hilb}).
\end{proof}

Note that, unlike in Proposition~\ref{prop_enough_to_take_k_equal_r} for cactus varieties, in the last part of Corollary~\ref{cor_Gorenstein_loci_enough_for_cactus} it is not enough to consider $X$ which is longer than $r$.
\begin{example}
  Suppose $A=\kk[\alpha, \beta]/ (\alpha^2, \beta^2)$ and $X = \Spec A$ embedded in some way into $\PP^N_{\kk}$. Assume $r=3$.
  Then $X$ is a scheme of degree $4$ (in particular, it is longer than $3$), its second and third cactus varieties coincide $\cactus{2}{X} = \cactus{3}{X} \simeq \PP^2_{\kk}$ (independent of the embedding), 
     but $\Hilb^{Gor}_{3}(X)$ is empty.
  In particular, $\cactus{3}{X} \ne \overline{\bigcup \set{\sspan{R} :  [R] \in \Hilb^{Gor}_{3}(X)(\kk)}}$.
  
  It is also straightforward to produce analogous examples, where $\cactus{r-1}{X}$ and $\cactus{r}{X}$ are not equal, or where $\Hilb^{Gor}_{3}(X) \ne \emptyset$.
\end{example}

\begin{cor}\label{cor_all_Gor_smoothable_implies_secant_equal_cactus}
   Suppose $X \subset \PP^N_{\kk}$ is an embedded projective scheme,
      such that $X$ has at least $r$ points over $\bar{\kk}$.
   If $\reduced{(\Hilb^{Gor}_{r}(X))} = \Hilb^{Gor, sm}_{r}(X)$,
     then $\cactus{r}{X}= \sigma_r(X)$.
\end{cor}
\begin{proof}
    By Corollary~\ref{ref:persistence_of_smoothability:cor} we have
    $\reduced{(\Hilb^{Gor}_{k}(X))} = \Hilb^{Gor, sm}_{k}(X)$ for all $1 \le k \le r$.
   We always have  $\sigma_r(X) \subset \cactus{r}{X}$. 
   On the other hand:
   \begin{multline*}
      \cactus{r}{X}=\bigcup_{k=1}^{r} \cactus{k}{X, \reduced{\Hilb^{Gor}_{k}(X)}} 
                   =\bigcup_{k=1}^{r} \cactus{k}{X, \Hilb^{Gor, sm}_{k}(X)}\\
                   \subset \bigcup_{k=1}^{r} \cactus{k}{X, \Hilb^{sm}_{k}(X)} = \sigma_r(X).
   \end{multline*}
\end{proof}

\begin{cor}\label{cor_closure_not_needed_Gorenstein}
   Assume $X \subset \PP^N_{\kk}$ is an embedded projective scheme, $r$ and $d$ are positive integers,
      and $X$ has at least $r$ points over $\bar{\kk}$. 
   Moreover, suppose the Hilbert function $H_{\bullet}(d)$ is constant on $\reduced{\Hilb_r(X)}$
     (for example, this holds whenever $d \ge r-1$, see Corollaries~\ref{cor_regular_implies_independent_embedding} and~\ref{cor_finite_schemes_r_regular}).
   Then:
   \begin{enumerate}
      \item \label{item_closure_no_needed_Gorenstein_any_field}
             For every point $p$ in $\cactus{r}{\nu_d(X)}$, 
               there is a point $z \in \reduced{\Hilb^{Gor}_r(X)}$ with the
               residue field $\kappa(z)$, 
               such that the image of the $\kappa(z)$-linear span of $z$ under the projection to $\PP (\DPV{d})$ contains $p$.
      \item \label{item_closure_no_needed_Gorenstein_alg_closed_field}
            If $\kk$ is in addition algebraically closed, then 
             \[
                \cactus{r}{\nu_d(X)}(\kk) = \bigcup_{[R] \in \reduced{\Hilb^{Gor}_r(X)}(\kk)} \sspan{R}(\kk), 
             \]
             i.e.~every $\kk$-rational point of the cactus variety is in the span of a finite Gorenstein subscheme of $X$ of degree $r$ 
               (in addition, in the situation of the secant variety, the scheme is smoothable in $X$).
   \end{enumerate}
\end{cor}
\begin{proof}
    First we prove the claim of
    Point~\ref{item_closure_no_needed_Gorenstein_alg_closed_field}.
    From Proposition~\ref{prop_closure_not_needed_cactus} and
    Lemma~\ref{lem_Gorenstein_are_enough_for_cactus} it follows that every
    $\kk$-point is in the span of a Gorenstein subscheme of degree at most
    $r$. Since $X$ has at least $r$ points over $\kk$, every Gorenstein
    subscheme may be extended to degree $r$ subscheme.
    To prove Point~\ref{item_closure_no_needed_Gorenstein_any_field}, we pick
    a point $p$ with residue field $\KK$ and make a base change to
    $\overline{\KK}$ to reduce to setting of
    Point~\ref{item_closure_no_needed_Gorenstein_alg_closed_field}.
\end{proof}


\section{Secant and cactus varieties of Veronese reembeddings}\label{sect_secants_of_Veronese}

In this section we generalise results obtained over
    complex numbers $\mathbb{C}$ in \cite{nisiabu_jabu_cactus} and
    \cite{jabu_ginensky_landsberg_Eisenbuds_conjecture}, i.e.~we prove Theorems~\ref{thm_BGL} and~\ref{thm_BB}.
As before, throughout this section we assume that the base field $\kk$ is an arbitrary field of characteristic $p\ge 0$.
In the proofs we typically reduce to the case $\kk$ is algebraically closed using the base change theorems from previous sections.

\subsection{High degree Veronese reembeddings preserve secant varieties to smooth subvarieties}

    In this subsection, in Theorem~\ref{ref:bglfirstthm:thm}, we generalise \cite[Thm~1.1]{jabu_ginensky_landsberg_Eisenbuds_conjecture} to smooth projective varieties over an arbitrary field $\kk$.
    It is stated in a simplier and slightly weaker form in the introduction as Theorem~\ref{thm_BGL}.
    \begin{thm}\label{ref:bglfirstthm:thm}
            Let $X\subseteq \PP_{\kk}^n=\PP V$ be a smooth subvariety and $r\in \mathbb{N}$.
            Suppose that the ideal sheaf $\ccI_X$ of $X$ is Castelnuovo-Mumford $\delta$-regular for an integer $\delta \ge 0$. 
            Then for all $d\geq r - 1 + \delta$, one has the equality of reduced subschemes (subvarieties) of $\PP(\DPV{d})$:
            \[
                \sigma_r(\nu_d(X)) = \reduced{\left(\sigma_r(\nu_d(\PP_{\kk}^n))\cap \sspan{\nu_d(X)}\right)}.
            \]
    \end{thm}
    So in this theorem, $X$ and $r$ are fixed, and then after sufficiently high degree of Veronese reembeddings the secant varieties of $X$ and $\PP V$ 
       are related to each other via a simple intersectionformula. 
    The crucial step is a generalisation of \cite[Prop.~2.5]{jabu_ginensky_landsberg_Eisenbuds_conjecture}
        by~Corollary~\ref{ref:smoothingeverytwhere:cor} and
        \cite[Lemma~2.8]{jabu_ginensky_landsberg_Eisenbuds_conjecture}
        by~Proposition~\ref{ref:smooth_subvar_and_pushing:prop}.
        The only other possibly field-dependent results in this
        section are the proof ``Main Lemma'' \cite[Lem.~1.2]{jabu_ginensky_landsberg_Eisenbuds_conjecture}, which is
        slightly generalised by Lemma~\ref{ref:bglmainlemma:lem}.

    \begin{proof}[Proof of Theorem~\ref{ref:bglfirstthm:thm}]
        By invariance under a base change
        (Proposition~\ref{prop_base_change_for_secant_and_cactus}) it is enough to prove the statement for algebraically closed $\kk$.
        The inclusion  $\sigma_r(\nu_d(X)) \subset  \reduced{\left(\sigma_r(\nu_d(\PP_{\kk}^n))\cap \sspan{\nu_d(X)}\right)}$ is straightforward:
           if $Y \subset Z$, then $\sigma_r(Y) \subset \sigma_r(Z)$, and $\sigma_r(Y) \subset \sspan{Y}$, and $\sigma_r(Y)$ is always reduced.
        To prove the reverse inclusion it is enough to prove it on
        $\kk$-rational points, since they are dense in both varieties.
        
        Let $z \in \reduced{\left(\sigma_r(\nu_d(\PP_{\kk}^n))\cap \sspan{\nu_d(X)}\right)}(\kk)$.
        By Proposition~\ref{prop_closure_not_needed_secant} there exists a finite scheme $R\subset\PP V$ of degree $r$ smoothable in $\PP V$
          such that $z \in \sspan{\nu_d(R)}$.
        Thus $z \in \sspan{\nu_d(X)}\cap \sspan{\nu_d(R)}$ and by Lemma~\ref{ref:bglmainlemma:lem} we have
          $z \in \sspan{\nu_d(X \cap R)}$.
        Although $X\cap R$ needs not to be smoothable itself (Corollary~\ref{cor_nonsmoothable_intersection_of_smoothable_and_smooth}),
          by Proposition~\ref{ref:smooth_subvar_and_pushing:prop} there exists a finite scheme $Q \subset X$ of degree $r$ 
          smoothable in $X$ such that $X\cap R \subset Q$.
          By Proposition~\ref{prop_relative_span_of_a_point_contained}
          the claim of the theorem follows:
        \[
           z \in \in \sspan{\nu_d(X \cap R)} \subset  \sspan{\nu_d(Q)} \subset
           \cactus{r}{X, \Hilb^{sm}_r(X)} \subset \sigma_r(\nu_d(X)).\qedhere
        \]
    \end{proof}

    The necessity of assuming that $X$ is smooth in Theorem~\ref{ref:bglfirstthm:thm} 
      is discussed in \cite[\S3]{jabu_ginensky_landsberg_Eisenbuds_conjecture}.
    Without assuming smoothness we have the analogous equality of cactus varieties.
    \begin{prop}\label{prop_BGL_for_cactus}
            Let $X\subseteq \PP_{\kk}^n=\PP V$ be any subvariety and $r\in \mathbb{N}$.
            Suppose that the ideal sheaf $\ccI_X$ of $X$ is Castelnuovo-Mumford $\delta$-regular for an integer $\delta \ge 0$. 
            Then for all $d\geq r - 1 + \delta$, one has the equality of reduced subschemes of $\PP(\DPV{d})$:
            \[
                \cactus{r}{\nu_d(X)} = \reduced{\left(\cactus{r}{\nu_d(\PP_{\kk}^n)}\cap \sspan{\nu_d(X)}\right)}.
            \]
    \end{prop}
    The proof is the same as for Theorem~\ref{ref:bglfirstthm:thm}, but without the smoothability discussion.
    In particular, there is no need to use  Proposition~\ref{ref:smooth_subvar_and_pushing:prop}, which is the only part where we use smoothness.

\subsection{Condition on smoothability of Gorenstein schemes}\label{sec_on_condition_star}
  
\begin{table}[htb]  
\begin{center}
Case of smooth $X$ and  $\cchar \kk = 0$.

\begin{tabular}{|l||c|c|c|c||cc}
\hline
  $\dim X$                  &  $1$, $2$, or $3$ & $4$ & $5$ & $6$ or more \\
\hline
\hline
$r = \fromto{\phantom{14}1}{13}$    &  \colorzielony{$(\star)$} &  \colorzielony{$(\star)$} & \colorzielony{$(\star)$} &  \colorzielony{$(\star)$}\\
\hline
$r = \phantom{1}14$                  &  \colorzielony{$(\star)$} &  \colorzielony{$(\star)$} & \colorzielony{$(\star)$} &  \colordynia{$\cancel{(\star)}$}\\
\hline  
$r = \fromto{\phantom{1}15}{41}$     &  \colorzielony{$(\star)$} &  ??                       &       ??                 &  \colordynia{$\cancel{(\star)}$}\\
\hline
$r = \fromto{\phantom{1}42}{139}$    &  \colorzielony{$(\star)$} &  ??                       & \colordynia{$\cancel{(\star)}$} &  \colordynia{$\cancel{(\star)}$}\\
\hline
$r = 140,\dotsc$   &  \colorzielony{$(\star)$} &  \colordynia{$\cancel{(\star)}$} &     \colordynia{$\cancel{(\star)}$} &  \colordynia{$\cancel{(\star)}$}\\
\hline                           
\end{tabular}
\end{center}
\caption{The condition~\ref{item_condition_on_smoothability_of_all_Gorenstein} for smooth varieties over fields of characteristic $0$.}\label{table_star_for_smooth_char_0}
\end{table}

\begin{table}[htb]  
\begin{center}
Case of smooth $X$ and  $\cchar \kk = p \ge 5$.

\begin{tabular}{|l||c|c|c|c||cc}
\hline
  $\dim X$                  &  $1$, $2$, or $3$ & $4$ & $5$ & $6$ or more \\
\hline
\hline
$r = \fromto{\phantom{14}1}{13}$    &  \colorzielony{$(\star)$} &  \colorzielony{$(\star)$} & \colorzielony{$(\star)$} &  \colorzielony{$(\star)$}\\
\hline
$r = \phantom{1}14$                  &  \colorzielony{$(\star)$} &  \colorzielony{$(\star)$} & \colorzielony{$(\star)$} &  ??\\
\hline  
$r = \fromto{\phantom{1}15}{41}$     &  \colorzielony{$(\star)$} &  ??                       &       ??                 & ??\\
\hline
$r = \fromto{\phantom{1}42}{139}$    &  \colorzielony{$(\star)$} &  ??                       & \colordynia{$\cancel{(\star)}$} &  \colordynia{$\cancel{(\star)}$}\\
\hline
$r = 140,\dotsc$   &  \colorzielony{$(\star)$} &  \colordynia{$\cancel{(\star)}$} &     \colordynia{$\cancel{(\star)}$} &  \colordynia{$\cancel{(\star)}$}\\
\hline                           
\end{tabular}
\end{center}
\caption{The condition~\ref{item_condition_on_smoothability_of_all_Gorenstein} for smooth varieties over fields of positive characteristic not equal to $2$ or $3$.}\label{table_star_for_smooth_char_p}
\end{table}

\begin{table}[htb]  
\begin{center}
Case of smooth $X$ and  $\cchar \kk = 2$ or $3$.

\begin{tabular}{|l||c|c|c|c||cc}
\hline
  $\dim X$                  &  $1$, $2$, or $3$ & $4$ & $5$ & $6$ or more \\
\hline
\hline
$r = \phantom{14}1,2$                        &  \colorzielony{$(\star)$} &  \colorzielony{$(\star)$} & \colorzielony{$(\star)$} &  \colorzielony{$(\star)$}\\
\hline
$r = \fromto{\phantom{14}3}{13}$    &  \colorzielony{$(\star)$} &  ?? & ?? &  ??\\
\hline
$r = \phantom{1}14$                  &  \colorzielony{$(\star)$} &  ?? & ?? &  \colordynia{$\cancel{(\star)}$}\\
\hline  
$r = \fromto{\phantom{1}15}{41}$     &  \colorzielony{$(\star)$} &  ??                       &       ??                 & \colordynia{$\cancel{(\star)}$}\\
\hline
$r = \fromto{\phantom{1}42}{139}$    &  \colorzielony{$(\star)$} &  ??                       & \colordynia{$\cancel{(\star)}$} &  \colordynia{$\cancel{(\star)}$}\\
\hline
$r = 140,\dotsc$   &  \colorzielony{$(\star)$} &  \colordynia{$\cancel{(\star)}$} &     \colordynia{$\cancel{(\star)}$} &  \colordynia{$\cancel{(\star)}$}\\
\hline                           
\end{tabular}
\end{center}
\caption{The condition~\ref{item_condition_on_smoothability_of_all_Gorenstein} for smooth varieties over fields of positive characteristic not equal to $2$ or $3$.}\label{table_star_for_smooth_char_2_3}
\end{table}

\begin{table}[htb]  
\begin{center}
Case of singular curve $X$ and arbitrary $\cchar \kk$.

\begin{tabular}{|l||c|c|}
\hline
  Singularities of $X$                 &  plannar singularities & nonplannar singularities \\
\hline
\hline
$r = 1$                        &  \colorzielony{$(\star)$} &  \colorzielony{$(\star)$}\\
\hline
$r = 2,\dotsc$    &  \colorzielony{$(\star)$} &   \colordynia{$\cancel{(\star)}$}\\
\hline
\end{tabular}
\end{center}
\caption{The condition~\ref{item_condition_on_smoothability_of_all_Gorenstein} for singular curves over any field.}\label{table_star_for_singular_curve}
\end{table}

Suppose $X\subset \PP V$ is an embedded projective variety and $r$ is an integer.
Recall the condition \ref{item_condition_on_smoothability_of_all_Gorenstein} from the introduction:
   \begin{itemize}
    \item[\ref{item_condition_on_smoothability_of_all_Gorenstein}]
           every zero-dimensional Gorenstein $\kkbar$-subscheme of
         $X_{\kkbar}:=X\times_{\Spec \kk}\Spec \kkbar$ 
         of degree at most $r$ is smoothable in $X_{\kkbar}$.
   \end{itemize}
\begin{prop}
   The condition \ref{item_condition_on_smoothability_of_all_Gorenstein} is equivalent to $\reduced{(\Hilb^{Gor}_r(X))} = \Hilb^{Gor, sm}_r(X)$.
\end{prop}
\begin{proof}
   By Proposition~\ref{prop_Gorenstein_base_change}
   and~\eqref{eq:basechangeforsmoothablecomponent} the equation
   $\reduced{(\Hilb^{Gor}_r(X))} = \Hilb^{Gor, sm}_r(X)$ may be checked after
   base change to $X_{\kkbar}$.
   Condition~\ref{item_condition_on_smoothability_of_all_Gorenstein} also
   depends only on $X_{\kkbar}$, we may assume $\kk = \kkbar$.
   Since $X$ is reduced, if $X$ does not have at least $r$ points over $\kk$,
      both \ref{item_condition_on_smoothability_of_all_Gorenstein} and $\reduced{(\Hilb^{Gor}_r(X))} = \Hilb^{Gor, sm}_r(X) = \emptyset$ trivially hold.
   Otherwise, assume $X$ has at least $r$ points over $\kk$.

   We always have $\Hilb^{Gor, sm}_r(X) \subset \reduced{(\Hilb^{Gor}_r(X))}$.
   By Corollary~\ref{ref:persistence_of_smoothability:cor} the other inclusion $\reduced{(\Hilb^{Gor}_r(X))} \subset \Hilb^{Gor, sm}_r(X)$ 
     holds if and only if $\reduced{(\Hilb^{Gor}_k(X))}  \subset \Hilb^{Gor, sm}_k(X)$ for all $k \le r$.
   The set of $\kk$-rational points is dense, hence the latter inclusions hold 
     if and only if $\reduced{(\Hilb^{Gor}_k(X))}(\kk)  \subset \Hilb^{Gor, sm}_k(X)$,
     that is, if and only if for every finite Gorenstein subscheme $R \subset X$ of degree $k\le r$,
     $[R] \in \Hilb^{Gor, sm}_k(X)$.
   By Proposition~\ref{ref:embeddedvsHilbert:prop} this is equivalent to  smoothability of $R$ in $X$, i.e. to~\ref{item_condition_on_smoothability_of_all_Gorenstein}.
\end{proof}

By Corollary~\ref{ref:smoothingeverytwhere:cor} and Proposition~\ref{prop_smoothability_depends_only_on_sing_type} the condition~\ref{item_condition_on_smoothability_of_all_Gorenstein} 
  depends only on the types of singularities of $X$, not on its global geometry. In particular, if $X$ is smooth, then \ref{item_condition_on_smoothability_of_all_Gorenstein} 
  depends only on three integers $\dim X$, $r$, and $\cchar \kk$. The known
  results when the conditions hold are summarised in
  Tables~\ref{table_star_for_smooth_char_0}, \ref{table_star_for_smooth_char_p}, \ref{table_star_for_smooth_char_2_3}.
  First,~\ref{item_condition_on_smoothability_of_all_Gorenstein} holds for all
  characteristics and $\dim X \leq 2$ by results of Fogarty~\cite{fogarty} and
  for $\dim X = 3$ by Kleppe's~\cite{kleppe_smoothability_in_codim_three}. It
  also holds when $r\leq 13$ in all characteristics other than $2$ or $3$ by~\cite[Theorem~A]{casnati_jelisiejew_notari_Hilbert_schemes_via_ray_families}
  and for degree $r = 14$ and $\dim X\leq 5$ by the same theorem. Finally, we
  proved above that algebras of degree at most two are smoothable regardless
  of characteristic. This justifies all cases
  where~\ref{item_condition_on_smoothability_of_all_Gorenstein} holds.
  Emsalem and Iarrobino provided an example of length $14$ algebra embedded in
  $\AA_{\kk}^6$ which is
  non-smoothable~\cite{emsalem_iarrobino_small_tangent_space}. The example
  requires a computer verification, so we know that it exists in
  characteristic two, three or zero (though most probably it exists in all
  characteristics). This example implies 
  that~for $r\geq 14$ and $\dim X\geq 6$ the
  condition~\ref{item_condition_on_smoothability_of_all_Gorenstein} does not
  hold. For $\dim X = 4$, $r\geq 140$ and $\dim X = 5$, $r\geq 42$ the
  examples of nonsmoothable Gorenstein schemes given
  in~\cite[Proposition~6.2]{nisiabu_jabu_cactus} generalise to any field.
  Moreover, Table~\ref{table_star_for_singular_curve} briefly summarises the results of \cite{altman_iarrobino_kleiman_irreducibility_of_compactified_jacobian} on singular curves.

\subsection{Secant, cactus and catalecticants}\label{sec_secant_cactus_catalecticants}

Suppose $X\subset \PP V$ is an embedded projective variety.  
We have seen in Corollary~\ref{cor_Gorenstein_loci_enough_for_cactus} that if $\reduced{(\Hilb^{Gor}_r(X))} = \Hilb^{Gor, sm}_r(X)$, 
  then the secant and cactus varieties coincide. 
Here we claim an inverse of this statement, see also \cite[Thm~1.4(ii)]{nisiabu_jabu_cactus}.
\begin{prop}\label{prop_cactus_equals_secant_then_star_holds}
   Suppose $d \ge 2r-1$. 
   If $\cactus{r}{\nu_d(X)} = \sigma_r(\nu_d(X))$, then  $\reduced{(\Hilb^{Gor}_r(X))} = \Hilb^{Gor, sm}_r(X)$, i.e.~\ref{item_condition_on_smoothability_of_all_Gorenstein} holds.
\end{prop}
\begin{proof}[Sketch of proof]
   We may make a base change to an algebraically closed base field $\kk$
   using  Proposition~\ref{prop_base_change_for_secant_and_cactus},
   Equation~\eqref{eq:basechangeforsmoothablecomponent},
   and Lemma~\ref{ref:GorensteinLocus:lem}.
   Then it is enough to show that $\kk$-rational points coincide for $\Hilb^{Gor}_r(X)$  and $\Hilb^{Gor, sm}_r(X)$.
   This is shown in the same way as the proof of  \cite[Thm~1.4(ii)]{nisiabu_jabu_cactus}.
\end{proof}

Recall the catalecticant equations arising from vector bundles.
Say, $X\subset \PP W$ is an embedded projective variety with $W= H^0(L)^*$ for a line bundle $L$ on $X$.
The main case of interest is when $W = \DPV{d}$, and $L = \ccO_X(d)$ for sufficiently large $d$.
Let $\ccE$ be a vector bundle on $X$.
Then the natural evaluation map of vector bundles:
\[
   \ccE \otimes \ccE^*\otimes L \to L
\]
determines a linear map  $H^0(\ccE) \otimes H^0(\ccE^*\otimes L) \to H^0(L)$, or equivalently,
  a bilinear map $\phi_{\ccE}\colon W \times H^0(\ccE) \to H^0(\ccE^* \otimes L)^*$.
  Choosing bases of $H^0(\ccE)$, $H^0(\ccE \otimes L)^*$ and $W$, 
  we can see this bilinear map as a $\dim H^0(\ccE) \times \dim H^0(\ccE^* \otimes L)$ matrix,
  whose entries linearly depend on the coordinates on $W$.
In particular, for any integer $r$, the $(r+1)\times(r+1)$ minors of this matrix give homogeneous polynomials of degree $r+1$
  in coordinates of $W$, that is a subset of $S^{r+1} W$.
We denote the scheme in $\PP W$ defined by this collection of equations 
  (or by the homogeneous ideal generated by these equations)
  by $(\rk \phi_{\ccE} \le r)$.

We have following theorem of Ga{\l}{\k a}zka \cite{galazka_vb_cactus}:
  \begin{thm}\label{thm_Galazka}
     Suppose $X$, $\ccE$ and $\phi_{\ccE}$ are as above. 
     If $e$ is the vector bundle rank of $\ccE$, then 
     $\cactus{r}{X}\subset (\rk \phi_{\ccE} \le e r)$.
  \end{thm}
  \begin{proof}[Sketch of proof]
     As usually, first reduce to the case of algebraically closed $\kk$.
     The cactus variety is invariant under the base field change by
     Proposition~\ref{prop_base_change_for_secant_and_cactus}.
     Also the sections of a vector bundle $\ccE$ are invariant by \cite[Tag~02KH(2)]{stacks_project}
       and the resulting $\dim H^0(\ccE) \times \dim H^0(\ccE^* \otimes L)$ matrix is the same 
       matrix after base field extension. 
    Therefore the equations defining $(\rk \phi_{\ccE} \le e r)$ are invariant and we may assume $\kk$ is algebraically closed.
    
    Further, it is enough to prove the statement on $\kk$-rational points:
    \[
       \cactus{r}{X}(\kk)\subset (\rk \phi_{\ccE} \le e r)(\kk),
    \]
      which is done in \cite[Thm~5]{galazka_vb_cactus}.
  \end{proof}  
  
Finally, we prove a partial inverse: in some sufficiently nice situations the cactus variety (or secant variety) is equal to $ (\rk \phi_{\ccE} \le e r)$.
Many explicit situations of this type are described in the literature, see for instance \cite{landsberg_ottaviani_VB_method_equ_for_secants}, \cite{oeding_ottaviani_eigenvectors_of_tensors}, \cite{michalik_mgr}.
Here we concentrate on the result of \cite[Thm~1.8]{nisiabu_jabu_cactus}, which deals with Veronese reembeddings of projective schemes. 

In the first place we observe the following reduction from the case of projective space, whose proof is based on \cite[Section~7]{nisiabu_jabu_cactus}.
\begin{cor}\label{cor_cactus_reduce_to_PV}
   Let $X \subset \PP V$ be an embedded projective variety, and $r >0$ be an integer.
   Assume $\ccI_X$ is Castelnuovo-Mumford $\delta$-regular and $d \ge \delta+r-1$.
   We consider the Veronese reembeddings $\nu_d(X)\subset \nu_d(\PP V) \subset \PP (\DPV{d})$ and their cactus varieties.
   Suppose $\ccE$ is a vector bundle on $\PP V$ such that the map $\phi_{\ccE} \colon \DPV{d} \times H^0(\ccE) \to H^0(\ccE^*(d))^*$ 
     determines the locus $ (\rk \phi_{\ccE} \le er)$ such that  $\cactus{r}{\nu_d(\PP V)}=\reduced{(\rk \phi_{\ccE} \le e r)} $.
   Then the restricted vector bundle $\ccE|_X$ determines the locus $(\rk \phi_{\ccE|_X} \le e r)$ such that  $\cactus{r}{\nu_d(X)} =\reduced{(\rk \phi_{\ccE|_X} \le er)} $.
   
   Moreover, if in addition $\reduced{(\Hilb_r^{Gor}(X))} = \Hilb_r^{Gor, sm}(X)$, 
      then
      \[
       \sigma_r(\nu_d(X)) =\reduced{(\rk \phi_{\ccE|_X} \le er)}.
      \]
\end{cor}
\begin{proof}
   Let $W \subset \DPV{d}$ be the linear subspace such that $\sspan{\nu_d(X)} = \PP W$. 
   Note that $W= H^0(\ccO_X(d))^*$ by Lemma~\ref{lem_regular_then_complete_linear_system}.
   We claim we have the following series of inclusions:
   \begin{multline}
     \cactus{r}{\nu_d(X)} \stackrel{\text{Thm~\ref{thm_Galazka}}}{\subset} \reduced{\left(\rk \phi_{\ccE|_X} \le er\right)} \subset \reduced{\left((\rk \phi_{\ccE} \le er) \cap \PP W\right)} \\
     = \reduced{\left(\cactus{r}{\nu_d(\PP V)} \cap \PP W \right)}  \stackrel{\text{Prop.~\ref{prop_BGL_for_cactus}}}{=} \cactus{r}{\nu_d(X)}.\label{equ_inclusions_cactus_reduce_to_PV}
   \end{multline}
   The first one follows from Theorem~\ref{thm_Galazka} since the cactus variety is reduced. 
   By the construction, $(\rk \phi_{\ccE|_X} \le er) \subset \PP( H^0(\ccO_X(d))^*) = \PP W$.
   To justify the inclusion $(\rk \phi_{\ccE|_X} \le er) \subset (\rk \phi_{\ccE} \le er)$,
     consider the composition:
   \[
      W \times H^0(\ccE) \to  W \times H^0(\ccE|_X) \stackrel{\phi_{\ccE|_X}}{\longrightarrow} H^0(\ccE|_X^*(d))^* \to H^0(\ccE^*(d))^*
   \]
   The composition is equal to $\phi_{\ccE}|_{W \times H^0(\ccE)}$, thus the rank of $\phi_{\ccE}$  for points in $\PP W$ is at most the rank of $\phi_{\ccE|_X}$ as claimed.
   Hence the second inclusion in \eqref{equ_inclusions_cactus_reduce_to_PV} holds.
   The third relation (equality)  is a consequence of our assumptions,
      while the final equality is just Proposition~\ref{prop_BGL_for_cactus}.
   Thus, all inclusions there are equalities and in particular $\cactus{r}{\nu_d(X)} = \reduced{\left(\rk \phi_{\ccE|_X} \le er\right)}$.
   
   The ``moreover'' statement follows from Corollary~\ref{cor_all_Gor_smoothable_implies_secant_equal_cactus}.
\end{proof}

Next, we restrict our attention to the case of $X=\nu_d(\PP V)$ and  $e=1$, i.e. $\ccE \simeq \ccO_{\PP V}(i)$ is a line bundle on $\PP V$.
\begin{thm}[{\cite[Thm~1.5]{nisiabu_jabu_cactus}}]\label{thm_BB_for_PV}
    Suppose $d \ge 2r$ and $r \le i \le d-r$. Then $\cactus{r}{\nu_d(\PP V)} = \reduced{(\rk \phi_{\ccO_{\PP V}(i)} \le r)}$.
\end{thm}
\begin{proof}[Sketch of proof]
    As always, we reduce to the case of algebraically closed field $\kk$ and to prove that $\kk$-rational points coincide.
    This is done in the same way as in the sketch of proof of Theorem~\ref{thm_Galazka}.
    Next we follow the proof of \cite[Thm~1.5]{nisiabu_jabu_cactus}, see \cite[Section~5]{nisiabu_jabu_cactus}, using the apolarity in the form of Section~\ref{sec_apolarity}.
\end{proof}

\begin{proof}[Proof of Theorem~\ref{thm_BB}]
    Suppose as in \ref{item_if_star_holds_then_sigma_eq_cactus} that the condition~\ref{item_condition_on_smoothability_of_all_Gorenstein} holds (i.e.~$\reduced{(\Hilb_r^{Gor}(X))} = \Hilb_r^{Gor, sm}(X)$),
      $d \ge \max(2r, r-1 + \delta)$, where $\delta$ is such that $X$ is Castelnuovo-Mumford $\delta$-regular, and $i$ is an integer such that $r\le i \le d-r$.
    Then $\cactus{r}{\nu_d(\PP V)} = \reduced{(\rk \phi_{\ccO_{\PP V}(i)} \le r)}$ by Theorem~\ref{thm_BB_for_PV}.
    By Corollary~\ref{cor_cactus_reduce_to_PV}, this implies $\sigma_r(\nu_d(X)) =\reduced{(\rk \phi_{\ccO_X(i)} \le r)}$, as claimed in  \ref{item_if_star_holds_then_sigma_eq_cactus}.
    
    Now suppose \ref{item_condition_on_smoothability_of_all_Gorenstein} does not hold, i.e.~$\reduced{(\Hilb_r^{Gor}(X))} \supsetneqq \Hilb_r^{Gor, sm}(X)$ and $d \ge 2r-1$.
    Then by Proposition~\ref{prop_cactus_equals_secant_then_star_holds} the inclusion $ \sigma_r(\nu_d(X)) \subsetneqq \cactus{r}{\nu_d(X)}$ is strict.
    But by Theorem~\ref{thm_Galazka}:
    \[
       \cactus{r}{\nu_d(X)}\subset \bigcap_{\ccE} \reduced{(\rk(\phi_{\ccE}) \le e_{\ccE} r)}\subset \PP (H^0(\ccO_X(d))^*).
    \]
    Hence the claim of \ref{item_if_sigma_eq_cactus_then_star_holds} holds. 
\end{proof}

%
    \addcontentsline{toc}{section}{References}
    \bibliography{wygladzalnosc}
    \bibliographystyle{alpha}

    \end{document}